\documentclass[fleqn,ijoc]{informs4_modified}

\OneAndAHalfSpacedXI 

\usepackage[small, margin=1cm]{caption}

\usepackage{color}
\definecolor{strcolor}{rgb}{0.6, 0.2, 0.6}
\definecolor{commentcolor}{rgb}{0.3125, 0.5, 0.3125}
\definecolor{keycol}{rgb}{0, 0, 1}

\usepackage{bbm}

\usepackage{listings}
\usepackage{hyperref}

\usepackage{epstopdf}
\usepackage{algorithm,algpseudocode}
\hypersetup{colorlinks=true, 
	breaklinks=true,
	urlcolor= blue, 
	linkcolor= blue, 
	citecolor=blue}

\usepackage{epsfig}
\usepackage{mathtools}
\usepackage{relsize}
\usepackage{enumitem}
\usepackage{subcaption}
\usepackage{xspace}
\usepackage{dsfont}
\usepackage{multicol}
\usepackage[export]{adjustbox}

\usepackage[defaultcolor=red]{changes}

\usepackage{tikz}
\usepackage{pgfplots}
\pgfplotsset{compat = newest}

\newcommand{\tr}[1]{\ensuremath{{#1}^\text{T}}}

\newcommand{\uset}[2]{\ensuremath{\underset{#1}{#2}}}

\DeclarePairedDelimiter\abs{\lvert}{\rvert}%

\newcommand{\Q}{\mathcal{Q}}
\newcommand{\R}{\mathbb{R}}

\newcommand{\bfx}{\mathbf{x}}
\newcommand{\bfQ}{\mathbf{Q}}
\newcommand{\bfr}{\mathbf{r}}
\newcommand{\bfb}{\mathbf{b}}
\newcommand{\bfW}{\mathbf{W}}
\newcommand{\bfLambda}{\boldsymbol\Lambda}
\newcommand{\bfP}{\mathbf{P}}
\newcommand{\bfY}{\mathbf{Y}}
\newcommand{\bfz}{\mathbf{z}}
\newcommand{\bfB}{\mathbf{B}}
\newcommand{\bfc}{\mathbf{c}}
\newcommand{\bfd}{\mathbf{d}}
\newcommand{\bfM}{\mathbf{M}}
\newcommand{\bfzero}{\mathbf{0}}
\newcommand{\bfpi}{\boldsymbol\pi}
\newcommand{\bff}{\mathbf{f}}
\newcommand{\bfg}{\mathbf{g}}
\newcommand{\bftheta}{\boldsymbol\theta}
\newcommand{\bfa}{\mathbf{a}}
\newcommand{\bfZ}{\mathbf{Z}}

\newcommand{\pmc}{piecewise McCormick}
\newcommand{\pp}{partitioning points}

\newcommand{\inp}[2]{\langle{#1},{#2}\rangle}

\usepackage{fancyhdr}
\usepackage{titlesec}

\newcommand{\setSupplementStyle}{
  \renewcommand{\thesubsection}{S.\arabic{subsection}}
  \renewcommand{\thesubsubsection}{\thesubsection.\arabic{subsubsection}}
}

\usepackage{listofitems} 
\usetikzlibrary{arrows.meta} 
\usepackage[outline]{contour} 
\contourlength{1.4pt}

\tikzset{>=latex} 
\usepackage{xcolor}
\colorlet{myred}{red!80!black}
\colorlet{myblue}{blue!80!black}
\colorlet{mygreen}{green!60!black}
\colorlet{myorange}{orange!70!red!60!black}
\colorlet{mydarkred}{red!30!black}
\colorlet{mydarkblue}{blue!40!black}
\colorlet{mydarkgreen}{green!30!black}
\tikzstyle{node}=[thick,circle,draw=myblue,minimum size=22,inner sep=0.5,outer sep=0.6]
\tikzstyle{node in}=[node,green!20!black,draw=mygreen!30!black,fill=mygreen!25]
\tikzstyle{node hidden}=[node,blue!20!black,draw=myblue!30!black,fill=myblue!20]
\tikzstyle{node convol}=[node,orange!20!black,draw=myorange!30!black,fill=myorange!20]
\tikzstyle{node out}=[node,red!20!black,draw=myred!30!black,fill=myred!20]
\tikzstyle{connect}=[thick,mydarkblue] 
\tikzstyle{connect arrow}=[-{Latex[length=4,width=3.5]},thick,mydarkblue,shorten <=0.5,shorten >=1]
\tikzset{ 
  node 1/.style={node in},
  node 2/.style={node hidden},
  node 3/.style={node out},
}

\def\blot{\quad \mbox{$\vcenter{ \vbox{ \hrule height.4pt
				\hbox{\vrule width.4pt height.9ex \kern.9ex \vrule width.4pt}
				\hrule height.4pt}}$}}

\usepackage{natbib}
\bibpunct[, ]{(}{)}{,}{a}{}{,}%
\def\bibfont{\fontsize{8}{9.5}\selectfont}%
\TheoremsNumberedThrough     
\ECRepeatTheorems

\EquationsNumberedThrough    

\gdef\AQ#1{}
\gdef\CQ#1{}
\begin{document}
	

	\RUNAUTHOR{Kannan, Nagarajan, and Deka} %

	\RUNTITLE{ML for Solving Nonconvex QCQPs to Global Optimality}

\TITLE{{Strong Partitioning and a Machine Learning Approximation for Accelerating the Global Optimization of Nonconvex QCQPs}}


	\ARTICLEAUTHORS{
\AUTHOR{Rohit Kannan,\textsuperscript{a,*} Harsha Nagarajan,\textsuperscript{b,*} Deepjyoti Deka\textsuperscript{c}}

\AFF{$^{a}$Grado Department of Industrial and Systems Engineering, Virginia Tech, Blacksburg, VA, USA \\ 
$^{b}$Applied Mathematics and Plasma Physics, Theoretical Division, Los Alamos National Laboratory, Los Alamos, NM, USA \\ 
$^{c}$MIT Energy Initiative, Massachusetts Institute of Technology, Cambridge, MA, USA \\ 
$^{*}$Corresponding authors \\ {\bf Contact:} 
\href{mailto:rohitkannan@vt.edu}{rohitkannan@vt.edu}, \href{mailto:harsha@lanl.gov}{harsha@lanl.gov}, \href{mailto:deepj87@mit.edu}{deepj87@mit.edu}}

}
	

\ABSTRACT{We learn optimal instance-specific heuristics for the global minimization of nonconvex quadratically-constrained quadratic programs (QCQPs).
Specifically, we consider partitioning-based convex mixed-integer programming relaxations for nonconvex QCQPs and propose the novel problem of strong partitioning to optimally partition variable domains \textit{without} sacrificing global optimality.
Since solving this max-min strong partitioning problem exactly can be very challenging, we design a local optimization method that leverages generalized gradients of the value function of its inner-minimization problem.
However, even solving the strong partitioning problem to local optimality can be time-consuming.
To address this, we propose a simple and practical machine learning (ML) approximation for {homogeneous} families of QCQPs.
{Motivated by practical applications, we conduct a detailed computational study using the open-source global solver Alpine to evaluate the effectiveness of our ML approximation in accelerating the repeated solution of homogeneous QCQPs with fixed structure. 
Our study considers randomly generated QCQP families, including instances of the pooling problem, that are benchmarked using state-of-the-art global optimization software.}
Numerical experiments demonstrate that our ML approximation of strong partitioning reduces Alpine's solution time by a factor of $2$ to $4.5$ \textit{on average}, with maximum reduction factors ranging from $10$ to $200$ across {these} QCQP families.
}

\FUNDING{We gratefully acknowledge funding from LANL's Center for Nonlinear Studies and the U.S.\ Department of Energy's LDRD program for the projects ``20230091ER: Learning to Accelerate Global Solutions for Non-convex Optimization'' and ``20210078DR: The Optimization of Machine Learning: Imposing Requirements on Artificial Intelligence.''}




\KEYWORDS{Nonconvex Quadratically-Constrained Quadratic Programming; Global Optimization; Piecewise McCormick Relaxations; Strong Partitioning; Sensitivity Analysis; Machine Learning; Pooling Problem. \\ \hspace*{-0.15in} \textbf{Initial Submission:} December 31, 2022. \:\: \textbf{Revised:} February 22, 2023; September 15, 2024. \:\: \textbf{Final Version:} May 20, 2025}

\maketitle

\section{Introduction}

Many real-world applications involve the repeated solution of the same underlying quadratically-constrained quadratic program (QCQP) with slightly varying model parameters.
Examples include the pooling problem with varying input qualities~\citep{misener2013glomiqo} and the cost-efficient operation of the power grid with varying loads and renewable sources~\citep{bienstock2022mathematical}.
These hard optimization problems are typically solved using off-the-shelf global optimization software~\citep{bestuzheva2021scip,misener2014antigone,nagarajan2019adaptive,sahinidis1996baron} that do not fully exploit the shared problem structure---heuristics within these implementations are engineered to work well \textit{on average} over a \textit{diverse} set of instances and may perform suboptimally for instances from a specific application~\citep{liu2019tuning}.
Recent work~\citep{bengio2021machine,lodi2017learning} has shown that tailoring branching decisions can significantly accelerate branch-and-bound (B\&B) algorithms for mixed-integer linear programs (MILPs).
In contrast, only a few papers (see Section~\ref{subsec:learning_minlps}) attempt to use machine learning (ML) to accelerate the guaranteed global solution of nonconvex \textit{nonlinear} programs.

We use ML to accelerate partitioning algorithms~\citep{bergamini2008improved,saif2008global,wicaksono2008piecewise} for the global minimization of nonconvex QCQPs.
Partitioning algorithms determine lower bounds on the optimal value of a nonconvex QCQP using piecewise convex relaxations.
They begin by selecting a subset of the continuous variables involved in nonconvex terms for partitioning.
At each iteration, they refine the partitions of these variables' domains and update the piecewise convex relaxations in their lower bounding formulation.
A convex mixed-integer program (MIP) is then solved to determine a lower bound~\citep{lu2018tight,sundar2021piecewise}.
Partitioning algorithms typically use heuristics to specify the locations of {\pp} and continue to refine their variable partitions until the lower bounds converge to the optimal objective value.
Since the complexity of their MIP relaxations can grow significantly with each iteration, the choice of {\pp} in the initial iterations can have a huge impact on the performance of these algorithms~\citep{nagarajan2019adaptive}.
Despite their importance, the optimal selection of partitioning points is not well understood, and current approaches resort to heuristics such as bisecting the active partition~\citep{castro2016normalized,saif2008global,wicaksono2008piecewise} or adding {\pp} near the lower bounding solution~\citep{bergamini2008improved,nagarajan2019adaptive} to refine variable partitions.

\textbf{Proposed approach.} 
We learn to optimally partition variables' domains in a given QCQP \textit{without} sacrificing global optimality.
Similar to strong branching in B\&B algorithms for MIPs~\citep{achterberg2007constraint,misener2013glomiqo}, we introduce \textit{strong partitioning}, a new concept in the global optimization literature, to select {\pp} for the construction of piecewise convex relaxations of nonconvex problems.
The key idea behind strong partitioning is to determine a specified number of {\pp} per variable such that the resulting piecewise convex relaxation-based lower bound is maximized.
We formulate strong partitioning as a \mbox{max-min} problem, where the outer-maximization selects the {\pp} and the inner-minimization solves the piecewise convex relaxation-based lower bounding problem for a given partition. 
{Since solving this max-min problem exactly can be very challenging, we design a local optimization method that uses generalized gradients of the value function of the inner-minimization problem within a bundle solver for nonsmooth nonconvex optimization.
However, even finding a local solution to this max-min problem can be computationally expensive as each iteration of the bundle method requires the solution of a MIP.
Therefore, we propose a simple and practical off-the-shelf ML model to imitate the strong partitioning strategy for {homogeneous} QCQP instances.}
{Motivated by practical applications, we evaluate the performance of strong partitioning and its ML approximation for the repeated solution of homogeneous QCQPs with fixed structure using the open-source global solver Alpine~\citep{nagarajan2016tightening,nagarajan2019adaptive}.
For our computational study, we generate random QCQP families, including instances of the pooling problem, and benchmark them using the state-of-the-art global optimization solvers BARON~\citep{sahinidis1996baron} and Gurobi.}
{Numerical experiments demonstrate that our ML approximation of strong partitioning reduces Alpine's solution time by a factor of $2$ to $4.5$ on average, with a maximum reduction factor of $10$ to $200$ across {these} QCQP families.
Additionally, the results indicate that an efficient ML model capable of perfectly imitating the strong partitioning strategy would reduce Alpine's solution time by a factor of $3.5$ to $16.5$ on average, and by a maximum factor of $15$ to $700$ over the same set of instances.
These observations underscore the potential of strong partitioning as an expert strategy for learning to accelerate the global optimization of nonconvex QCQPs.}

This paper is organized as follows.
Section~\ref{sec:relatedwork} reviews ML approaches to accelerate the \textit{guaranteed global solution} of MILPs and mixed-integer nonlinear programs.
Section~\ref{sec:partitioning_bounds} outlines partitioning-based bounding methods for nonconvex QCQPs.
Section~\ref{sec:strongpart} introduces strong partitioning and designs an algorithm for its local solution with theoretical guarantees.
{Section~\ref{sec:ml_approx} outlines our ML approximation of strong partitioning for {homogeneous} QCQP families,} and Section~\ref{sec:computexp} presents detailed computational results demonstrating the effectiveness of using strong partitioning and our ML approximation to select Alpine's partitioning points {for the repeated solution of homogeneous QCQPs with fixed structure}.
We conclude with directions for future research in Section~\ref{sec:conclusion}.

\textbf{Notation.}
Let $[n] := \{1,\dots,n\}$, $\mathbb{R}_+$ denote the set of nonnegative reals, $\mathcal{S}^n$ denote the set of symmetric $n \times n$ matrices, and $\inp{\mathbf{A}}{\mathbf{B}}$ denote the Frobenius inner product between $\mathbf{A}, \mathbf{B} \in \mathcal{S}^n$.
We write $v_i$ to denote the $i$th component of \mbox{$\mathbf{v} = (v_1,v_2,\dots,v_n) \in \mathbb{R}^n$}, $\mathbf{M}_i$ and $M_{ij}$ to denote the $i$th row and $(i,j)$th entry of a matrix~$\mathbf{M}$, and $\abs{\mathcal{C}}$, $\text{ri}(\mathcal{C})$, and $\text{conv}(\mathcal{C})$ to denote the cardinality, relative interior, and convex hull of a set~$\mathcal{C}$.

\section{Related work}
\label{sec:relatedwork}

{Optimization solvers tune algorithmic parameters through extensive testing on benchmark libraries. Many solvers also incorporate problem-specific analysis to set key parameters.
However, these tunings are typically based on efficiently computable heuristics designed to perform well across a broad range of instances, which may not fully exploit the unique characteristics of each specific instance.
Machine learning offers the potential to efficiently approximate more sophisticated, instance-specific heuristics, and could lead to significant computational gains for particularly challenging instances.}
\citet{bengio2021machine,cappart2021combinatorial} and~\citet{lodi2017learning} survey the burgeoning field of using ML to accelerate MILP and combinatorial optimization algorithms.
In the next sections, we review related approaches on learning to branch for MILPs and learning to accelerate the guaranteed solution of (mixed-integer) nonlinear programs.

\subsection{Learning to branch for MILPs}

Branch-and-bound and its variants form the backbone of modern MILP solvers.
The choice of branching variable at each node of the B\&B tree can have a huge impact on the run time of B\&B algorithms~\citep{achterberg2007constraint}.
{Strong branching is a heuristic for selecting the branching variable that often empirically results in a small number of B\&B nodes explored.}
It selects the variable that maximizes the product of improvements in the lower bounds of the two child nodes, assuming both are feasible.
While strong branching results in a $65\%$ reduction in the number of nodes explored by the B\&B tree on average (relative to the default branching strategy) over standard test instances, it also leads to a $44\%$ increase in the average solution time~\citep{achterberg2007constraint}.
MILP solvers therefore tend to use computationally cheaper heuristic approximations of strong branching, such as reliability, pseudocost, or hybrid branching.
Motivated by the promise of strong branching, most approaches on learning to branch for MILPs aim to develop a computationally efficient and effective ML approximation, e.g., using extremely randomized trees~\citep{alvarez2017machine}, support vector machines~\citep{khalil2016learning}, or graph neural networks~\citep{gasse2019exact,nair2020solving}.
Other ML approaches for branching variable selection use online and reinforcement learning~\citep{he2014learning}, or learn combinations of existing heuristics to make better branching decisions~\citep{di2016dash}.

\subsection{Learning to solve mixed-integer nonlinear problems}
\label{subsec:learning_minlps}

There are relatively fewer approaches in the literature for accelerating the \textit{guaranteed global solution} of nonconvex (mixed-integer) nonlinear programs using ML. 
To the best of our knowledge, none of these approaches use ML to accelerate partitioning-based global optimization algorithms.

\citet{baltean2019scoring} consider the global solution of nonconvex QCQPs using semidefinite programming relaxations.
To mitigate the computational burden of these relaxations, they use ML to construct effective linear outer-approximations.
Specifically, they train a neural network to select cuts from a semidefinite relaxation based on their sparsity and predicted impact on the objective.
Their approach results in computationally efficient relaxations that can be effectively integrated into global solvers.

\citet{ghaddar2022learning} explore branching variable selection in a B\&B search tree embedded within the reformulation-linearization technique for solving polynomial problems~\citep{sherali2013reformulation}.
They use ML to choose the ``best branching strategy'' from a portfolio of branching rules, designing several hand-crafted features to optimize a quantile regression forest-based approximation of their performance metric.
\citet{gonzalez2022polynomial} build on this approach by using ML to select a subset of second-order cone and semidefinite constraints and further strengthen the formulation.

\citet{bonami2018learning} train classifiers to predict whether linearizing products of binary variables or binary and continuous variables is computationally advantageous for solving mixed-integer quadratic programs.
\citet{nannicini2011probing} train a support vector machine classifier to decide if an expensive optimality-based bound tightening routine should replace a cheaper feasibility-based routine for mixed-integer nonlinear programs.
\citet{cengil2022learning} consider the AC optimal power flow problem and train a neural network to identify a small subset of lines and buses for which an optimality-based bound tightening routine is applied.
Finally, \citet{lee2020accelerating} use classification and regression techniques to identify effective cuts for the generalized Benders decomposition master problem.

In the next section, we review partitioning algorithms for the global minimization of QCQPs, before introducing strong partitioning and an ML approximation to accelerate these algorithms.

\section{Partitioning-based bounds for QCQPs}
\label{sec:partitioning_bounds}

Consider the nonconvex QCQP
\begin{alignat}{2}
\label{eqn:orig_qcqp}
&\min_{\bfx \in [0,1]^{n}} \:\: && \bfx^{\textup{T}} \bfQ^0 \bfx + (\bfr^0)^{\textup{T}} \bfx \\
&\quad\: \text{s.t.} && \bfx^{\textup{T}} \bfQ^i \bfx + (\bfr^i)^{\textup{T}} \bfx \leq b_i, \quad \forall i \in [m], \nonumber
\end{alignat}
where $\bfb \in \R^{m}$, $\bfr^i \in \R^n$, $\forall i \in \{0\} \cup [m]$, and $\bfQ^i \in \mathcal{S}^n$, $\forall i \in \{0\} \cup [m]$, are \textit{not} assumed to be positive semidefinite.
QCQPs with equality constraints and different variable bounds can be handled using simple transformations.
Polynomial optimization problems may also be reformulated as QCQPs through the addition of variables and constraints.
\citet{pardalos1991quadratic} show that the special case of problem~\eqref{eqn:orig_qcqp} where $\bfQ^0$ has a single negative eigenvalue and $\bfQ^i = \bfzero$, $\forall i \in [m]$, is NP-hard.

QCQPs arise in several applications~\citep{furini2019qplib,misener2013glomiqo} such as facility location~\citep{koopmans1957assignment}, refinery optimization~\citep{kannan2018algorithms,yang2016integrated}, and electric grid optimization \citep{bienstock2022mathematical}.
By introducing auxiliary variables and constraints, we can reformulate problem~\eqref{eqn:orig_qcqp} into the following equivalent form:
\begin{alignat}{2}
\label{eqn:qcqp}
v^* := \: &\min_{\bfx \in [0,1]^n, \bfW \in \mathcal{S}^n} \:\: && \tr{(\bfr^0)} \bfx + \inp{\bfQ^0}{\bfW} \tag{QCQP} \\
&\quad\quad\:\: \text{s.t.} && \tr{(\bfr^i)} \bfx + \inp{\bfQ^i}{\bfW} \leq b_i, \quad \forall i \in [m], \nonumber\\
& && W_{ij} = x_i x_j, \quad \forall (i, j) \in \mathcal{B}, \nonumber\\
& && W_{kk} = x^2_k, \quad \forall k \in \mathcal{Q}, \nonumber
\end{alignat}
where the index sets
$\mathcal{B} \subset \{(i,j) \in [n]^2 : i < j \}$ and $\mathcal{Q} \subset [n]$ denote the (pairs of) variables participating in distinct bilinear and univariate quadratic terms.
We assume for simplicity that~\eqref{eqn:qcqp} is feasible, and define $\mathcal{F} := \{(\bfx,\bfW) \in [0,1]^n \times \mathcal{S}^n : \tr{(\bfr^i)} \bfx + \inp{\bfQ^i}{\bfW} \leq b_i, \: \forall i \in [m]\}$ for convenience.

\citet{al1983jointly} propose to use termwise McCormick relaxations~\citep{mccormick1976computability} to construct lower bounds on the optimal value of \eqref{eqn:qcqp}. 
Specifically, they employ the following convex lower bounding problem within a spatial B\&B framework for solving~\eqref{eqn:qcqp} to global optimality:
\begin{alignat}{2}
\label{eqn:mccormick}
&\min_{(\bfx,\bfW) \in \mathcal{F}} \:\: && \tr{(\bfr^0)} \bfx + \inp{\bfQ^0}{\bfW} \\
&\quad\:\: \text{s.t.} && \max\{0, x_i + x_j -1\} \leq W_{ij} \leq \min\{x_i, x_j\}, \quad \forall (i, j) \in \mathcal{B}, \nonumber\\
& && x^2_k \leq W_{kk} \leq x_k, \quad \forall k \in \mathcal{Q}. \nonumber
\end{alignat}
{Several papers~\citep[e.g.,][]{bao2011semidefinite,bergamini2008improved,billionnet2012extending,burer2008finite,nohra2021spectral,nohra2021sdp,saif2008global,sherali2013reformulation,shor1987quadratic,wicaksono2008piecewise}} improve upon this termwise McCormick bound, usually at an increase in the computational cost but with the goal of reducing the overall time for the B\&B algorithm to converge.
In this work, we use {\pmc} relaxations~\citep{bergamini2008improved,castro2016normalized,nagarajan2019adaptive,saif2008global,wicaksono2008piecewise} to iteratively strengthen the lower bounding problem~\eqref{eqn:mccormick}.

Piecewise McCormick relaxations begin by partitioning the domains of (a subset of) variables participating in nonconvex terms into subintervals.
{Let 
\[
\mathcal{NC} := \big\{i \in [n] : \exists j \in [n] \: \text{such that} \: (i,j) \in \mathcal{B} \: \text{or} \: (j,i) \in \mathcal{B} \big\} \cup \mathcal{Q}
\]
denote the set of indices of variables participating in nonconvex terms within~\eqref{eqn:qcqp}.
We assume, without loss of generality, that $\mathcal{NC} = \{1, 2, \dots, \abs{\mathcal{NC}}\}$, i.e., only the first $\abs{\mathcal{NC}}$ variables $x_i$ participate in nonconvex terms.
Furthermore, we assume for simplicity that the domain of each variable $x_i$, $i \in \mathcal{NC}$, is partitioned into $d+1$ subintervals, where $d \geq 1$ (other partitioning schemes can be handled similarly).
Let $\bfP$ denote the $\abs{\mathcal{NC}} \times (d+2)$ matrix of partitioning points (including the variable bounds $0$ and $1$), where
\[
\bfP_i := (P_{i1}, P_{i2},\dots,P_{i(d+1)}, P_{i(d+2)}), \quad \text{with} \quad 0 =: P_{i1} \leq P_{i2} \leq \dots \leq P_{i(d+1)} \leq P_{i(d+2)} := 1, \quad \forall i \in \mathcal{NC},
\]
denotes the vector of $d+2$ {\pp} for variable $x_i$, $i \in \mathcal{NC}$.
Unlike approaches that select a subset of variables $x_i$, $i \in \mathcal{NC}$, involved in nonconvex terms for partitioning~\citep{nagarajan2019adaptive}, we partition the domains of \textit{all} such variables.
While partitioning only a subset may suffice to guarantee convergence, the resulting partitioning algorithms for~\eqref{eqn:qcqp} may suffer from the cluster problem in reduced-space global optimization~\citep{kannan2017cluster,kannan2018convergence}, potentially requiring significantly more iterations.}

The {\pmc} relaxation-based lower bounding problem for~\eqref{eqn:qcqp} can be expressed as:
\begin{alignat}{2}
\label{eqn:piecewise_mccormick}
&\min_{(\bfx,\bfW) \in \mathcal{F}} \:\: && \tr{(\bfr^0)} \bfx + \inp{\bfQ^0}{\bfW} \\
&\quad\:\: \text{s.t.} && (x_i, x_j, W_{ij}) \in \mathcal{PMR}^{B}_{ij}(\bfP_i,\bfP_j), \quad \forall (i, j) \in \mathcal{B}, \nonumber\\
& && (x_k, W_{kk}) \in \mathcal{PMR}^{Q}_k(\bfP_k), \quad \forall k \in \mathcal{Q}, \nonumber
\end{alignat}
where $\mathcal{PMR}^B_{ij}(\bfP_i, \bfP_j)$ and $\mathcal{PMR}^{Q}_{k}(\bfP_k)$
denote the feasible regions of the {\pmc} relaxations of $W_{ij} = x_i x_j$ and $W_{kk} = x^2_k$, respectively, obtained using the partitioning matrix~$\bfP$.
While there are several ways of formulating these {\pmc} relaxations, we use the ``convex combination'' or ``lambda'' formulation below~\citep[see][for enhancements in the multilinear setting]{kim2022piecewise}.

The piecewise \mbox{McCormick} relaxation for the constraint $W_{ij} = x_i x_j$ is written as~\citep{sundar2021piecewise}:
\begin{align*}
\mathcal{PMR}^B_{ij}(\bfP_i, \bfP_j) &:= \Big\{ (x_i, x_j, W_{ij}) \: : \: \exists \bfLambda^{ij} \in \R^{(d+2) \times (d+2)}_+, \: \bfY_{i} \in \{0,1\}^{(d+1)}, \: \bfY_{j} \in \{0,1\}^{(d+1)} \\
&\hspace*{1.7in} \text{s.t.} \:\: (x_i, x_j, W_{ij}, \bfLambda^{ij}, \bfY_{i}, \bfY_{j}) \:\: \text{satisfies} \:\: \eqref{eqn:lambda_form_bilinear_conv_comb}-\eqref{eqn:lambda_form_bilinear_active2} \Big\},
\end{align*}
where
{
\begin{subequations}
\begin{alignat}{3}
&x_i = \sum_{k,l=1}^{d+2} \Lambda^{ij}_{kl} P_{il}, \quad && x_j = \sum_{k,l=1}^{d+2} \Lambda^{ij}_{kl} P_{jk}, \quad && W_{ij} = \sum_{k,l=1}^{d+2} \Lambda^{ij}_{kl} P_{il} P_{jk}, \label{eqn:lambda_form_bilinear_conv_comb}\\
&\sum_{l=1}^{d+1} Y_{il} = 1,   \qquad\qquad\:\:\: &&\sum_{k=1}^{d+1} Y_{jk} = 1,   \quad\qquad\qquad &&\sum_{k,l=1}^{d+2} \Lambda^{ij}_{kl} = 1, \label{eqn:lambda_form_bilinear_bin} \\
&\sum_{k=1}^{d+2} \Lambda^{ij}_{k1} \leq Y_{i1},  \quad &&\sum_{k=1}^{d+2} \Lambda^{ij}_{k(d+2)} \leq Y_{i(d+1)},  \qquad &&\sum_{k=1}^{d+2} \Lambda^{ij}_{k(l+1)} \leq Y_{il} + Y_{i(l+1)}, \:\: \forall l \in [d], \label{eqn:lambda_form_bilinear_active} \\
&\sum_{l=1}^{d+2} \Lambda^{ij}_{1l} \leq Y_{j1},  \qquad\quad\:\:\: &&\sum_{l=1}^{d+2} \Lambda^{ij}_{(d+2) l} \leq Y_{j(d+1)},  \quad &&\sum_{l=1}^{d+2} \Lambda^{ij}_{(k+1)l} \leq Y_{jk} + Y_{j(k+1)}, \:\: \forall k \in [d]. \label{eqn:lambda_form_bilinear_active2}
\end{alignat}
\end{subequations}
}%
Only equations~\eqref{eqn:lambda_form_bilinear_conv_comb} depend on the partitioning matrix $\bfP$, which is a parameter in these constraints. The binary vectors $\bfY_i$ and $\bfY_j$ denote the active partition of $x_i$ and $x_j$---these variables are reused in the {\pmc} relaxations of other nonconvex terms involving $x_i$ or~$x_j$. 

The {\pmc} relaxation for $W_{kk} = x^2_k$ is written as~\citep{lu2018tight,nagarajan2019adaptive}:
\begin{align*}
\mathcal{PMR}^Q_{k}(\bfP_k) &:= \Big\{ (x_k, W_{kk}) \: : \: \exists \bfLambda^{k} \in \R^{(d+2)}_+, \: \bfY_{k} \in \{0,1\}^{(d+1)} \:\: \text{s.t.} \:\: (x_k, W_{kk}, \bfLambda^{k}, \bfY_{k}) \:\: \text{satisfies} \:\: \eqref{eqn:lambda_form_quadratic_conv_comb}-\eqref{eqn:lambda_form_quadratic_active} \Big\},
\end{align*}
where
{
\begin{subequations}
\begin{alignat}{3}
&x_k = \sum_{l=1}^{d+2} \Lambda^{k}_{l} P_{kl}, \qquad  && W_{kk} \leq \sum_{l=1}^{d+2} \Lambda^{k}_{l} (P_{kl})^2, \qquad && \sum_{l=1}^{d+1} Y_{kl} P_{kl} \leq x_k \leq \sum_{l=1}^{d+1} Y_{kl} P_{k(l+1)}, \label{eqn:lambda_form_quadratic_conv_comb}\\
&\sum_{l=1}^{d+1} Y_{kl} = 1,   &&\sum_{l=1}^{d+2} \Lambda^{k}_l = 1, && W_{kk} \geq x^2_k, \label{eqn:lambda_form_quadratic_bin} \\
&\Lambda^{k}_{1} \leq Y_{k1},  && \Lambda^{k}_{d+2} \leq Y_{k(d+1)},  &&\Lambda^{k}_{l+1} \leq Y_{kl} + Y_{k(l+1)}, \:\: \forall l \in [d]. \label{eqn:lambda_form_quadratic_active}
\end{alignat}
\end{subequations}
}%
Only equations~\eqref{eqn:lambda_form_quadratic_conv_comb} depend on the partitioning matrix $\bfP$, which is again a parameter in these constraints. The third set of constraints in~\eqref{eqn:lambda_form_quadratic_conv_comb} are redundant for the description of the set $\mathcal{PMR}^Q_{k}(\bfP_k)$, but strengthen its convex relaxation. 
Note that equations~\eqref{eqn:lambda_form_quadratic_bin} involve convex quadratic functions of $x_k$. 
Figure~\ref{fig:piecewise_mccormick} illustrates the {\pmc} relaxations for a bilinear and a univariate quadratic term.

\begin{figure}
\centering

\caption{Illustration of piecewise McCormick relaxations for a bilinear and a univariate quadratic term. The variable domains are changed from $[0,1]$ to $[-1,1]$ for better illustration.
The left and middle plots illustrate the lower and upper parts, respectively, of the piecewise McCormick relaxation for the bilinear term \mbox{$w_{12} = x_1 x_2$} on the domain $x_1, x_2 \in [-1,1]$ 
with partitions \mbox{$\bfP_1 = \bfP_2 = (-1,0,1)$}.
The right plot illustrates the piecewise McCormick relaxation for the quadratic term $w_{11} = x^2_1$ on the domain $x_1 \in [-1,1]$ (the lower part coincides with the red quadratic curve) with the partition $\bfP_1 = (-1,0,1)$.}
\label{fig:piecewise_mccormick}

\vspace*{0.2in}
\begin{subfigure}[t]{0.32\textwidth}
\includegraphics[width=\linewidth]{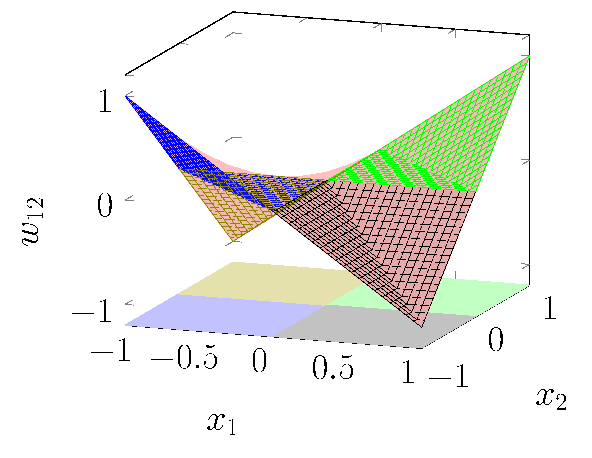}
\end{subfigure}
\hfill
\begin{subfigure}[t]{0.32\textwidth}
\includegraphics[width=\linewidth]{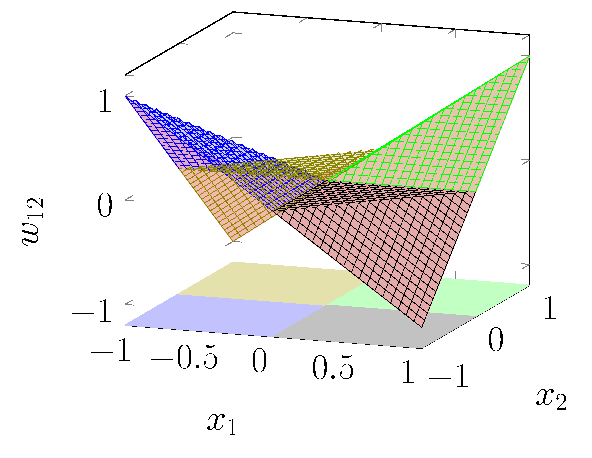}
\end{subfigure}
\hfill
\begin{subfigure}[t]{0.33\textwidth}
\includegraphics[width=\linewidth]{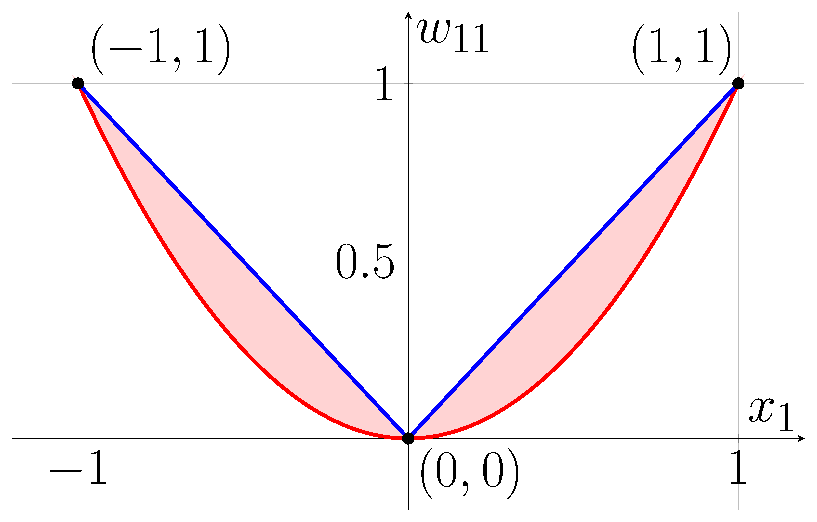}
\end{subfigure}

\end{figure}

Using the above representations of $\mathcal{PMR}^B_{ij}$ and $\mathcal{PMR}^Q_{k}$, we get the following extended \textit{convex} mixed-integer QCQP formulation for the {\pmc} relaxation problem~\eqref{eqn:piecewise_mccormick}:
\begin{alignat}{2}
\label{eqn:piecewise_mccormick_mip}
&\underset{\bfLambda \geq \bfzero, \bfY \in \mathcal{Y}}{\min_{(\bfx,\bfW) \in \mathcal{F}}} && \tr{(\bfr^0)} \bfx + \inp{\bfQ^0}{\bfW} \tag{PMR}\\
&\quad\:\: \text{s.t.} \quad && (x_i, x_j, W_{ij}, \bfLambda^{ij}, \bfY_i, \bfY_j) \:\: \text{satisfies} \:\: \eqref{eqn:lambda_form_bilinear_conv_comb}-\eqref{eqn:lambda_form_bilinear_active2}, \quad \forall (i,j) \in \mathcal{B}, \nonumber\\
& && (x_k, W_{kk}, \bfLambda^k, \bfY_k) \:\: \text{satisfies} \:\: \eqref{eqn:lambda_form_quadratic_conv_comb}-\eqref{eqn:lambda_form_quadratic_active}, \quad \forall k \in \mathcal{Q}, \nonumber
\end{alignat}
where $\bfY \in \{0,1\}^{\abs{\mathcal{NC}} \times (d+1)}$, $\mathcal{Y} := \big\{\bfY \in \{0,1\}^{\abs{\mathcal{NC}} \times (d+1)} : \sum_{l=1}^{d+1} Y_{il} = 1, \: \forall i \in \mathcal{NC} \big\}$ is an intersection of special-ordered sets of type~1, and $\bfLambda$ comprises $\bfLambda^{ij}$, $(i,j) \in \mathcal{B}$, and $\bfLambda^k$, $k \in \mathcal{Q}$.
{
We call the $j$th partition $[P_{ij}, P_{i(j+1)}]$ of variable $x_i$, $i \in \mathcal{NC}$, active if there exists an optimal solution to~\eqref{eqn:piecewise_mccormick_mip} with an \mbox{$\bfx$-component} $\bar{\bfx} \in [0,1]^n$ such that $P_{ij} \leq \bar{x}_i \leq P_{i(j+1)}$. 
Equivalently, the $j$th partition $[P_{ij}, P_{i(j+1)}]$ of $x_i$ is said to be active if there exists an optimal solution to~\eqref{eqn:piecewise_mccormick_mip} with a $\bfY$-component $\bar{\bfY} \in \mathcal{Y}$ such that $\bar{Y}_{ij} = 1$.
}

\begin{algorithm}[t]
\caption{Partitioning-based global optimization algorithm for~\eqref{eqn:qcqp}}
{
\begin{algorithmic}[1]

\State \textbf{Input:} relative optimality tolerance $\varepsilon_r > 0$.

\State \textbf{Initialization:} partitioning points $\bfP^0_i := (0, 1)$, $\forall i \in \mathcal{NC}$, best found solution $\{\hat{\bfx}\} = \emptyset$ with objective $UBD = +\infty$, lower bound $LBD = -\infty$, and iteration number $l = 0$.

\vspace*{0.05in}
\State Solve~\eqref{eqn:qcqp} locally. Update the incumbent $\hat{\bfx}$ and upper bound $UBD$ if relevant.

\vspace*{0.05in}
\State Solve the termwise McCormick relaxation~\eqref{eqn:mccormick}. Update the lower bound $LBD$.

\vspace*{0.05in}
\While{$\dfrac{UBD - LBD}{\abs{UBD} + 10^{-6}} > \varepsilon_r$ \textbf{or} $UBD = +\infty$}

\vspace*{0.05in}
\State Update $l \leftarrow l+1$.

\State\label{step:refine_part} \textbf{Partition refinement:} add new partitioning points to $\bfP^{l-1}$ to obtain the partitioning matrix $\bfP^l$.

\State Solve~\eqref{eqn:piecewise_mccormick_mip} with the partitioning matrix $\bfP^l$. Update the lower bound $LBD$.

\State Solve~\eqref{eqn:qcqp} locally. Update the incumbent $\hat{\bfx}$ and upper bound $UBD$ if needed.

\EndWhile

\vspace*{0.05in}
\State Return the $\varepsilon_r$-optimal $\bfx$-solution $\hat{\bfx}$, upper bound $UBD$, and lower bound $LBD$.

\end{algorithmic}
}
\label{alg:partitioning_algorithm}
\end{algorithm}

Algorithm~\ref{alg:partitioning_algorithm} outlines a partitioning-based global optimization algorithm that solves problem~\eqref{eqn:piecewise_mccormick_mip} to determine a sequence of lower bounds on the optimal value of~\eqref{eqn:qcqp}.
It assumes that local search can find a near-optimal solution to~\eqref{eqn:qcqp} within a finite number of iterations---a standard assumption that often holds in practice.
{The convergence of the lower bound is typically the limiting factor in the convergence of partitioning-based algorithms.
In particular, the choice of heuristic on line~\ref{step:refine_part} of Algorithm~\ref{alg:partitioning_algorithm} for refining the partitioning matrix $\bfP$ can greatly impact both the number of iterations and the time required for convergence.
This motivates the concept of strong partitioning, which selects the partitioning matrix $\bfP$ to maximize the {\pmc} relaxation-based lower bound.
We define partition refinement policies in Section~\ref{sec:partpolicy}, and introduce the strong partitioning policy and an ML approximation in Sections~\ref{subsec:strong_part_policy} and~\ref{sec:ml_approx}.
Section~\ref{sec:alpinepart} details the partition refinement policy implemented within Alpine~\citep{nagarajan2019adaptive}.}
Note that Algorithm~\ref{alg:partitioning_algorithm} can be readily adapted to include enhancements such as bound tightening.

\begin{algorithm}[t]
\caption{{Generic partition refinement policy for iteration $l$}}
\label{alg:generic_partitioning_refinement_strategy}
{
\begin{algorithmic}[1]

\State \textbf{Input:} partitioning matrix $\bfP^{l-1}$ used to construct {\pmc} relaxations at iteration $l-1$, index $\mathcal{A}(i,l-1) \in \mathbb{N}$ of an active partition for variable $x_i$ at iteration $l-1$, $\forall i \in \mathcal{NC}$, and the number of new partitioning points $d_l \in \mathbb{N}$ to be added to $\bfP^{l-1}_i$, $\forall i \in \mathcal{NC}$.

\State \textbf{Initialization:} set $\bfP^l = \bfP^{l-1}$.

\vspace*{0.05in}
\For{$i \in \mathcal{NC}$}

\State Add $d_l$ new partitioning points to the active partition $\big[P^{l-1}_{i(\mathcal{A}(i,l-1))}, P^{l-1}_{i(\mathcal{A}(i,l-1) + 1)}\big]$ for variable $x_i$ 

\Statex \hspace*{0.2in} to obtain a refined vector of partitioning points $\bfP^l_i$ for $x_i$.

\EndFor

\vspace*{0.05in}
\State \textbf{Output:} partitioning matrix $\bfP^l$ used to construct {\pmc} relaxations at iteration $l$.

\end{algorithmic}
}
\end{algorithm}

\subsection{{Partitioning policy}}
\label{sec:partpolicy}

{Algorithm~\ref{alg:generic_partitioning_refinement_strategy} outlines a generic policy for refining variable partitions at any given iteration $l$ of Algorithm~\ref{alg:partitioning_algorithm}.
It adds $d_l$ new partitioning points to the active partition of each variable $x_i$ involved in nonconvex terms.
Various partition refinement strategies, such as bisecting the active partition~\citep{castro2016normalized,saif2008global,wicaksono2008piecewise}, adding {\pp} near the lower bounding solution~\citep{bergamini2008improved,nagarajan2019adaptive}, and the strong partitioning and ML-based policies described in Sections~\ref{subsec:strong_part_policy} and~\ref{sec:ml_approx}, can be seen as specific instantiations of Algorithm~\ref{alg:generic_partitioning_refinement_strategy}.}

{In the next two sections, we introduce the strong partitioning and ML-based policies, which differ from the partitioning policy implemented in Alpine (see Section~\ref{sec:alpinepart} for details) only during the first iteration.}

\section{Strong partitioning for nonconvex QCQPs}
\label{sec:strongpart}

The choice of partitioning points in the initial iterations can greatly impact the strength of lower bounds, the number of iterations needed for convergence, and the overall solution time.
We introduce the \textit{Strong Partitioning} (SP) policy to overcome the limitations of existing heuristics for selecting partitioning points.

{Before presenting the concept of strong partitioning, we further relax~\eqref{eqn:piecewise_mccormick_mip} by outer-approximating the convex quadratic terms in equation~\eqref{eqn:lambda_form_quadratic_bin} to obtain the following MILP relaxation.}
For purely bilinear programs ($\Q = \emptyset$), problem~\eqref{eqn:piecewise_mccormick_mip} is already an MILP, and this outer-approximation step is unnecessary.
{
\begin{alignat}{2}
\label{eqn:piecewise_mccormick_mip_oa}
\underline{v}(\bfP) := &\underset{\bfLambda \geq \bfzero, \bfY \in \mathcal{Y}}{\min_{(\bfx,\bfW) \in \mathcal{F}}} && \tr{(\bfr^0)} \bfx + \inp{\bfQ^0}{\bfW} \tag{PMR-OA} \\
&\quad\:\: \text{s.t.} \quad && (x_i, x_j, W_{ij}, \bfLambda^{ij}, \bfY_i, \bfY_j) \:\: \text{satisfies} \:\: \eqref{eqn:lambda_form_bilinear_conv_comb}-\eqref{eqn:lambda_form_bilinear_active2}, \quad \forall (i,j) \in \mathcal{B}, \nonumber\\
& && (x_k, W_{kk}, \bfLambda^k, \bfY_k) \:\: \text{satisfies} \:\: \eqref{eqn:lambda_form_quadratic_conv_comb} \text{ and } \eqref{eqn:lambda_form_quadratic_active}, \quad \forall k \in \mathcal{Q}, \nonumber \\
& && \sum_{l=1}^{d+2} \Lambda^{k}_l = 1, \quad W_{kk} \geq (2 \alpha^k_j) x_k - (\alpha^k_j)^2, \quad \forall j \in \mathcal{J}_k, \:\: k \in \mathcal{Q}. \label{eqn:piecewise_mccormick_mip_oa_outer}
\end{alignat}%
}%
We explicitly indicate the dependence of the {\pmc} lower bound $\underline{v}$ on the partitioning matrix~$\bfP$.
Constraints~\eqref{eqn:piecewise_mccormick_mip_oa_outer} outer-approximate the inequalities $W_{kk} \geq x^2_k$ in equation~\eqref{eqn:lambda_form_quadratic_bin} at the points $\{\alpha^k_j\}_{j \in \mathcal{J}_k} \subset [0,1]$, which are assumed to include $P_{k1},\dots,P_{k(d+2)}$.
{This outer-approximation step is needed to compute a generalized gradient of the value function $\underline{v}$ with respect to the partitioning matrix $\bfP$, a key component of our SP algorithm (see Section~\ref{subsec:genlgrad} for details).}
\textit{We solve problem}~\eqref{eqn:piecewise_mccormick_mip_oa} \textit{only during strong partitioning} and revert to solving problem~\eqref{eqn:piecewise_mccormick_mip} when computing lower bounds within Algorithm~\ref{alg:partitioning_algorithm}.

{We recast~\eqref{eqn:piecewise_mccormick_mip_oa} into the following abstract form for mathematical convenience, using a suitably defined vector $\bfc$, matrices $\bar{\bfM}$ and $\bar{\bfB}$, matrix-valued functions $\bfM$ and $\bfd$ with co-domains $\mathbb{R}^{n_r \times n_c}$ and $\mathbb{R}^{n_r}$,
and a vector of variables $\bfz$ (including variables $\bfx$, $\bfW$, $\bfLambda$, and auxiliary/slack variables):}

\noindent 
\hspace*{0.3in}
\begin{minipage}{0.4\linewidth}
\vspace*{-0.4in}
\begin{align*}
\underline{v}(\bfP) &:= \min_{\bfY \in \mathcal{Y}} \: v(\bfP,\bfY), \qquad \text{where}
\end{align*}
\end{minipage}%
\hspace*{-0.2in}
\begin{minipage}{0.55\linewidth}
\begin{alignat}{3}
\label{eqn:piecewise_mccormick_mip_abstract_inner}
v(\bfP,\bfY) &:= \min_{\bfz \geq \bfzero} \:\: && \tr{\bfc} \bfz  \tag{PMR-OA-LP} \\
&\quad\quad \text{s.t.} \:\: && \bfM(\bfP) \bfz = \bfd(\bfP), \nonumber \\
& && \bar{\bfM} \bfz = \bar{\bfB}\: \text{vec}(\bfY), \nonumber
\end{alignat}
\end{minipage}

\vspace*{0.15in}
\noindent {where $\text{vec}(\bfY)$ denotes the vectorization of matrix $\bfY$. We omit in~\eqref{eqn:piecewise_mccormick_mip_abstract_inner} the third set of constraints in equation~\eqref{eqn:lambda_form_quadratic_conv_comb} because they are redundant for the {\pmc} relaxations $\mathcal{PMR}^Q_k$.
The constraints $\bar{\bfM} \bfz = \bar{\bfB}\: \text{vec}(\bfY)$ represent equations~\eqref{eqn:lambda_form_bilinear_active},~\eqref{eqn:lambda_form_bilinear_active2}, and~\eqref{eqn:lambda_form_quadratic_active} by adding slack variables.
They ensure that for any $\bfY \in \mathcal{Y}$, at most four of the $\bfLambda^{ij}$ variables in the extended formulation of each set $\mathcal{PMR}^B_{ij}$ and at most two of the $\bfLambda^k$ variables in the extended formulation of each set $\mathcal{PMR}^Q_k$ may be nonzero.
}

The concept of strong partitioning is analogous to strong branching in B\&B algorithms for MILPs. 
While strong branching for MILPs involves selecting the branching variable---a discrete choice---to maximize the product of improvements in the lower bounds of the two child nodes, strong partitioning selects {\pp} for \textit{each} partitioned variable---continuous choices within the variable domains---to maximize the {\pmc} relaxation lower bound.
It can be formulated as the following max-min problem:
\begin{align}
\label{eqn:strong_part}
\bfP^* &\in \argmax_{\bfP \in \mathcal{P}_d} \: \underline{v}(\bfP), \tag{SP}
\end{align}
where $\underline{v}(\bfP)$ is the optimal value function of problem~\eqref{eqn:piecewise_mccormick_mip_oa}, and the set $\mathcal{P}_d$ is defined as:
\[
\mathcal{P}_d := \bigl\{ \bfP \in [0,1]^{\abs{\mathcal{NC}} \times (d+2)} \: : \: 0 = P_{i1} \leq P_{i2} \leq \dots \leq P_{i(d+1)} \leq P_{i(d+2)} = 1, \: \forall i \in \mathcal{NC} \bigr\}.
\]
The strong partitioning problem~\eqref{eqn:strong_part} is challenging to solve even to local optimality because the inner problem~\eqref{eqn:piecewise_mccormick_mip_oa} includes binary decisions and its feasible region depends on~$\bfP$ (variables of the outer problem).
While~\eqref{eqn:strong_part} can be formulated as a generalized semi-infinite program, current global optimization algorithms for this problem class do not scale well~\citep{jungen2023libdips}.
Therefore, we design a local optimization method for~\eqref{eqn:strong_part} aimed at determining a partitioning matrix $\bar{\bfP} \in \mathcal{P}_d$ that yields a tight lower bound $\underline{v}(\bar{\bfP})$.
We use this local solution of~\eqref{eqn:strong_part} to specify the partitioning matrix $\bfP^1$ at the first iteration of Algorithm~\ref{alg:partitioning_algorithm}.
If the resulting lower and upper bounds $LBD$ and $UBD$ in Algorithm~\ref{alg:partitioning_algorithm} do not converge after the first iteration, we use the partitioning policy implemented within Alpine (see Section~\ref{sec:alpinepart}) to specify the partitioning matrix $\bfP^l$ from its second iteration ($l \geq 2$).
{Section~\ref{subsec:strong_part_policy} outlines the strong partitioning policy.}

We use generalized gradients of the value function of problem~\eqref{eqn:piecewise_mccormick_mip_oa} within a bundle solver for nonsmooth nonconvex optimization to solve problem~\eqref{eqn:strong_part} to local optimality.
{Although value functions of MILPs are generally discontinuous, problem~\eqref{eqn:piecewise_mccormick_mip_oa} possesses special structure because outer-approximations of the {\pmc} relaxations $\mathcal{PMR}^B_{ij}(\bfP_i, \bfP_j)$ and $\mathcal{PMR}^{Q}_{k}(\bfP_k)$ can be described using nonconvex piecewise-linear continuous functions (cf.\ Figure~\ref{fig:piecewise_mccormick}). 
As shown in Theorems~\ref{thm:smooth_case} and~\ref{thm:nonsmooth_case} in Section~\ref{subsec:genlgrad}, this structure allows us to compute sensitivity information for the value function~$\underline{v}$.}

We use the bundle solver MPBNGC, which requires function and generalized gradient evaluations at points \mbox{$\bfP \in \mathcal{P}_d$} during its algorithm~\citep[see][for details]{makela2003multiobjective}.
{Each function evaluation $\underline{v}(\bfP)$ necessitates solving an MILP, \eqref{eqn:piecewise_mccormick_mip_oa}.
Under suitable assumptions, a generalized gradient $\partial \underline{v}(\bfP)$ can be obtained by fixing variables~$\bfY$ to an optimal $\bfY$-solution of~\eqref{eqn:piecewise_mccormick_mip_oa} and computing a generalized gradient of the resulting LP, \eqref{eqn:piecewise_mccormick_mip_abstract_inner}.
We formalize these details in Section~\ref{subsec:genlgrad}.}
Section S.1 of the supplemental material summarizes the convergence guarantees of MPBNGC for completeness.

The example below shows that the value function $\underline{v}$ of problem~\eqref{eqn:piecewise_mccormick_mip_oa} may be nonsmooth.

\begin{example}
\label{exm:nonsmooth_value_fn}
Consider the following instance of the QCQP~\eqref{eqn:orig_qcqp}:
\begin{align*}
\min_{x \in [0,1]} \:\: & x \:\: \text{ s.t. } \:\: x^2 \geq (0.4)^2.
\end{align*}
Its optimal solution is $x^* = 0.4$ with optimal value $v^* = 0.4$.
Suppose we wish to partition the domain of $x$ into two sub-intervals ($d = 1$).
Let $\bfP = (0, p, 1)$ denote partitioning points for $x$ with $0 \leq p \leq 1$.
After some algebraic manipulation, the outer-approximation problem~\eqref{eqn:piecewise_mccormick_mip_oa} can be equivalently written as:
\begin{align*}
\underline{v}(p) = \uset{x \in [0,1]}{\min} \: x \:\: \text{s.t.} \:\: w \geq (0.4)^2, \:\: w \leq \max\{px, (1+p)x - p\}, \:\: w \geq 2\alpha_j x - \alpha^2_j, \: \forall j \in \mathcal{J},
\end{align*}
where $\{\alpha_j\}_{j \in \mathcal{J}} \subset [0,1]$
and we write $\underline{v}(p)$ to only indicate the dependence of the value function $\underline{v}$ on the single nontrivial partitioning point~$p$.
We can derive the {\pmc} lower bound to be:

\begin{minipage}[t]{0.4\textwidth}
\vspace*{-1in}
\[
\underline{v}(p) = 
\begin{cases}
\frac{0.16+p}{1+p}, & \text{if $0 \leq p \leq 0.4$}\\
\frac{0.16}{p}, & \text{if $0.4 < p \leq 1$}
\end{cases},
\]
\end{minipage}
\hfill
\begin{minipage}[t]{0.5\textwidth}
\centering
\begin{tikzpicture}[scale=4]
\draw[->] (-0.2, 0) -- (1.2, 0) node[right] {$p$};
\draw[->] (0, -0.1) -- (0, 0.5) node[above] {$\underline{v}(p)$};
\draw[domain=0:0.4, smooth, thick, variable=\x, blue] plot ({\x}, {(0.16 + \x)/(1 + \x)});
\draw[domain=0.4:1, smooth, thick, variable=\x, blue] plot ({\x}, {0.16/\x});
\draw[densely dotted] (0.4, 0.4) -- (0.4, 0);
\draw[densely dotted] (0.4, 0.4) -- (0, 0.4);
\node at (-0.04,0) [below] {$0$};
\node at (0.4,0) [below] {$0.4$};
\node at (1,0) [below] {$1$};
\node at (0,0.16) [left] {$0.16$};
\node at (0,0.4) [left] {$0.4$};
\end{tikzpicture}
\end{minipage}

\noindent which shows that $\underline{v}$ is continuous and piecewise differentiable at $p = 0.4$ for this example.
\end{example}

\subsection{Computing a generalized gradient of \texorpdfstring{$\underline{v}$}{the value function}}
\label{subsec:genlgrad}

We identify conditions under which a generalized gradient of the value function $\underline{v}$ can be computed in practice. (For notational simplicity, we discuss generalized gradients of $\underline{v}$ on $\mathcal{P}_d$ rather than on $\text{ri}(\mathcal{P}_d)$.)
Omitted proofs are provided in Section S.2 of the supplemental material.

Our first result implies the active partitions of the variables $x_i$, $i \in \mathcal{NC}$, in problem~\eqref{eqn:piecewise_mccormick_mip_oa} for a given value of~$\bfP$ remain active for all partitioning matrices in a sufficiently small relative neighborhood of $\bfP$. 
Its assumption that the $\bfY$-solution to problem~\eqref{eqn:piecewise_mccormick_mip_oa} is unique can be checked by adding a ``no-good cut'' and re-solving~\eqref{eqn:piecewise_mccormick_mip_oa} to verify the second-best solution for $\bfY$ has a strictly greater objective than~$\underline{v}(\bfP)$.

\begin{lemma}
\label{lem:value_fn_active_part}
Fix $\bfP \in \mathcal{P}_d$.
Suppose problem~\eqref{eqn:piecewise_mccormick_mip_oa} has a unique $\bfY$-solution~$\bfY^* \in \mathcal{Y}$ and $v(\cdot,\bfY^*)$ is continuous at $\bfP$.
Then $\underline{v}(\tilde{\bfP}) = v(\tilde{\bfP},\bfY^*)$, $\forall \tilde{\bfP} \in \mathcal{P}_d$ in a sufficiently small relative neighborhood of~$\bfP$.
\end{lemma}

{The next result characterizes the gradient of the value function $\underline{v}$ with respect to the ``unfixed'' partitioning points $P_{i2}, \dots, P_{i(d+1)}$, $i \in \mathcal{NC}$.
It assumes that problem~\eqref{eqn:piecewise_mccormick_mip_abstract_inner}, with $\bfY$ fixed to the $\bfY$-solution $\bfY^*$ of problem~\eqref{eqn:piecewise_mccormick_mip_oa}, has unique primal and dual optimal solutions.}

\begin{theorem}
\label{thm:smooth_case}
Suppose $\bfP \in \mathcal{P}_d$ and problem~\eqref{eqn:piecewise_mccormick_mip_oa} has a unique $\bfY$-solution $\bfY^* \in \mathcal{Y}$.
Consider problem~\eqref{eqn:piecewise_mccormick_mip_abstract_inner} with $\bfY$ fixed to $\bfY^*$.
If this LP has a unique primal solution $\bfz^*$ and a unique dual solution $\bfpi^*$, then
\[
\frac{\partial \underline{v}}{\partial P_{ij}}(\bfP) = \frac{\partial v}{\partial P_{ij}}(\bfP,\bfY^*) = \sum_{k=1}^{n_r} \pi^*_k \bigg[ \bigg( \sum_{l=1}^{n_c} z^*_l \frac{\partial M_{kl}}{\partial P_{ij}}(\bfP) \bigg) - \frac{\partial d_{k}}{\partial P_{ij}}(\bfP) \bigg], \quad \forall i \in \mathcal{NC}, \: j \in \{2,\dots,d+1\}.
\]
\end{theorem}

{Next, we derive a formula for the Clarke generalized gradient $\partial \underline{v}(\bfP)$ when the assumption in Theorem~\ref{thm:smooth_case} that the LP~\eqref{eqn:piecewise_mccormick_mip_abstract_inner}, with $\bfY$ fixed to $\bfY^*$, has unique primal and dual solutions does \textit{not} hold.}
Note that $\partial \underline{v}(\bfP)$, $\partial v(\bfP,\bfY^*)$, $\frac{\partial M_{kl}}{\partial \bfP}(\bfP)$, and $\frac{\partial d_{k}}{\partial \bfP}(\bfP)$ denote generalized gradients only with respect to the ``unfixed'' partitioning points $P_{i2}, \dots, P_{i(d+1)}$, $i \in \mathcal{NC}$.

\begin{theorem}
\label{thm:nonsmooth_case}
Suppose $\bfP \in \mathcal{P}_d$ and problem~\eqref{eqn:piecewise_mccormick_mip_oa} has a unique $\bfY$-solution $\bfY^* \in \mathcal{Y}$.
Consider problem~\eqref{eqn:piecewise_mccormick_mip_abstract_inner} with $\bfY$ fixed to $\bfY^*$.
Suppose $v(\cdot,\bfY^*)$ is finite and locally Lipschitz in a sufficiently small relative neighborhood of $\bfP$.
Then
\begin{align*}
\hspace*{-0.25in} \partial \underline{v}(\bfP) &= \partial v(\bfP,\bfY^*) = \textup{conv}\biggl(\biggl\{\sum_{k=1}^{n_r} \pi^*_k \bigg[ \bigg( \sum_{l=1}^{n_c} z^*_l \frac{\partial M_{kl}}{\partial \bfP}(\bfP) \bigg) - \frac{\partial d_{k}}{\partial \bfP}(\bfP) \bigg] \: : \: (\bfz^*,\bfpi^*) \textup{ is a primal-dual } \\
&\hspace*{2.5in} \textup{optimal pair for } \eqref{eqn:piecewise_mccormick_mip_abstract_inner} \textup{ with } \bfY \textup{ fixed to } \bfY^*\biggr\}\biggr).
\end{align*}
\end{theorem}

{The next result ensures that $v(\cdot,\bfY^*)$ is locally Lipschitz in a sufficiently small relative neighborhood of $\bfP \in \mathcal{P}_d$, provided that the matrix $\bfM(\bfP)$ in problem~\eqref{eqn:piecewise_mccormick_mip_abstract_inner} has full row rank and Slater's condition holds.}
(See~\citet{im2018sensitivity} and Assumption~5.1 of~\citet{de2021generalized} for details.)

\begin{lemma}
\label{lem:genl_grad_assumptions}
Suppose $\bfP \in \mathcal{P}_d$ and $\bar{\bfY} \in \mathcal{Y}$.
Consider problem~\eqref{eqn:piecewise_mccormick_mip_abstract_inner} with $\bfY$ fixed to $\bar{\bfY}$.
If the matrix $\bfM(\bfP)$ has full row rank and 
$\bfd(\bfP) \in \textup{int}\bigl( \{ \bfM(\bfP)\bfz : \bfz \geq \bfzero, \bar{\bfM} \bfz = \bar{\bfB}\: \text{vec}(\bar{\bfY}) \}  \bigr)$, 
then $v(\cdot,\bar{\bfY})$ is finite and locally Lipschitz in a sufficiently small relative neighborhood of $\bfP$.
\end{lemma}
\proof{Proof.}
See Proposition~5.3 of~\cite{de2021generalized} and pages~73 to~76 of~\cite{im2018sensitivity}.
\Halmos
\endproof

We now verify that the full rank assumption in Lemma~\ref{lem:genl_grad_assumptions} holds in general.

\begin{lemma}
\label{lem:fullrank}
The matrix $\bfM(\bfP)$ has full row rank for each $\bfP \in \textup{ri}(\mathcal{P}_d)$.
\end{lemma}

The following is the main result of this section.
It shows that for almost every $\bfP \in \mathcal{P}_d$ with respect to the uniform measure, problem~\eqref{eqn:piecewise_mccormick_mip_oa} either has a unique $\bfY$-solution (allowing us to use Theorem~\ref{thm:smooth_case} or~\ref{thm:nonsmooth_case} to compute a generalized gradient of $\underline{v}$ under mild conditions), or $\underline{v}(\bfP) = v^*$ (i.e., the optimal values of~\eqref{eqn:qcqp} and~\eqref{eqn:piecewise_mccormick_mip_oa} are equal, implying the partitioning matrix $\bfP$ is sufficient for the convergence of the lower bound), or both statements hold.

\begin{theorem}
\label{thm:main_thm}
At least one of the following statements holds for almost every $\bfP \in \mathcal{P}_d$:
\begin{enumerate}
\item $\underline{v}(\bfP) = v^*$,

\item Problem~\eqref{eqn:piecewise_mccormick_mip_oa} has a unique $\bfY$-solution.
\end{enumerate}
\end{theorem}
\proof{Proof.}
Consider $\bfP \in \mathcal{P}_d$, and let $\hat{\bfY} \in \mathcal{Y}$ and $\hat{\bfx} \in [0,1]^n$ denote the $\bfY$ and $\bfx$ components of an optimal solution to problem~\eqref{eqn:piecewise_mccormick_mip_oa}.
We examine the following cases:
\begin{enumerate}[label=(\alph*)]
\item There exists an optimal solution $\hat{\bfx}$ such that $\forall k \in \mathcal{Q}$, $\hat{x}_k \in \{P_{k1}, P_{k2}, \dots, P_{k(d+1)}, P_{k(d+2)}\}$, and $\forall (i,j) \in \mathcal{B}$, either $\hat{x}_i \in \{P_{i1}, P_{i2}, \dots, P_{i(d+1)}, P_{i(d+2)}\}$, or $\hat{x}_j \in \{P_{j1}, P_{j2}, \dots, P_{j(d+1)}, P_{j(d+2)}\}$, or both.

\item Case (a) does not hold for \textit{any} optimal solution $\hat{\bfx}$, i.e., there either exists at least one index $k \in \mathcal{Q}$ such that $\hat{x}_k \not\in \{P_{k1}, P_{k2}, \dots, P_{k(d+1)}, P_{k(d+2)}\}$, or there exists at least one pair of indices $(i,j) \in \mathcal{B}$ such that both $\hat{x}_i \not\in \{P_{i1}, P_{i2}, \dots, P_{i(d+1)}, P_{i(d+2)}\}$ and $\hat{x}_j \not\in \{P_{j1}, P_{j2}, \dots, P_{j(d+1)}, P_{j(d+2)}\}$.
\end{enumerate}

Suppose case (a) holds.
Since we assume that $\{P_{k1}, P_{k2}, \dots, P_{k(d+1)}, P_{k(d+2)}\} \subset \{\alpha^k_j\}$ for each $k \in \mathcal{Q}$, our outer-approximation of the piecewise McCormick relaxation \eqref{eqn:lambda_form_quadratic_conv_comb}-\eqref{eqn:lambda_form_quadratic_active} for the constraint $W_{kk} = x^2_k$ is exact (i.e., there is no relaxation gap) at the partitioning points $x_k \in \{P_{k1}, P_{k2}, \dots, P_{k(d+1)}, P_{k(d+2)}\}$.
Additionally, the piecewise McCormick relaxation \eqref{eqn:lambda_form_bilinear_conv_comb}-\eqref{eqn:lambda_form_bilinear_active2} for the constraint $W_{ij} = x_i x_j$ is exact either when $x_i \in \{P_{i1}, P_{i2}, \dots, P_{i(d+1)}, P_{i(d+2)}\}$, or when $x_j \in \{P_{j1}, P_{j2}, \dots, P_{j(d+1)}, P_{j(d+2)}\}$, or both.
Therefore, the point $\hat{\bfx}$ is feasible to the original QCQP~\eqref{eqn:orig_qcqp}, which implies $\underline{v}(\bfP) = v^*$.

Suppose instead that case (b) holds.
Additionally, suppose there are multiple $\bfY$-solutions to~\eqref{eqn:piecewise_mccormick_mip_oa}.
Let $\tilde{\bfY} \in \mathcal{Y}$ and $\tilde{\bfx} \in [0,1]^n$ denote the $\bfY$ and $\bfx$ components of \textit{another} optimal solution to problem~\eqref{eqn:piecewise_mccormick_mip_oa} with $\tilde{\bfY} \neq \hat{\bfY}$.
Since case (a) does not hold and $\tilde{\bfY} \neq \hat{\bfY}$, we have $\tilde{\bfx} \neq \hat{\bfx}$.
Moreover, there exist nonsingular basis matrices $\hat{\bfM}(\bfP)$ and $\tilde{\bfM}(\bfP)$ for the LPs~\eqref{eqn:piecewise_mccormick_mip_abstract_inner} corresponding to $\hat{\bfY}$ and $\tilde{\bfY}$, respectively, such that 
\begin{align}
\label{eqn:polynomial_eqn}
\underline{v}(\bfP) = v(\bfP,\hat{\bfY}) = \tr{\hat{\bfc}} [\hat{\bfM}(\bfP)]^{-1} \begin{pmatrix}
\bfd(\bfP) \\
\bar{\bfB}\: \text{vec}(\hat{\bfY})
\end{pmatrix} 
= \tr{\tilde{\bfc}} [\tilde{\bfM}(\bfP)]^{-1}
\begin{pmatrix}
\bfd(\bfP) \\
\bar{\bfB}\: \text{vec}(\tilde{\bfY})
\end{pmatrix}
= v(\bfP,\tilde{\bfY})
\end{align}
for suitable vectors $\tilde{\bfc}$ and $\hat{\bfc}$, which include only the components of $\bfc$ corresponding to the basic variables of these LPs.
Since not all components of $\hat{\bfx}$ equal $0$ or $1$, at least some of the entries of $\hat{\bfM}(\bfP)$ are functions of the partitioning points~$\bfP$.
Moreover, $v(\bfP,\hat{\bfY})$ and $v(\bfP,\tilde{\bfY})$ are not identical functions of $\bfP$ since $\tilde{\bfx} \neq \hat{\bfx}$.
Therefore, equation~\eqref{eqn:polynomial_eqn} yields a polynomial equation in~$\bfP$, which implies that the set of all $\bfP \in \mathcal{P}_d$ for which equation~\eqref{eqn:polynomial_eqn} holds has measure zero.
Noting that $\abs{\mathcal{Y}} < +\infty$ and the number of possible bases is finite for each $\bfY \in \mathcal{Y}$ concludes the proof.
\Halmos
\endproof

\begin{algorithm}[t]
\caption{{Strong partitioning (SP) policy for the first iteration}}
\label{alg:enhancements}
{
\begin{algorithmic}[1]

\vspace*{0.05in}
\State \textbf{Input:} maximum number $d$ of partitioning points to be added per variable (excluding variable bounds).

\vspace*{0.05in}
\Statex \textbf{Preprocessing steps}

\State \textbf{Initialization:} partitioning vectors $\hat{\bfP}^0_i := (0, 1)$, $\forall i \in \mathcal{NC}$. \label{algo:preprocess_begin}

\For{$k = 1, 2, \dots, d$}

\State Solve problem~\eqref{eqn:piecewise_mccormick_mip_oa} with the partitioning matrix $\hat{\bfP}^{k-1}$. Let $\hat{\bfx}^{k-1}$ denote an $\bfx$-solution.

\For{$i \in \mathcal{NC}$} 

\If{$\hat{x}^{k-1}_i \approx \tilde{x}_i$ for some partitioning point $\tilde{x}_i$ in $\hat{\bfP}^{k-1}_i$}

\State Set $\hat{\bfP}^k_i = \hat{\bfP}^{k-1}_i$.

\Else

\State Insert $\hat{x}^{k-1}_i$ within $\hat{\bfP}^{k-1}_i$ to obtain a vector $\hat{\bfP}^k_i$ satisfying $\hat{P}^k_{ij} \leq \hat{P}^k_{i(j+1)}$, $\forall j \in [\textup{dim}(\hat{\bfP}^k_i) - 1]$.

\EndIf

\EndFor

\EndFor


\State For each $i \in \mathcal{NC}$, let $n_i := \textup{dim}(\hat{\bfP}^d_i) - 2$ and set $P^0_{ij} := \begin{cases}
0, & \text{if } j \in [d - n_i] \\
\hat{P}^d_{i(j-d+n_i)}, & \text{if } j \in \{d+1 - n_i, \dots, d+2\}
\end{cases}$.

\vspace*{0.04in}
\State For each $i \in \mathcal{NC}$, fix variables $P_{i1}, \dots, P_{i(d+1 - n_i)}$ to $0$ and variables $P_{i(d+2)}$ to $1$ while solving~\eqref{eqn:strong_part}.

\State \textbf{Preprocessing output:} initial guess (and variable fixings) $\bfP^0$ for problem~\eqref{eqn:strong_part}. \label{algo:preprocess_end}

\vspace*{0.07in}

\Statex \textbf{Solving the strong partitioning problem~\eqref{eqn:strong_part}}

\State Solve the max-min problem~\eqref{eqn:strong_part} using the initial guess (and variable fixings) $\bfP^0$ to obtain a solution $\bfP^1 \in \mathcal{P}_d$ with objective $\bar{v} := \underline{v}(\bfP^1)$.  \label{algo:maxmin}

\vspace*{0.07in}
\noindent
\textbf{Postprocessing steps}

\For{$j = 2, 3, \dots, d+1$} \label{algo:postprocess_begin}

\For{$i \in \mathcal{NC}$}

\State Set $\bfP = \bfP^1$, fix $P_{ij} = 0$, and sort $\bfP$ such that $P_{ik} \leq P_{i(k+1)}$, $\forall k \in [d+1]$.

\State Solve problem~\eqref{eqn:piecewise_mccormick_mip_oa} with this partitioning matrix $\bfP$ to obtain a lower bound $\hat{v} := \underline{v}(\bfP)$.

\If{$\hat{v} \geq \bar{v} - 10^{-6} \abs{\bar{v}}$}

\State Update $\bfP^1 = \bfP$.

\EndIf

\EndFor

\EndFor

\State \textbf{Postprocessing output:} partitioning matrix $\bfP^1$ used to construct {\pmc} relaxations in the first iteration of Algorithm~\ref{alg:partitioning_algorithm}. (In practice, redundant partitioning points are not added.)
\label{algo:postprocess_end}

\end{algorithmic}
}
\end{algorithm}

\subsection{{Strong partitioning policy}}
\label{subsec:strong_part_policy}

{Algorithm~\ref{alg:enhancements} details the strong partitioning policy for the first iteration.
It includes preprocessing steps to mitigate the computational burden of solving problem~\eqref{eqn:strong_part} to local optimality, and postprocessing steps to enable the ML model outlined in Section~\ref{sec:ml_approx} to more effectively imitate this policy.
The preprocessing heuristics determine an initial guess $\bfP^0$ for the solution of the max-min problem~\eqref{eqn:strong_part}. They also eliminate a subset of the {\pp} $\bfP$ by fixing them to zero, which can reduce both the per-iteration cost and the number of iterations required for the bundle solver to converge.
Despite these enhancements, we observe in our numerical experiments that solving the max-min problem~\eqref{eqn:strong_part} can still be computationally prohibitive.
Therefore, in the next section, we propose a practical off-the-shelf ML model to imitate this strong partitioning policy for homogeneous families of QCQPs.
The postprocessing heuristics in Algorithm~\ref{alg:enhancements} remove (redundant) partitioning points in the solution $\bfP^1$ of the max-min problem that do not significantly impact the piecewise McCormick relaxation lower bound.
Reducing the number of nontrivial partitioning points can enhance the ability of ML models to effectively imitate the SP policy.}

{The output partitioning matrix $\bfP^1$ of Algorithm~\ref{alg:enhancements} is used to construct {\pmc} relaxations in the first iteration of Algorithm~\ref{alg:partitioning_algorithm}.
If the lower bound $LBD$ obtained using $\bfP^1$ and the upper bound $UBD$ have not converged, the strong partitioning policy reverts to the heuristic partitioning policy implemented in Alpine (see Section~\ref{sec:alpinepart} for details) to specify the partitioning matrices $\bfP^l$ for iterations $l \geq 2$ of Algorithm~\ref{alg:partitioning_algorithm}.}

\section{{ML approximation of strong partitioning for {homogeneous} QCQPs}}
\label{sec:ml_approx}

{While strong partitioning can yield more effective {\pp} than existing heuristic partitioning methods, solving the max-min problem~\eqref{eqn:strong_part} to local optimality can be time-consuming, making its direct application within Algorithm~\ref{alg:partitioning_algorithm} impractical.
To address this, we propose using AdaBoost regression models~\citep{drucker1997improving,freund1997decision}, with regression tree base estimators~\citep{breiman2017classification}, to imitate the strong partitioning policy for {homogeneous} QCQP instances.
The trained AdaBoost models are then used to efficiently specify the partitioning matrix $\bfP^1$ for constructing {\pmc} relaxations in the first iteration of Algorithm~\ref{alg:partitioning_algorithm} for new QCQP instances from the family.
}

\begin{algorithm}[t]
\caption{{ML-based partitioning policy for the first iteration}}
\label{alg:ml_approx}
{
\begin{algorithmic}[1]

\State \textbf{Input:} maximum number $d$ of partitioning points per variable (excluding variable bounds), data $\{(\bff^l, \bfP^{1,l})\}_{l=1}^{N}$, where $\bff^l \in \R^{d_f}$ represents a feature vector and \mbox{$\bfP^{1,l} \in \mathcal{P}_d$} denotes the strong partitioning points determined using Algorithm~\ref{alg:enhancements} for the $l$th QCQP instance, number of data folds $K \in \{2,\dots,N\}$, maximum number $N_{wl}$ of weak learners for the \texttt{AdaBoostRegressor} model, and maximum depth $D_{max}$ of its \texttt{DecisionTreeRegressor} base estimator models.

\State \textbf{Preprocessing:} randomly split the set of instance indices $[N]$ into $K$ disjoint folds of (approximately) equal size. Let $\mathcal{I}_k \subset [N]$ denote the set of indices in the $k$th fold for each $k \in [K]$.

\vspace*{0.05in}
\For{$k = 1, 2, \dots, K$}

\For{$i \in \mathcal{NC}$}

\For{$j = 2, 3, \dots, d+1$}

\State Gather the training data $\mathcal{T}_{ijk} := \{(\bff^l,P^{1,l}_{ij})\}_{l\in [N] \backslash \mathcal{I}_k}$ consisting of input-output pairs

\Statex \hspace*{0.7in} corresponding to the $j$th partitioning point of variable $x_i$, omitting the $k$th data fold.

\State Train an \texttt{AdaBoostRegressor} model on $\mathcal{T}_{ijk}$ using a \texttt{DecisionTreeRegressor} 

\Statex \hspace*{0.7in} as the base estimator with maximum depth $D_{max}$, and a maximum of $N_{wl}$ weak learners. 

\Statex \hspace*{0.7in} \mbox{(All other hyperparameters for \texttt{AdaBoostRegressor} and \texttt{DecisionTreeRegressor}} 

\Statex \hspace*{0.7in}are set to their default values.)

\State Use the trained \texttt{AdaBoostRegressor} model along with the features $\{\bff^l\}_{l \in \mathcal{I}_k}$ to generate

\Statex \hspace*{0.7in} predictions $\{\hat{P}^{1,l}_{ij}\}_{l \in \mathcal{I}_k}$, with each $\hat{P}^{1,l}_{ij} \in [0,1]$, for the strong partitioning points $\{P^{1,l}_{ij}\}_{l \in \mathcal{I}_k}$

\Statex \hspace*{0.7in} for instances in the $k$th data fold.

\EndFor

\For{$l \in \mathcal{I}_k$}

\State Sort the predictions $\{\hat{P}^{1,l}_{ij}\}_{j=2}^{d+1}$ such that $\hat{P}^{1,l}_{ij} \leq \hat{P}^{1,l}_{i(j+1)}$, $\forall j \in \{2,\dots, d\}$.

\EndFor

\EndFor

\EndFor

\vspace*{0.05in}
\State \textbf{Output:} $K$-fold out-of-sample ML predictions $\{\hat{\bfP}^{1,l}\}_{l=1}^{N}$, where $\hat{\bfP}^{1,l} \in \mathcal{P}_d$ is used to construct {\pmc} relaxations in the first iteration of Algorithm~\ref{alg:partitioning_algorithm} for the $l$th QCQP instance.

\end{algorithmic}
}
\end{algorithm}

Our AdaBoost regression models sequentially train a series of regression trees, starting with an initial regression tree that fits the data. 
Each subsequent regression tree is trained by placing more emphasis on the training instances that the previous models struggled to predict accurately, achieved by adjusting the relative weights of those instances. 
The final predictions of the AdaBoost regression models are then obtained by combining the outputs of all the regression tree models using a weighted average. 
By iteratively focusing on the most challenging training instances, these models capture complex relationships between the input features and the output strong partitioning points while mitigating overfitting. 
However, a potential drawback of AdaBoost regression models is their complexity, which can make them difficult to interpret.

{Algorithm~\ref{alg:ml_approx} outlines the machine learning-based partitioning policy for the first iteration.
It takes as input a feature vector for each of $N$ QCQP instances from the family, along with the strong partitioning points determined using Algorithm~\ref{alg:enhancements} for these instances.
The algorithm begins by splitting the training data into $K$ disjoint folds.
For each fold $k \in [K]$, it excludes the feature and strong partitioning data from that fold and learns a separate AdaBoost regression model for each partitioning point $P^1_{ij}$, $i \in \mathcal{NC}$ and $j \in \{2,\dots,d+1\}$, to imitate the strong partitioning policy for specifying that partitioning point.
This trained regression model is then used to predict the strong partitioning points $\{P^{1,l}_{ij}\}_{l \in \mathcal{I}_k}$ for the instances in the $k$th fold.}

{The output partitioning matrix $\hat{\bfP}^{1,l}$ from Algorithm~\ref{alg:ml_approx} is used to construct {\pmc} relaxations at the first iteration of Algorithm~\ref{alg:partitioning_algorithm} for the $l$th QCQP instance.
Similar to the strong partitioning policy, this AdaBoost-based policy reverts to the heuristic partitioning policy implemented in Alpine (see Section~\ref{sec:alpinepart} for details) to specify the partitioning matrices from the second iteration of Algorithm~\ref{alg:partitioning_algorithm}.}

We use the following instance-specific features as inputs to our AdaBoost regression models: 
\begin{enumerate}[label=(\roman*)]
\item Parameters~$\bftheta$ that uniquely parametrize each QCQP instance in the family (see Section~\ref{subsec:test_instances} for details). 
\item The best found feasible solution during presolve, obtained via a single local solve of~\eqref{eqn:qcqp}. 
\item The McCormick lower bounding solution, obtained by solving the convex problem~\eqref{eqn:mccormick}. 
\end{enumerate}
Although it is theoretically sufficient to use only the parameters~$\bftheta$ as features, since they uniquely identify each QCQP instance in the family, we also include features (ii) and (iii) as they are relatively inexpensive to compute and intuitively help inform the partitioning strategy. 
These additional features are complex transformations of the parameters~$\bftheta$, which might be difficult to uncover otherwise.

In contrast with much of the literature on learning for MILPs, we train separate ML models for each QCQP family since both the feature and output dimensions of our AdaBoost regression models depend on the problem dimensions.
{While we plan to design more advanced ML architectures that can accommodate variable feature and output dimensions in future work, we do not view the need to train different ML models for each QCQP family to be a major limitation.
This is because decision-makers often care about solving instances of the \textit{same} underlying QCQP with only a few varying model parameters, which means they only need to train a single set of ML models with fixed feature and output dimensions for their QCQP family.}

{
Finally, we mention that we briefly explored the use of neural network models as alternatives to our AdaBoost regression models.
Although we do not provide detailed results, we found that our AdaBoost regression models significantly outperformed these simple neural networks, both in predicting the strong partitioning points and in the performance of the predicted points when used to specify the partitions in the first iteration of Algorithm~\ref{alg:partitioning_algorithm}.
We posit that AdaBoost models might perform better than neural networks for the following reasons.
First, AdaBoost models tend to identify redundant strong partitioning points more accurately.
The preprocessing and postprocessing heuristics in Algorithm~\ref{alg:enhancements} are designed to set partitioning points to zero (making them redundant) if they are not expected to significantly impact the {\pmc} relaxation lower bound.
We observe that our AdaBoost models typically predict these redundant partitioning points as zero.
Consequently, using AdaBoost predictions to inform the partitioning matrix $\bfP^1$ in Algorithm~\ref{alg:partitioning_algorithm} results in easier {\pmc} relaxation problems~\eqref{eqn:piecewise_mccormick_mip} in the first iteration, leading to significant computational speedups.
Next, the optimal solution of the max-min problem~\eqref{eqn:strong_part} is often discontinuous when viewed as a function of the QCQP instance features.
AdaBoost models may be better suited to predict this discontinuous mapping from the features of a QCQP instance to an optimal solution of~\eqref{eqn:strong_part}, particularly in cases where training data is scarce.}

{In the next section, we evaluate both the quality of our AdaBoost regression model predictions and their effectiveness when used to specify the partitioning matrix $\bfP^1$ in the first iteration of Algorithm~\ref{alg:partitioning_algorithm}.}

\section{Numerical experiments}
\label{sec:computexp}

We briefly describe the families of test instances used in our experiments in Section~\ref{subsec:test_instances}.
Section~\ref{subsec:computational_setup_and_outline} outlines our computational setup and the metrics used in our experiments.
Section~\ref{sec:alpinepart} details the partitioning policy implemented within Alpine, which serves as a benchmark for our strong partitioning and AdaBoost regression-based policies.
In Section~\ref{subsec:ml_approx}, we evaluate the effectiveness of the AdaBoost models in imitating the strong partitioning policy.
Finally, Section~\ref{subsec:results} compares the performance of Alpine's default partitioning policy with the strong partitioning and AdaBoost policies in selecting Alpine's partitioning points for the repeated solution of homogeneous QCQPs with fixed structure.
Sections S.3 to S.5 of the supplemental material provide the full methodology for generating our test instances, benchmark their difficulty using the solvers BARON and Gurobi, and present additional tables for our numerical experiments.
The data and code supporting this study are available on GitHub~\citep{kannan2025github}.

\subsection{Test instances}
\label{subsec:test_instances}

Motivated by real-world applications that require solving the same underlying QCQP with slightly varying model parameters, we evaluate our approach on three homogeneous families of nonconvex QCQPs.

\subsubsection{Random bilinear programs}
\label{subsubsec:random_bilinear}

We consider a parametric family of bilinear programs of the form:
\begin{alignat*}{2}
v(\bftheta) := &\min_{\bfx \in [0,1]^{n}} \:\: && \bfx^{\textup{T}} \bfQ^0(\bftheta) \bfx + (\bfr^0(\bftheta))^{\textup{T}} \bfx \\
&\quad\: \text{s.t.} && \bfx^{\textup{T}} \bfQ^i(\bftheta) \bfx + (\bfr^i(\bftheta))^{\textup{T}} \bfx \leq b_i, \quad \forall i \in [m_I], \\
& && (\bfa^j)^{\textup{T}} \bfx = d_j, \quad \forall j \in [m_E],
\end{alignat*}
where the matrices and vectors depend affinely on a parameter vector~$\bftheta \in [-1,1]^{d_{\theta}}$. 
We generate $1000$ random instances each for $n \in \{10, 20, 50\}$ with $\abs{\mathcal{B}} = \min\{5n, \binom{n}{2}\}$ bilinear terms, $\abs{\mathcal{Q}} = 0$ quadratic terms, $m_I = n$ bilinear inequalities, and $m_E = 0.2n$ linear equalities~\citep{bao2011semidefinite}. 
Each family of instances with a fixed dimension $n$ is constructed to have the same set of $\abs{\mathcal{B}}$ bilinear terms.

\subsubsection{Random QCQPs with bilinear \textit{and} univariate quadratic terms}
\label{subsubsec:random_qcqps}

We additionally generate $1000$ random QCQP instances each for $n \in \{10, 20, 50\}$ with $\abs{\mathcal{B}} = \min\{5n, \binom{n}{2}\}$ bilinear terms and $\abs{\mathcal{Q}} = \lfloor 0.25n \rfloor$ univariate quadratic terms. 
All instances for a fixed dimension $n$ comprise the same set of bilinear and univariate quadratic terms.
The remaining problem structure and data generation follow the same approach as in the bilinear case.

\subsubsection{The pooling problem}
\label{subsubsec:pooling_problem}

We formulate the pooling problem as a parametric family of bilinear programs of the form~\citep{luedtke2020strong}:
\begin{alignat*}{2}
v(\bftheta) := &\min_{\bfx \in [\bfzero,\mathbf{u}]} \:\: && (\bfr^0)^{\textup{T}} \bfx \\
&\quad\: \text{s.t.} && \bfx^{\textup{T}} \bfQ^i(\bftheta) \bfx + (\bfr^i(\bftheta))^{\textup{T}} \bfx \leq 0, \quad \forall i \in [m_I], \\
& && \bfx^{\textup{T}} \bfQ^i \bfx + (\bfr^i)^{\textup{T}} \bfx = 0, \quad \forall i \in \{m_I+1,\dots,m_E\}, \\
& && (\bfa^j)^{\textup{T}} \bfx \leq d_j, \quad \forall j \in [m'_I], \\
& && (\bfa^j)^{\textup{T}} \bfx = 1, \quad \forall j \in \{m'_I+1,\dots,m'_E\},
\end{alignat*}
where the parameters~$\bftheta$ correspond to input qualities.
We generate $1000$ pooling instances by perturbing the input qualities by up to 20\%.

\subsection{{Computational setup and evaluation metrics for our numerical experiments}}
\label{subsec:computational_setup_and_outline}

\subsubsection{{Computational setup}}
\label{subsubsec:computational_setup}

We evaluate the performance of different partitioning policies within Alpine (\url{https://github.com/lanl-ansi/Alpine.jl}).
We use Julia 1.6.3 and JuMP.jl~v1.1.1 to formulate our test instances and Alpine.jl v0.4.1 to solve them.
We use Gurobi 9.1.2 via Gurobi.jl v0.11.3 with $\texttt{MIPGap} = 10^{-6}$ for solving LPs, MILPs, and convex mixed-integer QCQPs within Alpine. 
We use Ipopt 3.14.4 via Ipopt.jl~v1.0.3 (with $\texttt{max\_iter} = 10^4$) to solve the random bilinear and QCQP instances in Sections~\ref{subsubsec:random_bilinear} and~\ref{subsubsec:random_qcqps} locally within Alpine.
To solve the random pooling instances in Section~\ref{subsubsec:pooling_problem} locally within Alpine, we switch to Artelys Knitro 12.4.0 via KNITRO.jl v0.13.0 (with $\texttt{algorithm} = 3$) due to Ipopt's reduced effectiveness for these instances.
Each Alpine run is assigned a time limit of $2$ hours, with target relative and absolute optimality gaps of $10^{-4}$ and $10^{-9}$. 
We deactivate bound tightening techniques within Alpine due to their ineffectiveness on our medium and large-scale instances. 
We partition the domains of all variables involved in nonconvex terms within Alpine and set its remaining options to their default values.

We evaluate the performance of the strong partitioning and ML-based policies when used to specify the partitioning matrix $\bfP^1$ in Alpine's first iteration.
Algorithm~\ref{alg:enhancements} for strong partitioning is implemented in Julia~1.6.3 and integrated within Alpine.jl v0.4.1. 
To solve problem~\eqref{eqn:strong_part} to local optimality, we use the bundle solver MPBNGC 2.0~\citep{makela2003multiobjective} via MPBNGCInterface.jl (\url{https://github.com/milzj/MPBNGCInterface.jl}), with the options \mbox{$\texttt{OPT\_LMAX} = 20$}, $\texttt{OPT\_EPS} = 10^{-9}$, and $\texttt{OPT\_NITER} = \texttt{OPT\_NFASG} = 500$.
We do not specify a time limit for Algorithm~\ref{alg:enhancements}.
We consider strong partitioning with either two or four {\pp} per partitioned variable ($d = 2$ or $d = 4$) in addition to the variable bounds.
We use scikit-learn~v0.23.2~\citep{scikit-learn} to design our AdaBoost regression-based approximation of strong partitioning, setting $N = 1000$, $K = 10$, $N_{wl} = 1000$, $D_{max} = 25$, and either $d = 2$ or $d = 4$ in Algorithm~\ref{alg:ml_approx}.
We do not carefully tune the hyperparameters in Algorithm~\ref{alg:ml_approx}, as our primary focus is on the performance of partitioning points prescribed by our AdaBoost-based policy when used within Algorithm~\ref{alg:partitioning_algorithm}.

All of our experiments were conducted 
on nodes of the Darwin cluster at LANL, which are equipped with dual socket Intel Broadwell 18-core processors (E5-2695\_v4 CPUs, base clock rate of 2.1 GHz), EDR InfiniBand, and 125 GB of memory.
Each instance was run exclusively on a single node, and different solution approaches were executed sequentially to minimize the impact of variability in machine performance.

\subsubsection{{Evaluation metrics for our numerical experiments}}
\label{subsubsec:outline_experiments}

{We assess the effectiveness of the AdaBoost regression models in imitating the strong partitioning policy by evaluating the scaled mean absolute errors (MAEs) of their out-of-sample predictions for the strong partitioning points.}
{We then compare the performance of Alpine’s default partitioning policy with the use of the strong partitioning and AdaBoost-based policies within Alpine using the following metrics:}
\begin{enumerate}[label=\roman*.]
\item \textbf{Solution times:} shifted GM, median, minimum, and maximum solution times, as well as statistics on the speedup or slowdown relative to the performance of Alpine's default partitioning policy.
\item \textbf{Effective optimality gaps:} GM, median, minimum, and maximum effective optimality gap after Alpine's first iteration, along with the percentage of instances for which the effective optimality gap reaches the minimum value of $10^{-4}$ after the first iteration.
\item \textbf{Number of instances solved:} the number of instances solved within the time limit and the GM of the residual optimality gap for the unsolved instances.
\end{enumerate}
The shifted geometric mean of a positive vector $t$ of solution times (in seconds) is defined as:
\[
\text{Shifted GM}(t) = \exp\Bigl(\frac{1}{N}\sum_{i=1}^{N} \ln(t_i + 10) \Bigr) - 10.
\]
The residual optimality gap for an unsolved instance is defined as:
\[
\textup{TLE Gap } = \frac{\text{UB} - \text{LB}}{10^{-6} + \abs{\text{UB}}},
\]
where $\text{UB}$ and $\text{LB}$ are the upper and lower bounds on the optimal value returned by the solver at termination.

We define the \textit{effective} relative optimality gap as:
\begin{align}
\label{eqn:effective_gap}
\text{Effective Optimality Gap } &= \max \biggl\{ 10^{-4}, \frac{v^* - v^{\text{LBD}}}{10^{-6} + \abs{v^*}} \biggr\},
\end{align}
where $v^*$ is the optimal objective value, $v^{\text{LBD}}$ is Alpine's lower bound after one iteration, and $10^{-4}$ is the target relative optimality gap.
By measuring the gap of $v^{\text{LBD}}$ relative to the optimal value $v^*$ instead of the best found feasible solution, we do not let the performance of the local solver impact our evaluation of the different partitioning policies.
Thresholding this optimality gap at $10^{-4}$ also lends equal importance to all optimality gaps less than this target since all such gaps are sufficient for Alpine's convergence.

We also plot solution profiles comparing the performance of different partitioning policies within Alpine, and histograms showing the reduction in effective optimality gaps achieved by using the strong partitioning and ML-based policies relative to Alpine's default partitioning policy.
We do not include performance profiles due to their known issues (see \url{http://plato.asu.edu/bench.html}).

\subsection{{Alpine's partitioning policy}}
\label{sec:alpinepart}

Algorithm~\ref{alg:adaptive_partitioning} outlines the partitioning policy implemented within Alpine~\citep{nagarajan2019adaptive}.
It adds up to two {\pp} to the active partition for each variable $x_i$, $i \in \mathcal{NC}$, around a reference point~$\bar{x}^{l-1}_i$.
For the first iteration, $\bfP^0_i := (0,1)$, $\forall i \in \mathcal{NC}$, and the reference point $\bar{\bfx}^0$ is set either to a feasible local solution from presolve, if one is found, or to a solution to the termwise McCormick relaxation~\eqref{eqn:mccormick} otherwise.
In subsequent iterations ($l > 1$), $\bar{\bfx}^{l-1}$ is specified as the $\bfx$-component of a solution to the {\pmc} relaxation problem~\eqref{eqn:piecewise_mccormick_mip} at iteration $l-1$, constructed using the partitioning matrix $\bfP^{l-1}$.
The parameter~$\Delta$ (default value $= 10$) is a dimensionless scaling factor for the size of the partition constructed around the reference point~$\bar{\bfx}^{l-1}$.
A larger value of $\Delta$ results in a narrower partition around $\bar{\bfx}^{l-1}$.

This policy empirically performs well on Alpine's test library and is motivated by the observation that uniformly partitioning variable domains~\citep[as proposed in][]{castro2016normalized,saif2008global,wicaksono2008piecewise} often creates many partitions that do not significantly improve the {\pmc} relaxation lower bound. 
However, it relies on heuristic choices that could be improved.
For example, it uses the \textit{same} parameter $\Delta$ to partition the domains of \textit{all} variables and only considers symmetric partition refinements around the point~$\bar{\bfx}^{l-1}$.
Additionally, the quality of its partitioning points $\bfP^1$ in the first iteration depends on the quality of the feasible solution found during presolve, with suboptimal presolve solutions potentially leading to slow convergence.
Numerical experiments in Section~\ref{subsec:results} demonstrate how strong partitioning and its ML approximation effectively address these limitations of Alpine's partitioning policy.

\begin{algorithm}[t]
\caption{{Alpine's adaptive partition refinement policy for iteration $l$}}
\label{alg:adaptive_partitioning}
{
\begin{algorithmic}[1]

\State \textbf{Input:} partitioning matrix $\bfP^{l-1} \in [0,1]^{\abs{\mathcal{NC}} \times 2l}$ used to construct {\pmc} relaxations at iteration $l-1$, index $\mathcal{A}(i,l-1) \in [2l - 1]$ of an active partition for variable $x_i$ at iteration $l-1$, $\forall i \in \mathcal{NC}$, a reference point $\bar{\bfx}^{l-1}$ within the active partition, i.e., with $\bar{x}^{l-1}_i \in [P^{l-1}_{i(\mathcal{A}(i,l-1))}, P^{l-1}_{i(\mathcal{A}(i,l-1) + 1)}]$, $\forall i \in \mathcal{NC}$, and parameter $\Delta \geq 4$.

\State \textbf{Initialization:} set $\bfP^l = \bfP^{l-1}$.

\vspace*{0.05in}
\For{$i \in \mathcal{NC}$}

\State Let $\text{width}(\mathcal{A}(i,l-1)) := P^{l-1}_{i(\mathcal{A}(i,l-1) + 1)} - P^{l-1}_{i(\mathcal{A}(i,l-1))}$. Add two partitioning points $\underline{p}^l_i$ and $\bar{p}^l_i$ to the

\Statex \hspace*{0.2in} active partition $[P^{l-1}_{i(\mathcal{A}(i,l-1))}, P^{l-1}_{i(\mathcal{A}(i,l-1) + 1)}]$ for variable $x_i$ as follows:
\begin{align*}
\bfP^l_i &:= \biggl( P^{l-1}_{i1}, \dots, P^{l-1}_{i(\mathcal{A}(i,l-1))}, \: \underline{p}^l_i, \: \bar{p}^l_i, \dots, P^{l-1}_{i(2l)} \biggr), \quad \text{where} \\
\underline{p}^l_i &:= \max\biggl\{P^{l-1}_{i(\mathcal{A}(i,l-1))}, \: \bar{x}^{l-1}_i - \frac{\text{width}(\mathcal{A}(i,l-1))}{\Delta}\biggr\}, \\
\bar{p}^l_i &:= \min\biggl\{P^{l-1}_{i(\mathcal{A}(i,l-1) + 1)}, \: \bar{x}^{l-1}_i + \frac{\text{width}(\mathcal{A}(i,l-1))}{\Delta}\biggr\}.
\end{align*}

\Statex \hspace*{0.2in} {\big(In practice, the partitioning points $\underline{p}^l_i$ and $\bar{p}^l_i$ are not added to $\bfP^l_i$ if $\bar{x}^{l-1}_i - \frac{\text{width}(\mathcal{A}(i,l-1))}{\Delta} \leq P^{l-1}_{i(\mathcal{A}(i,l-1))}$}

\Statex \hspace*{0.3in} {or $\bar{x}^{l-1}_i + \frac{\text{width}(\mathcal{A}(i,l-1))}{\Delta} \geq P^{l-1}_{i(\mathcal{A}(i,l-1)+1)}$, respectively.\big)}

\EndFor

\vspace*{0.05in}
\State \textbf{Output:} partitioning matrix $\bfP^l$ used to construct {\pmc} relaxations at iteration $l$.

\end{algorithmic}
}
\end{algorithm}

\begin{table}[t]
\centering
\caption{Statistics on the scaled mean absolute errors (MAEs) of the out-of-sample predictions of the AdaBoost regression models when they are trained to predict $d = 2$ strong partitioning points per variable.}
\begin{tabular}{ r | c c c c c }
\hline
Scaled MAE & $<0.01$ & $<0.02$ & $<0.05$ & $<0.1$ & $<0.2$ \\ 
& \multicolumn{5}{c}{\% of the $2\abs{\mathcal{NC}}$ Partitioning Points} \\ \hline
Bilinear, $n = 10$ &  60  &  75   &  80   &  95  &  100  \\
Bilinear, $n = 20$ &  15  &  22.5   &  60  &  87.5  &  97.5  \\
Bilinear, $n = 50$ &  31  &  39   &  70  &  94  &  100  \\[0.05in]
QCQP, $n = 10$ &  65  &  80   &  95   &  100  &  100  \\
QCQP, $n = 20$ &  35  &  37.5   &  77.5  &  92.5  &  100  \\
QCQP, $n = 50$ &  56  &  66   &  85  &  99  &  100  \\[0.05in]
Pooling           &  65.7  &  70.9   &  78.6   &  89.5  &  97.2  \\ \hline
\end{tabular}
\label{tab:scaled_maes}
\end{table}

\subsection{{Effectiveness of the ML model in imitating the strong partitioning policy {for homogeneous QCQPs}}}
\label{subsec:ml_approx}

We evaluate the effectiveness of our AdaBoost regression models in imitating the strong partitioning policy for the different QCQP families.
We do not extensively tune the hyperparameters of our AdaBoost models, as our primary focus is on the performance of the partitioning points they prescribe when used within Algorithm~\ref{alg:partitioning_algorithm}.
Section~\ref{subsubsec:computational_setup} details the parameter settings used in Algorithm~\ref{alg:ml_approx}.

Table~\ref{tab:scaled_maes} presents statistics on the scaled mean absolute errors (MAEs) of the out-of-sample predictions for the $2\abs{\mathcal{NC}}$ {\pp} generated by the AdaBoost models.
These models are trained using $10$-fold cross-validation ($K = 10$) to predict $d = 2$ strong partitioning points per variable.
The MAEs are averaged over $1000$ instances in each family and are scaled by the upper bounds of the corresponding $\bfx$ variables (which are equal to one for the random bilinear and QCQP instances).
Approximately $90\%$ or more of the {\pp} predicted by the AdaBoost models have a scaled MAE of less than $10\%$ for each problem family, indicating that this ML framework effectively and efficiently imitates strong partitioning across different problem families.

\begin{table}[t]
\centering
\caption{
Statistics on solution times. 
Columns correspond to the shifted geometric mean, median, minimum, and maximum times over the subset of 1000 instances that did not hit the {2 hour} time limit. 
{The times for Alpine+SP2 and Alpine+SP4 do not include the time for running Algorithm~\ref{alg:enhancements} to determine strong partitioning points.}
{The times for Alpine+ML2 and Alpine+ML4 include the time required to generate and sort the ML predictions in Algorithm~\ref{alg:ml_approx}.}
{The last two columns denote the number of instances for which each method hits the time limit and the corresponding geometric mean of residual optimality gaps at termination, respectively.}
}
\begin{tabular}{ c | c | c c c c | c c }
\hline
Problem Family & Solution Method & \multicolumn{4}{c|}{Solution Time (seconds)} & \\
& & Shifted GM & Median & Min & Max & \# TLE & TLE Gap (GM) \\ \hline
&  Alpine (default)  &  0.51  &  0.47  &  0.14  &  2.41  &  0  &  $\underline{\hspace{0.3cm}}$  \\
Bilinear, $n = 10$  &  Alpine+SP2  &  0.11  &  0.10  &  0.06  &  0.28  &  0  &  $\underline{\hspace{0.3cm}}$  \\
&  Alpine+ML2  &  0.15  &  0.10  &  0.06  &  1.64  &  0  &  $\underline{\hspace{0.3cm}}$  \\[0.1in]
&  Alpine (default)  &  21.4  &  21.9  &  5.1  &  161.5  &  0  &  $\underline{\hspace{0.3cm}}$  \\
&  Alpine+SP2  &  4.2  &  2.0  &  0.8  &  132.6  &  0  &  $\underline{\hspace{0.3cm}}$  \\
Bilinear, $n = 20$  &  Alpine+ML2  &  10.0  &  7.8  &  1.1  &  116.0  &  0  &  $\underline{\hspace{0.3cm}}$  \\
&  Alpine+SP4  &  2.4  &  1.9  &  0.8  &  94.2  &  0  &  $\underline{\hspace{0.3cm}}$  \\
&  Alpine+ML4  &  9.3  &  7.2  &  1.0  &  117.4  &  0  &  $\underline{\hspace{0.3cm}}$  \\[0.1in]
&  Alpine (default)  &  405.9  &  336.2  &  48.0  &  7135.9  &  24  &  $4.4 \times 10^{-4}$  \\
Bilinear, $n = 50$  &  Alpine+SP2  &  52.8  &  34.9  &  4.2  &  5705.1  & 4  &  $1.6 \times 10^{-4}$  \\
&  Alpine+ML2  &  101.6  &  83.6  &  6.6  &  7071.7  &  5  &  $1.8 \times 10^{-4}$  \\[0.1in]
&  Alpine (default)  &  0.85  &  0.81  &  0.62  &  2.29  &  0  &  $\underline{\hspace{0.3cm}}$  \\
QCQP, $n = 10$  &  Alpine+SP2  &  0.10  &  0.09  &  0.07  &  0.27  &  0  &  $\underline{\hspace{0.3cm}}$  \\
&  Alpine+ML2  &  0.27  &  0.12  &  0.07  &  2.89  &  0  &  $\underline{\hspace{0.3cm}}$  \\[0.1in]
&  Alpine (default)  &  40.1  &  35.6  &  4.6  &  241.1  &  0  &  $\underline{\hspace{0.3cm}}$  \\
&  Alpine+SP2  &  7.7  &  1.7  &  0.8  &  135.4  &  0  &  $\underline{\hspace{0.3cm}}$  \\
QCQP, $n = 20$  &  Alpine+ML2  &  13.0  &  9.5  &  1.0  &  180.1  &  0  &  $\underline{\hspace{0.3cm}}$  \\
&  Alpine+SP4  &  2.4  &  1.5  &  0.7  &  125.7  &  0  &  $\underline{\hspace{0.3cm}}$  \\
&  Alpine+ML4  &  9.4  &  6.4  &  0.9  &  101.2  &  0  &  $\underline{\hspace{0.3cm}}$  \\[0.1in]
&  Alpine (default)  &  391.5  &  289.1  &  36.6  &  7198.2  &  0  &  $\underline{\hspace{0.3cm}}$  \\
QCQP, $n = 50$  &  Alpine+SP2  &  63.3  &  51.9  &  4.2  &  6055.2  &  0  &  $\underline{\hspace{0.3cm}}$ \\
&  Alpine+ML2  &  100.5  &  118.2  &  5.3  &  6514.2  &  0  &  $\underline{\hspace{0.3cm}}$  \\[0.1in]
&  Alpine (default)  &  242.8  &  212.5  &  25.9  &  7091.9  &  7  &  $2.9 \times 10^{-4}$  \\
Pooling  &  Alpine+SP2  &  66.7  &  49.7  &  1.6  &  6127.1  & 5  &  $2.1 \times 10^{-4}$  \\
&  Alpine+ML2  &  117.1  &  101.9  &  11.4  &  6097.0  &  1  &  $2.8 \times 10^{-4}$  \\ \hline
\end{tabular}
\label{tab:alpine_times}
\end{table}

\subsection{{Evaluating the performance of strong partitioning and ML-based policies in Alpine}}
\label{subsec:results}

We compare the performance of Alpine's default partitioning policy with the strong partitioning (Alpine+SP2) and AdaBoost-based (Alpine+ML2) policies, which select two partitioning points per variable in each iteration of Alpine, using the metrics detailed in Section~\ref{subsubsec:outline_experiments}.
For the bilinear $n = 20$ and QCQP $n = 20$ families, we also compare these policies with the strong partitioning (Alpine+SP4) and AdaBoost-based (Alpine+ML4) policies that select four partitioning points per variable in Alpine’s first iteration.
{We reiterate that the strong partitioning and AdaBoost policies differ from Alpine's default partitioning policy only during the first iteration of Algorithm~\ref{alg:partitioning_algorithm}.}
{All reported times for Alpine with the strong partitioning and AdaBoost policies exclude the time required to run Algorithm~\ref{alg:enhancements} and Algorithm~\ref{alg:ml_approx}.}

Table~\ref{tab:alpine_times} presents statistics on the run times of Alpine with the default, strong partitioning, and AdaBoost-based partitioning policies across the different QCQP families.
Table~\ref{tab:alpine_gaps} provides statistics on the effective optimality gaps~\eqref{eqn:effective_gap} of Alpine with the different partitioning policies after the first iteration.
Figures~\ref{fig:plots_bilinear},~\ref{fig:plots_qcqp}, and~\ref{fig:plots_pooling} plot solution profiles and histograms showing the reduction in effective optimality gaps achieved after the first iteration by the strong partitioning and AdaBoost policies relative to Alpine's default partitioning policy.
Section S.5 of the supplemental material presents tables reporting the speedup or slowdown of Alpine with the strong partitioning and AdaBoost-based policies relative to default Alpine, as well as statistics on the time required to run Algorithm~\ref{alg:enhancements} to determine strong partitioning points for the different QCQP families.

\begin{figure}
\centering
\caption{
Results for the random bilinear instances. 
{The times for Alpine+SP2 and Alpine+SP4 do not include the time for running Algorithm~\ref{alg:enhancements} to determine strong partitioning points.}
Top row: solution profiles indicating the percentage of instances solved by the different methods within time T seconds (higher is better). 
Bottom row: histograms of the ratios of the effective optimality gaps~\eqref{eqn:effective_gap} of default Alpine with Alpine+SP2 and with Alpine+ML2 after one iteration (larger is better).}
\label{fig:plots_bilinear}

\vspace*{0.1in}
\begin{subfigure}[b]{0.32\textwidth}
\includegraphics[width=\textwidth]{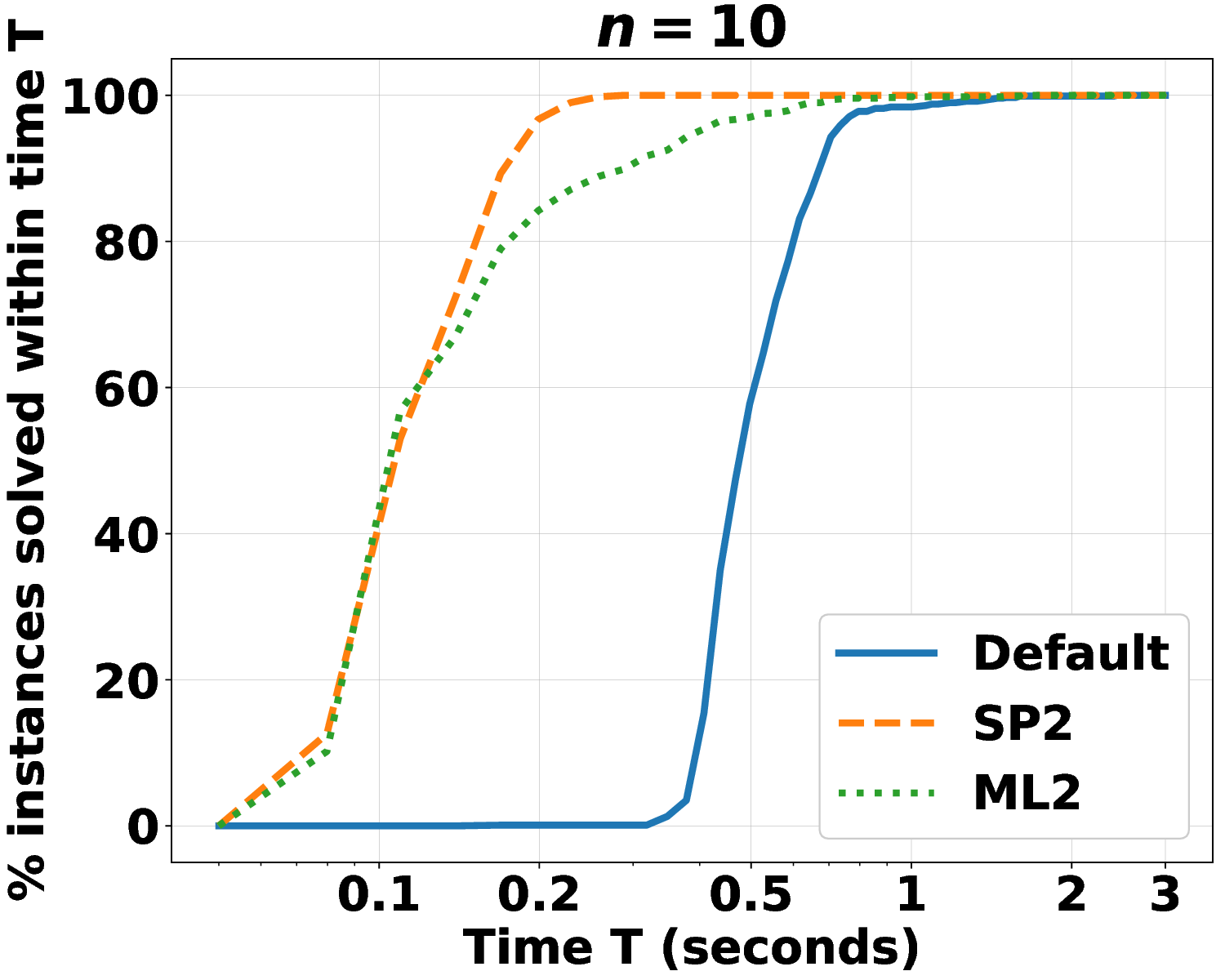}
\end{subfigure}%
\hfill
\begin{subfigure}[b]{0.32\textwidth}
\includegraphics[width=\textwidth]{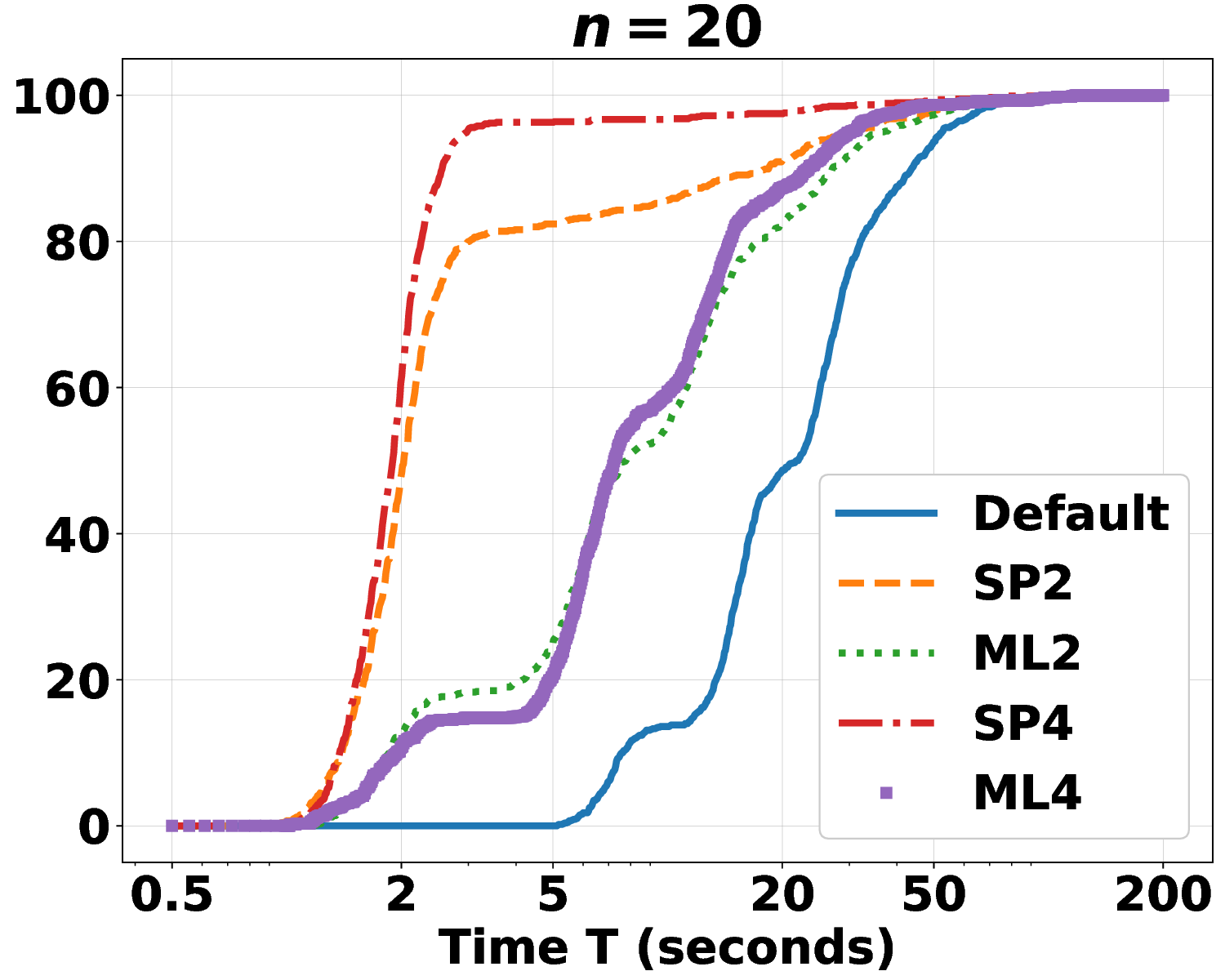}
\end{subfigure}%
\hfill
\begin{subfigure}[b]{0.32\textwidth}
\includegraphics[width=\textwidth]{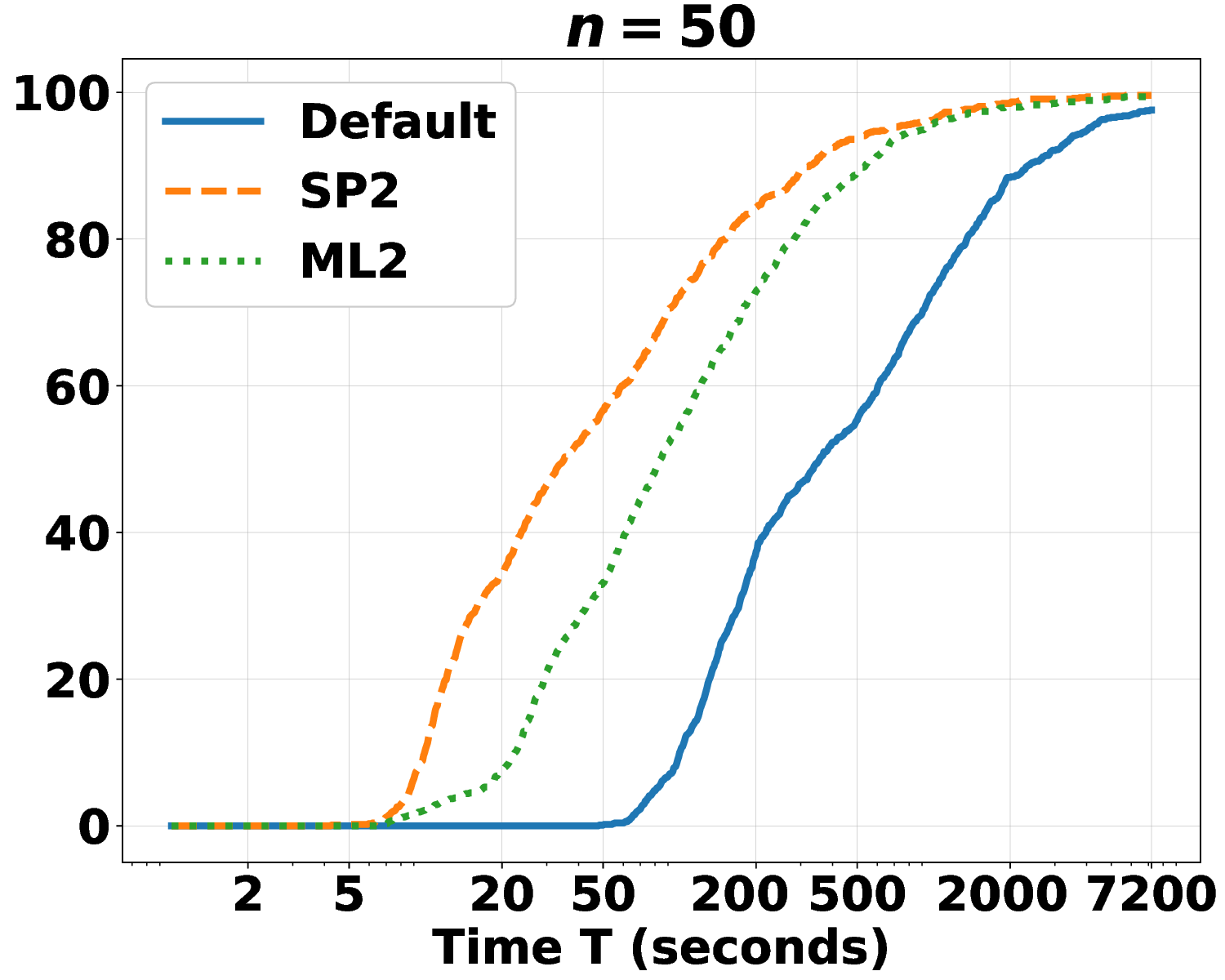}
\end{subfigure}\\[0.1in]
\begin{subfigure}[b]{0.32\textwidth}
\includegraphics[width=\textwidth]{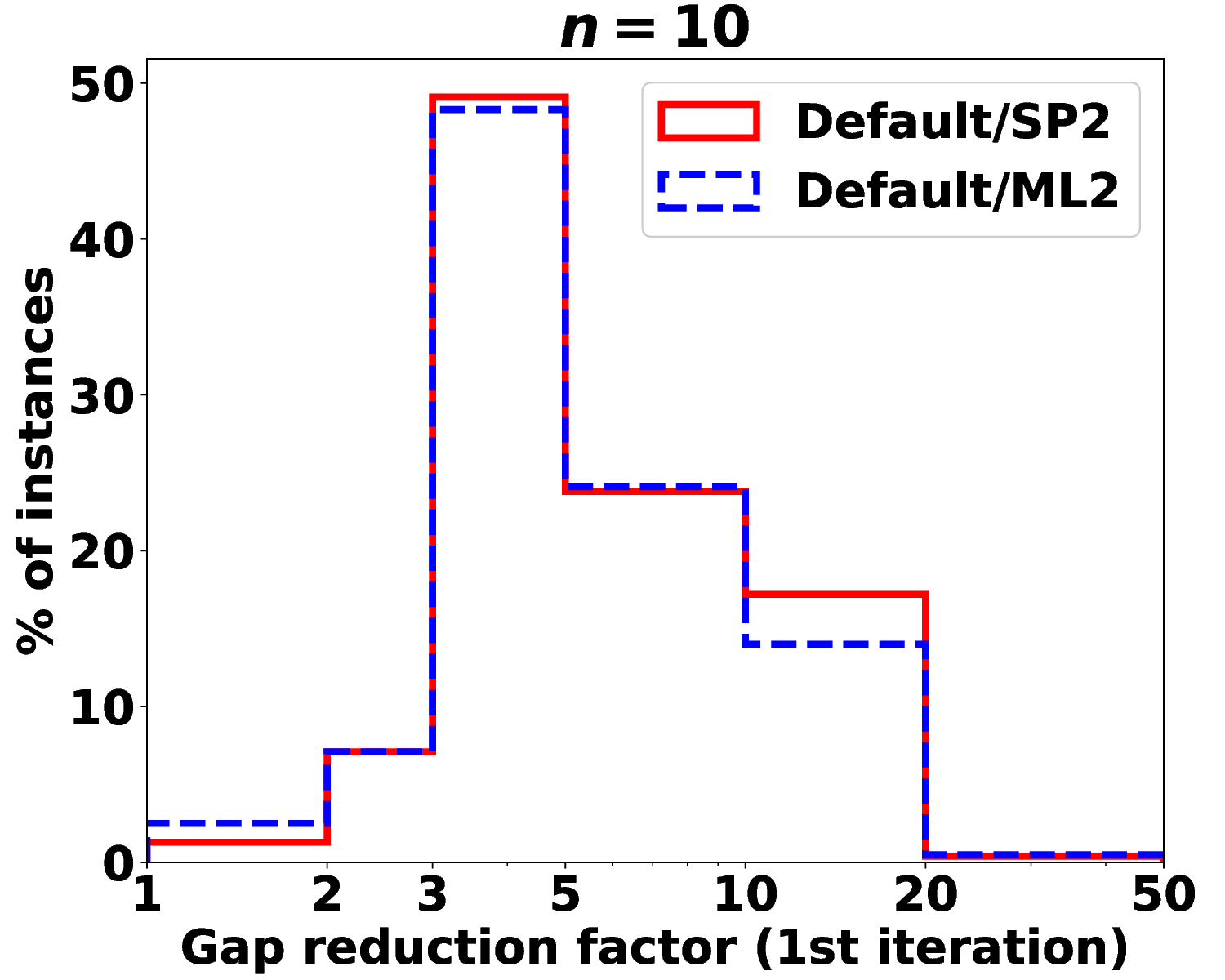}
\end{subfigure}%
\hfill
\begin{subfigure}[b]{0.32\textwidth}
\includegraphics[width=\textwidth]{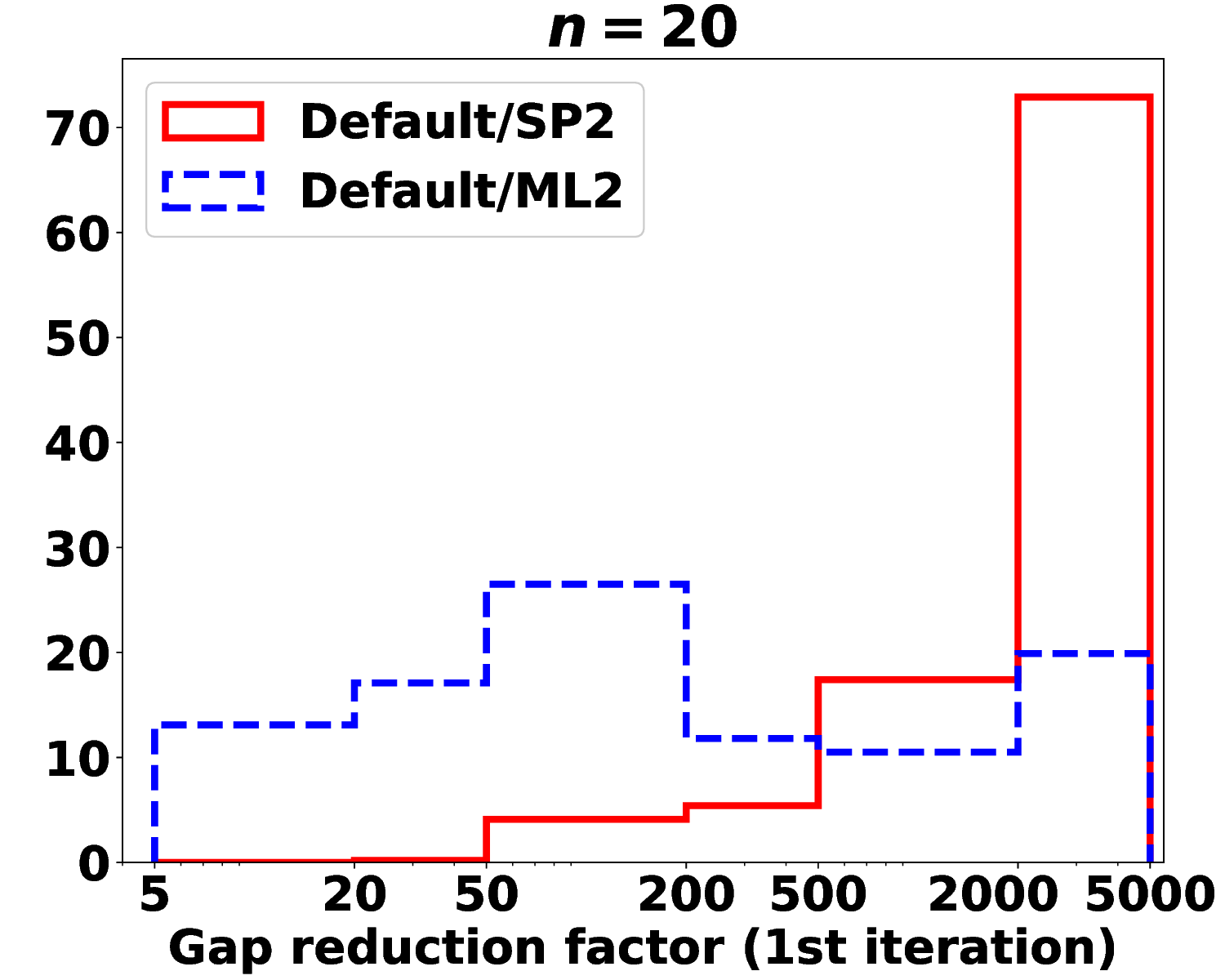}
\end{subfigure}%
\hfill
\begin{subfigure}[b]{0.32\textwidth}
\includegraphics[width=\textwidth]{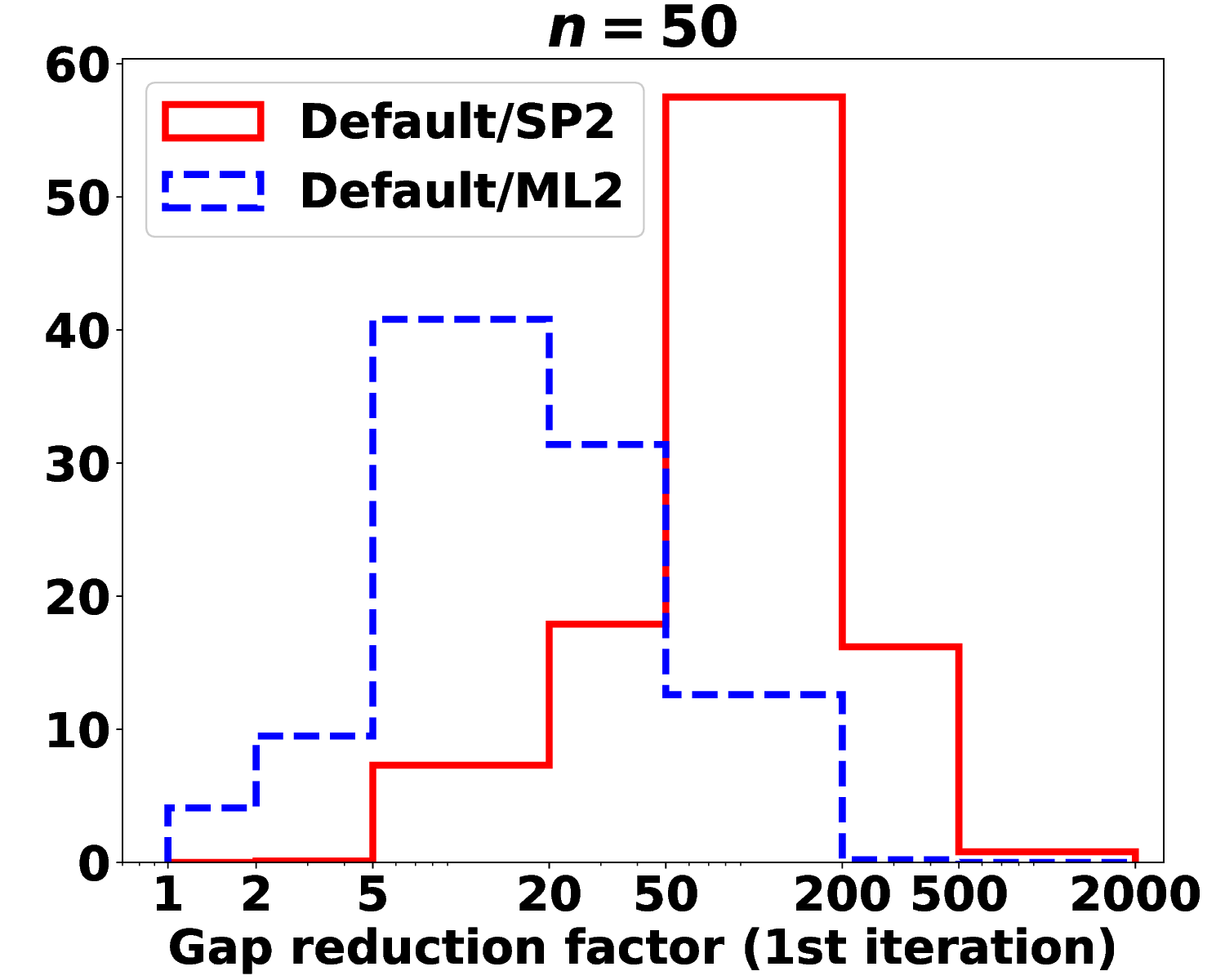}
\end{subfigure}
\end{figure}

\begin{figure}
\centering
\caption{
Results for the random QCQP instances. 
{The times for Alpine+SP2 and Alpine+SP4 do not include the time for running Algorithm~\ref{alg:enhancements} to determine strong partitioning points.}
Top row: solution profiles indicating the percentage of instances solved by the different methods within time T seconds (higher is better). 
Bottom row: histograms of the ratios of the effective optimality gaps~\eqref{eqn:effective_gap} of default Alpine with Alpine+SP2 and with Alpine+ML2 after one iteration (larger is better).}
\label{fig:plots_qcqp}

\vspace*{0.1in}
\begin{subfigure}[b]{0.32\textwidth}
\includegraphics[width=\textwidth]{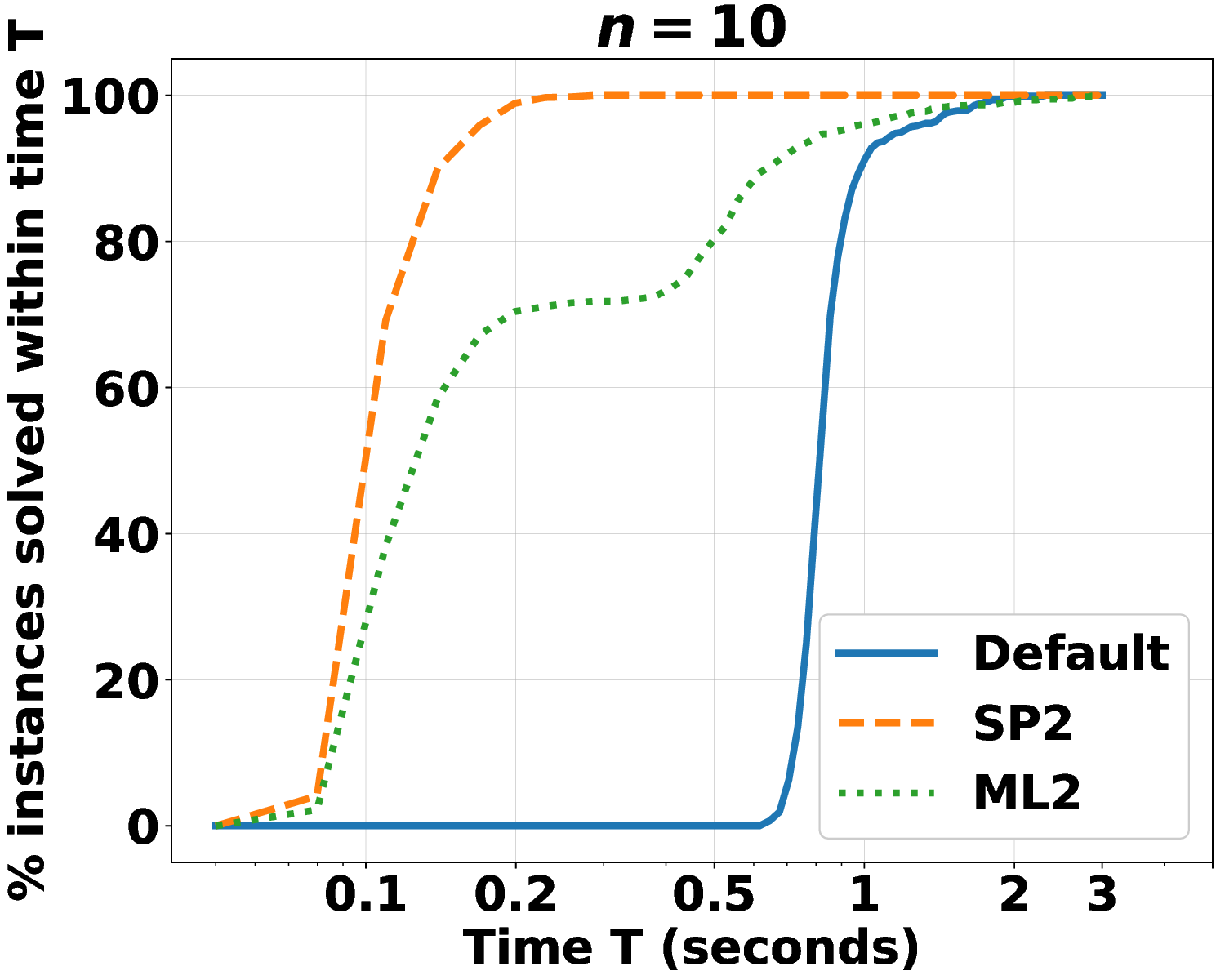}
\end{subfigure}%
\hfill
\begin{subfigure}[b]{0.32\textwidth}
\includegraphics[width=\textwidth]{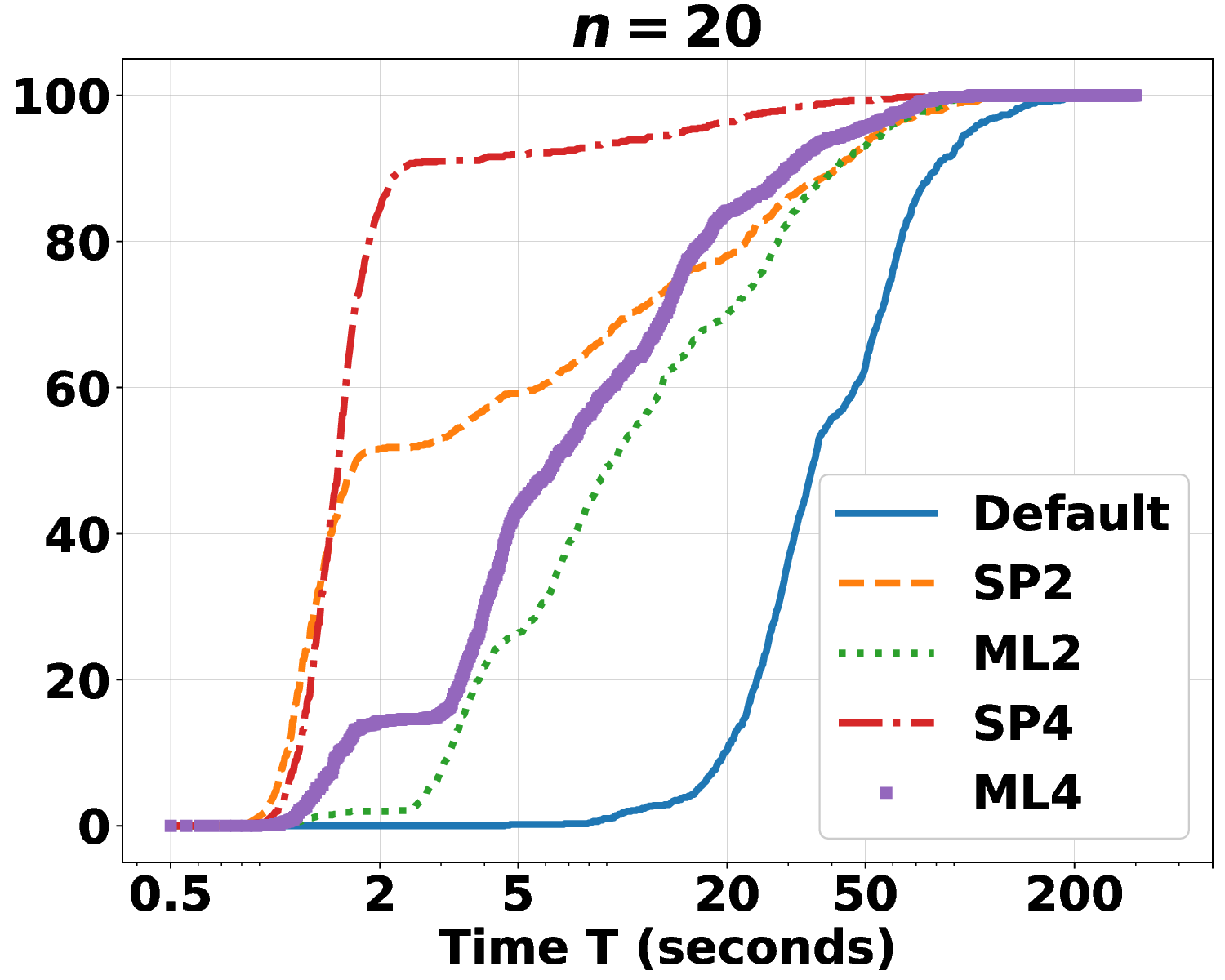}
\end{subfigure}%
\hfill
\begin{subfigure}[b]{0.32\textwidth}
\includegraphics[width=\textwidth]{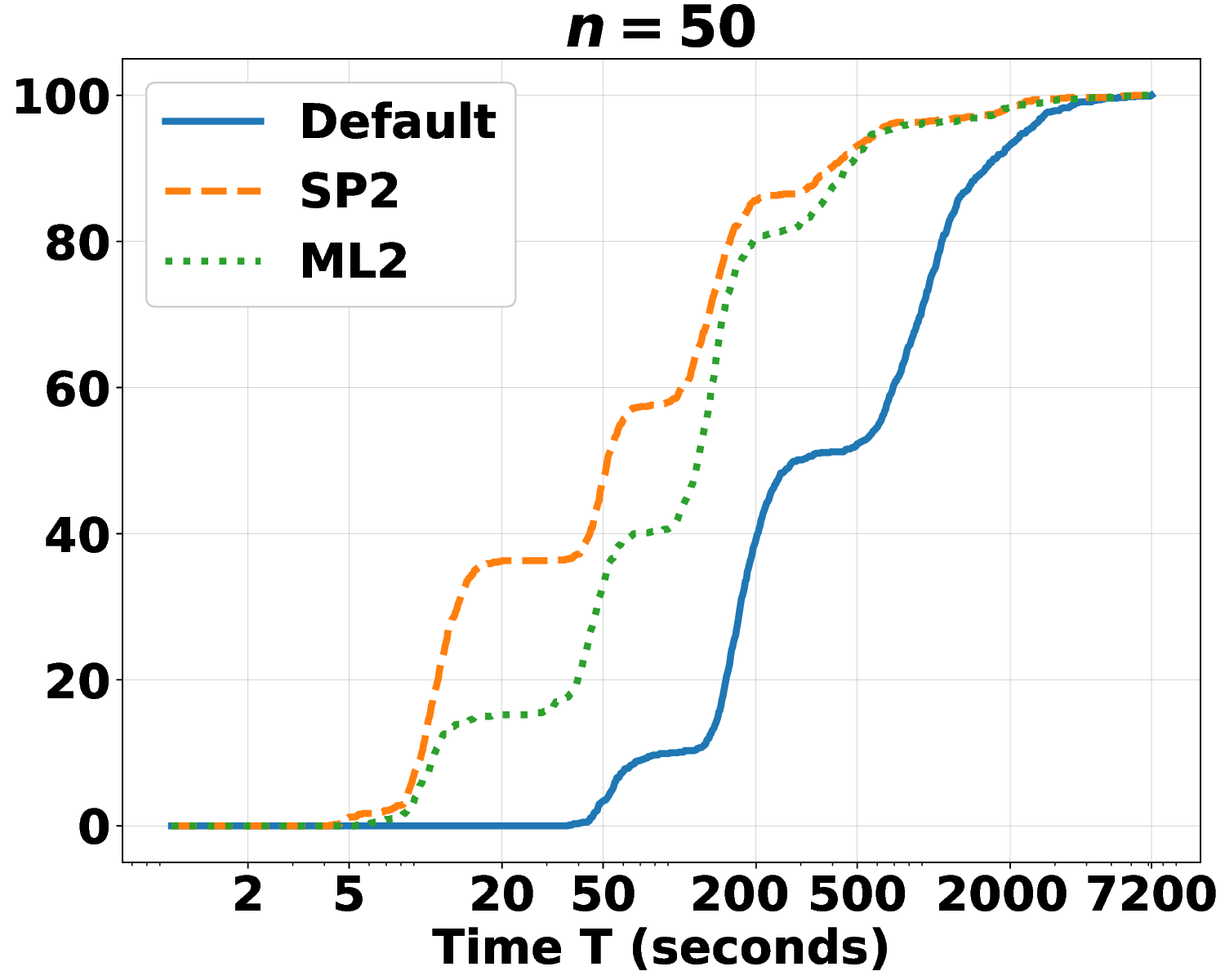}
\end{subfigure}\\[0.1in]
\begin{subfigure}[b]{0.32\textwidth}
\includegraphics[width=\textwidth]{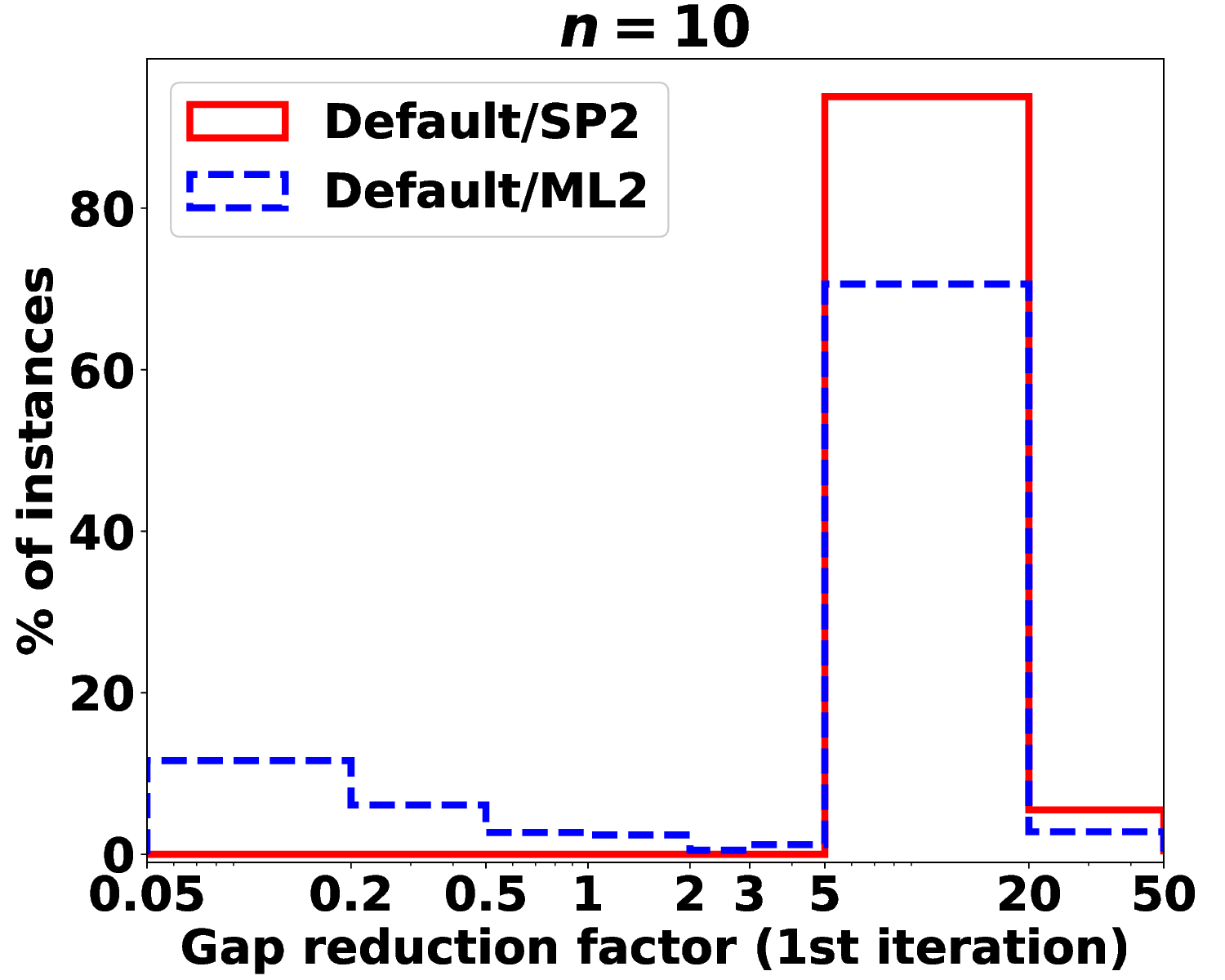}
\end{subfigure}%
\hfill
\begin{subfigure}[b]{0.32\textwidth}
\includegraphics[width=\textwidth]{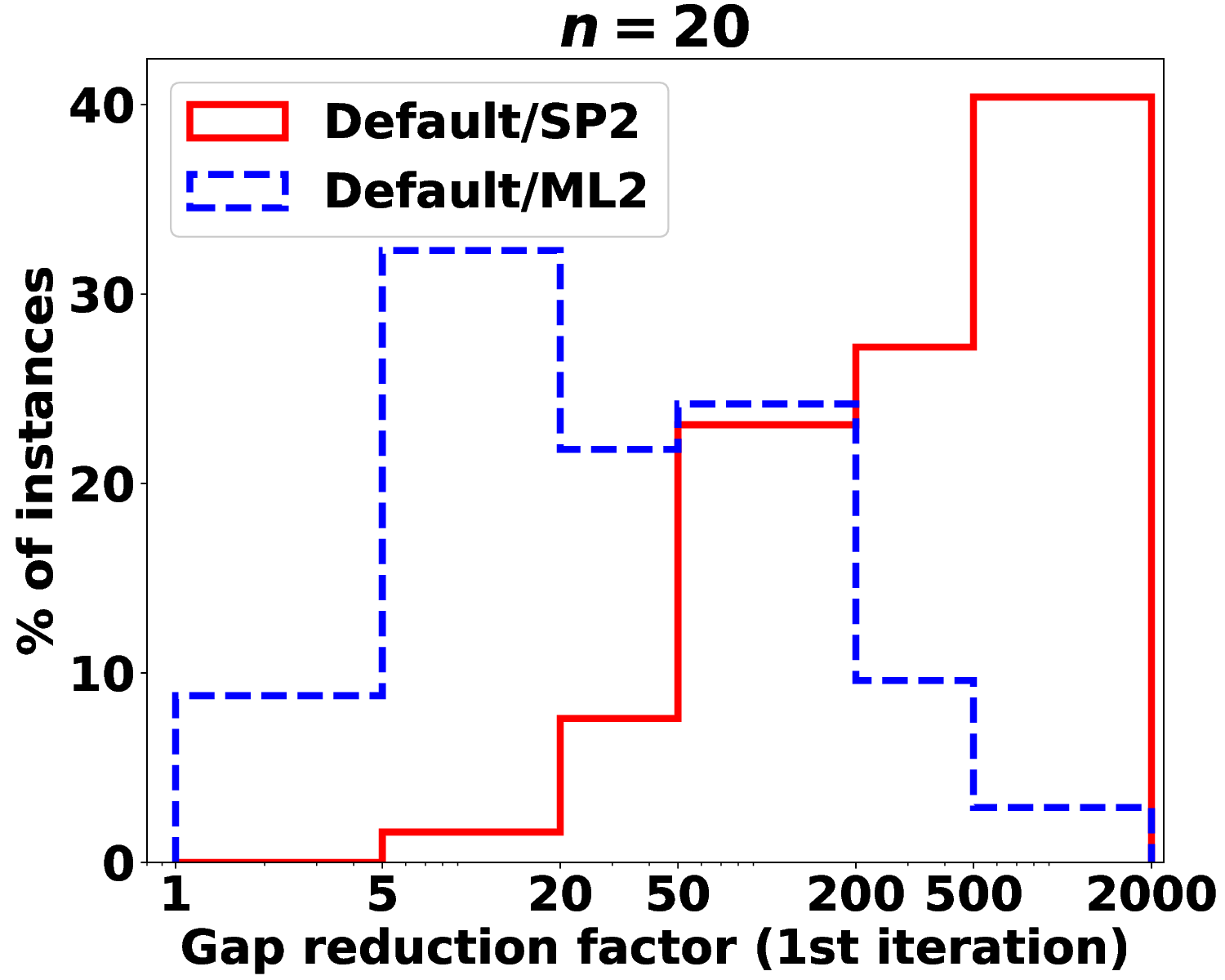}
\end{subfigure}%
\hfill
\begin{subfigure}[b]{0.32\textwidth}
\includegraphics[width=\textwidth]{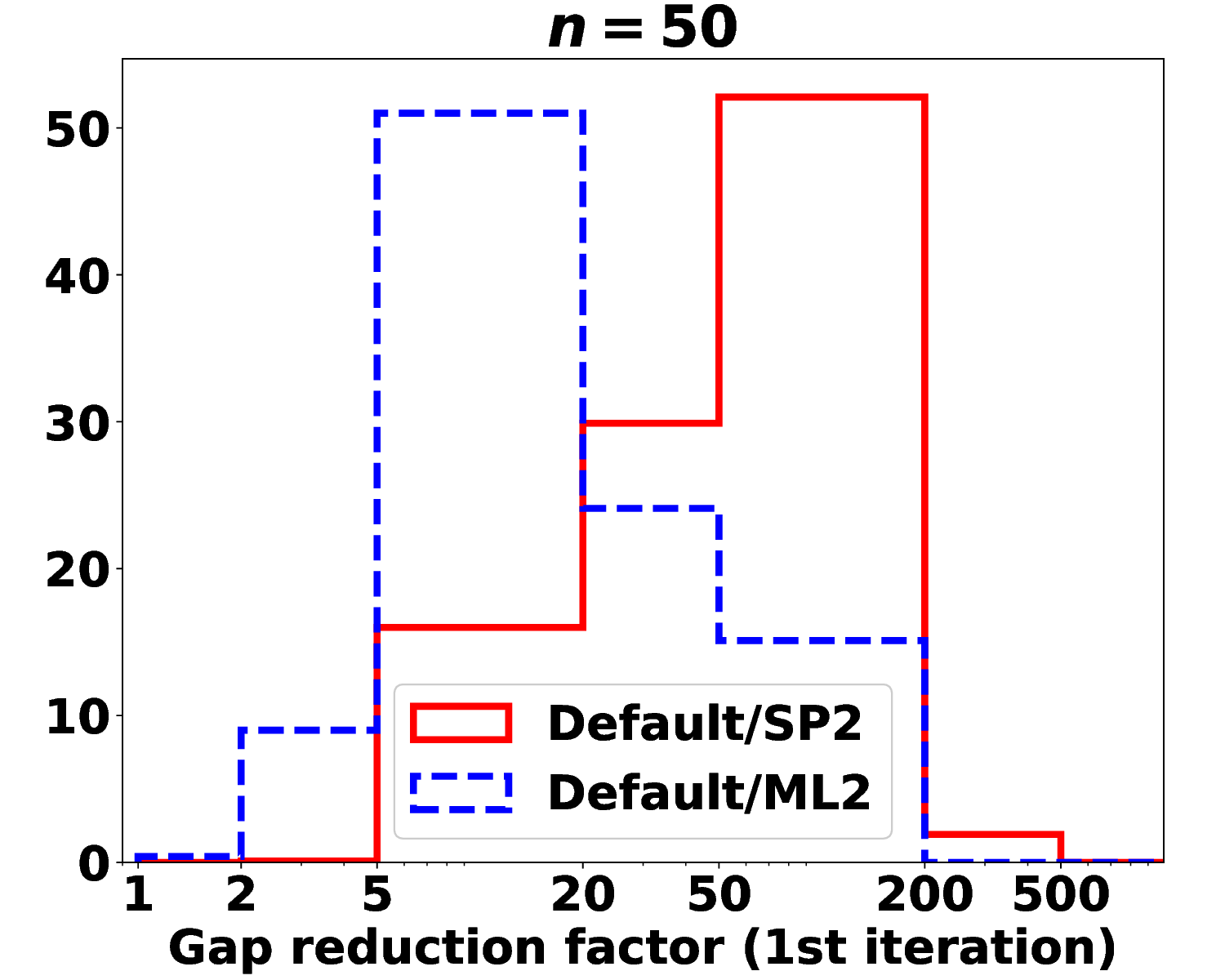}
\end{subfigure}
\end{figure}

\begin{figure}
\centering
\caption{Results for the random pooling instances.
{The times for Alpine+SP2 do not include the time for running Algorithm~\ref{alg:enhancements} to determine strong partitioning points.}
Left plot: solution profile indicating the percentage of instances solved by the different methods within time T seconds (higher is better). Right plot: histograms of the ratios of the effective optimality gaps~\eqref{eqn:effective_gap} of default Alpine with Alpine+SP2 and with Alpine+ML2 after one iteration (larger is better).}
\label{fig:plots_pooling}

\vspace*{0.05in}
\begin{subfigure}[b]{0.48\textwidth}
\centering
\includegraphics[width=0.66\textwidth,right]{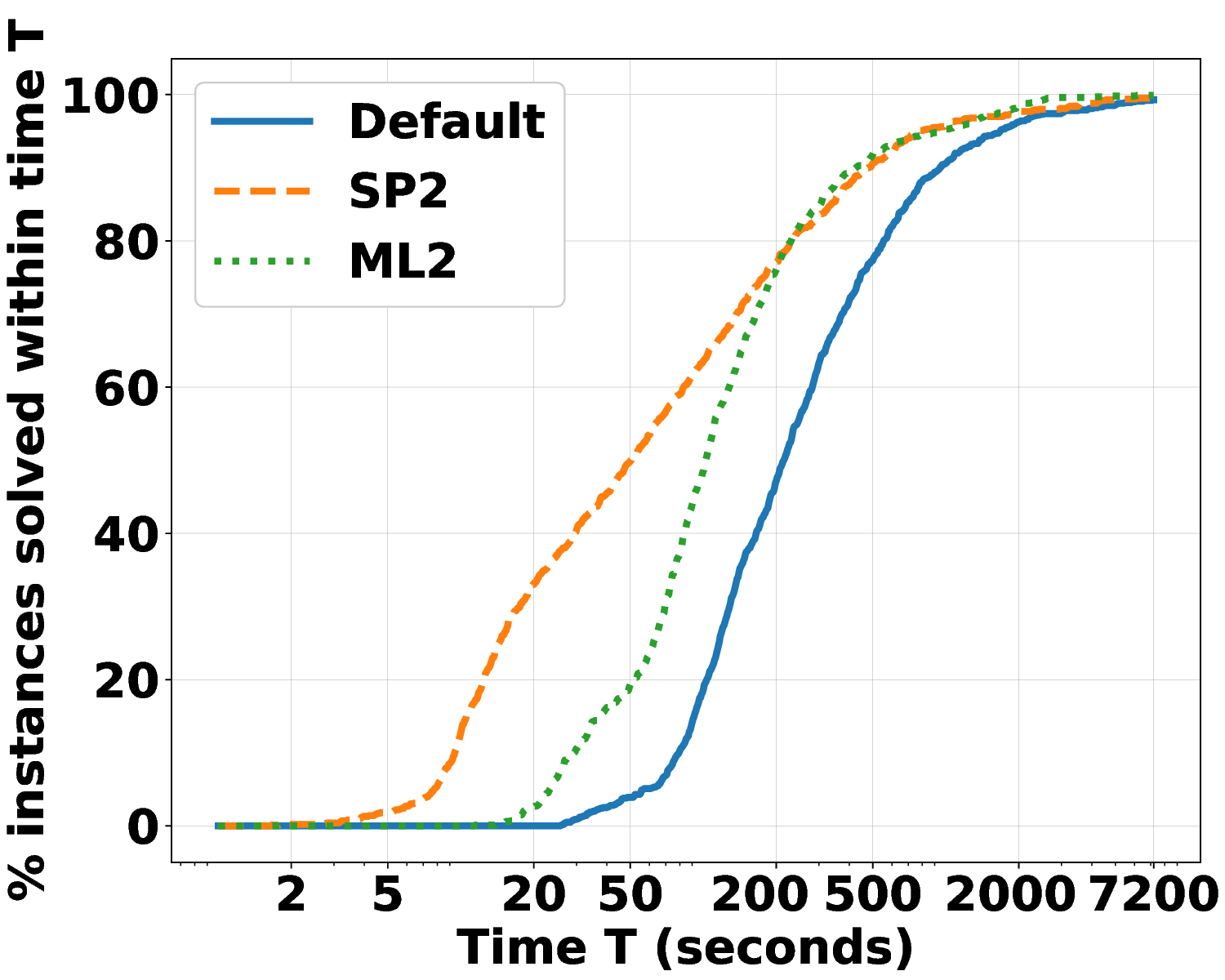}
\end{subfigure}%
\hfill
\begin{subfigure}[b]{0.48\textwidth}
\centering
\includegraphics[width=0.66\textwidth,left]{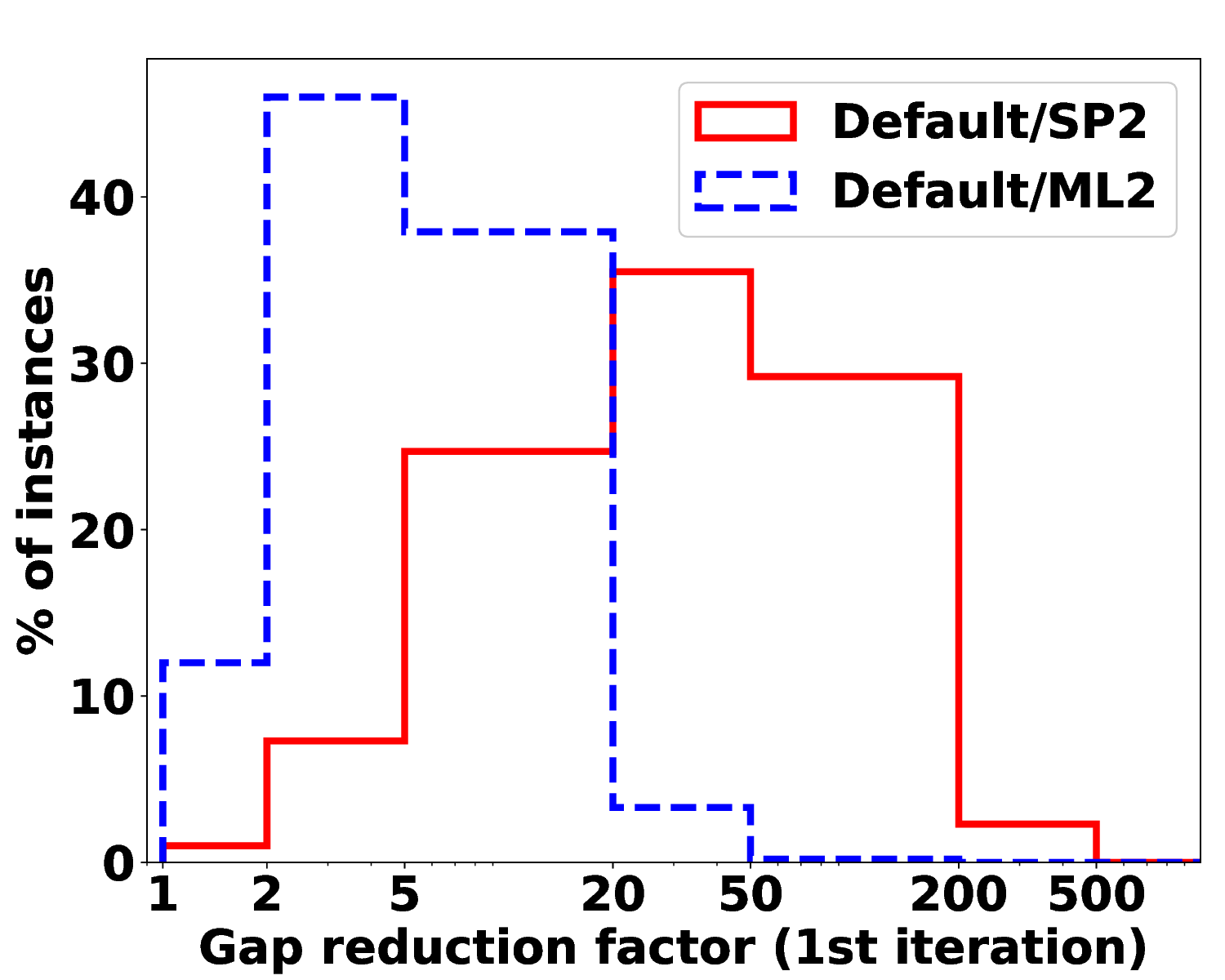}
\end{subfigure}
\end{figure}

\textbf{Bilinear instances.}
Table~\ref{tab:alpine_times} indicates that Alpine+SP2 reduces the shifted GM of default Alpine's solution time by factors of $4.5$, $5.1$, and $7.7$ for the \mbox{$n = 10$}, $n = 20$, and $n = 50$ families.
Alpine+ML2 offers a moderate approximation of Alpine+SP2, reducing the shifted GM of default Alpine's solution time by factors of $3.5$, $2.1$, and $4$, respectively, for $n = 10$, $n = 20$, and $n = 50$.
For the $n = 20$ instances, Alpine+SP4 and Alpine+ML4 reduce the shifted GM of default Alpine's solution time by factors of $9$ and $2.3$.

Table~7 in the supplemental material shows that Alpine+SP2 achieves at least $5\times$ speedup over default Alpine on $41.3\%$ of the $n = 10$ instances, and at least $10\times$ speedup on $39.9\%$ and $46.1\%$ of the $n = 20$ and $n = 50$ instances.
Alpine+ML2 achieves at least $5\times$ speedup on $40.1\%$, $22.2\%$, and $45.2\%$ of the $n = 10$, $n = 20$, and $n = 50$ instances.
Alpine+SP2 achieves a maximum speedup of $15\times$, $49\times$, and $685\times$ on the $n = 10$, $n = 20$, and $n = 50$ instances, while Alpine+ML2 achieves a maximum speedup of $13\times$, $38\times$, and $197\times$ on these instances.

{It is important to note that the times for Alpine+SP2 and Alpine+SP4 do not include the time required to run Algorithm~\ref{alg:enhancements} to determine strong partitioning points, making these solution times unachievable in practice.
Instead, these times reflect the potential performance that could be realized by an efficient ML model perfectly imitating the strong partitioning strategy. 
These results underscore the significant potential of strong partitioning as an expert strategy for accelerating the global optimization of nonconvex QCQPs.}

Table~\ref{tab:alpine_gaps} shows that Alpine+SP2 reduces the GM of default Alpine's effective optimality gap~\eqref{eqn:effective_gap} after the first iteration by factors of $5.5$, $2200$, and $80$, respectively, for the $n = 10$, $n = 20$, and $n = 50$ instances.
Alpine+ML2 also reduces the GM of default Alpine's gap after the first iteration by factors of $4.6$, $180$, and $15$ for $n = 10$, $n = 20$, and $n = 50$.
Notably, Alpine+SP2 closes the gap in the first iteration for $100\%$, $82.3\%$, and $46\%$ of the $n = 10$, $n = 20$, and $n = 50$ instances, while default Alpine is able to close the gap in the first iteration for at most $0.1\%$ of the instances across these families, highlighting the effectiveness of strong partitioning.
Table~\ref{tab:alpine_times} also shows that Alpine+SP2 and Alpine+ML2 terminate with smaller average gaps on the $n = 50$ instances where they time out, compared to default Alpine.

\begin{table}[t]
\centering
\caption{
Statistics on the effective optimality gap~\eqref{eqn:effective_gap} after Alpine's first iteration (note: the minimum possible value of the effective optimality gap is $10^{-4}$, the target gap). 
Columns show the geometric mean, median, minimum, and maximum effective gaps over 1000 instances. 
{The numbers in these columns have been multiplied by $10^4$ for ease of readability.}
The last column indicates the percentage of instances for which each method achieves the minimum possible effective optimality gap of $10^{-4}$ after Alpine's first iteration.
}
\begin{tabular}{ c | c | c c c c | c }
\hline
Problem Family & Solution Method & \multicolumn{4}{c|}{$10^4 \times$ Effective Optimality Gap} & \% Instances \\
& & GM & Median & Min & Max & Gap Closed \\ \hline
&  Alpine (default)  &  $5.5$  &  $4.5$  &  $1$  &  $343$  &  $0.1$ \\
Bilinear, $n = 10$  &  Alpine+SP2  &  $1$  &  $1$  &  $1$  &  $1$  &  $100$ \\
&  Alpine+ML2  &  $1.2$  &  $1$  &  $1$  &  $439$  &  $88.3$ \\[0.1in]
&  Alpine (default)  &  $2869$  &  $3324$  &  $734$  &  $4805$  &  $0$ \\
&  Alpine+SP2  &  $1.3$  &  $1$  &  $1$  &  $61$  &  $82.3$ \\
Bilinear, $n = 20$  &  Alpine+ML2  &  $16$  &  $19$  &  $1$  &  $1421$  &  $18.9$ \\
&  Alpine+SP4  &  $1$  &  $1$  &  $1$  &  $4.7$  &  $96.0$ \\
&  Alpine+ML4  &  $22$  &  $36$  &  $1$  &  $994$  &  $14.5$ \\[0.1in]
&  Alpine (default)  &  $144$  &  $170$  &  $1$  &  $693$ & $0.1$  \\
Bilinear, $n = 50$ &  Alpine+SP2  &  $1.7$  &  $1.2$  &  $1$  &  $5371$ & $46.0$ \\
&  Alpine+ML2  &  $9.5$  &  $9.4$  &  $1$  &  $4941$ & $5.6$ \\[0.1in]
&  Alpine (default)  &  $13$  &  $12$  &  $7.5$  &  $187$  &  $0$ \\
QCQP, $n = 10$  &  Alpine+SP2  &  $1$  &  $1$  &  $1$  &  $1$  &  $100$ \\
&  Alpine+ML2  &  $3$  &  $1$  &  $1$  &  $1291$  &  $71.8$ \\[0.1in]
&  Alpine (default)  &  $627$  &  $784$  &  $30$  &  $2064$  &  $0$ \\
&  Alpine+SP2  &  $2.1$  &  $1$  &  $1$  &  $66$  &  $52.2$ \\
QCQP, $n = 20$  &  Alpine+ML2  &  $20$  &  $25$  &  $1$  &  $578$  &  $2.0$ \\
&  Alpine+SP4  &  $1.1$  &  $1$  &  $1$  &  $36$  &  $92.6$ \\
&  Alpine+ML4  &  $15$  &  $17$  &  $1$  &  $668$  &  $14.7$ \\[0.1in]
&  Alpine (default)  &  $81$  &  $104$  &  $6.3$  &  $282$ & $0$  \\
QCQP, $n = 50$ &  Alpine+SP2  &  $1.6$  &  $1.3$  &  $1$  &  $10$ & $39.0$ \\
&  Alpine+ML2  &  $4.8$  &  $5.3$  &  $1$  &  $148$ & $14.9$ \\[0.1in]
&  Alpine (default)  &  $68$  &  $64$  &  $12$  &  $440$ & $0$  \\
Pooling &  Alpine+SP2  &  $2.4$  &  $1.4$  &  $1$  &  $31$ & $45.2$ \\
&  Alpine+ML2  &  $15$  &  $16$  &  $1$  &  $63$ & $0.1$ \\ \hline
\end{tabular}%
\label{tab:alpine_gaps}
\end{table}

\textbf{QCQP instances.}
Table~\ref{tab:alpine_times} shows that Alpine+SP2 reduces the shifted GM of default Alpine's solution time by factors of $8.4$, $5.2$, and $6.2$ for the $n = 10$, $n = 20$, and $n = 50$ families.
Alpine+ML2 provides a moderate approximation, reducing the shifted GM of default Alpine's solution time by factors of $3.1$, $3.1$, and $3.9$ for these instances.
For the $n = 20$ instances, Alpine+SP4 and Alpine+ML4 reduce the shifted GM of default Alpine's solution time by factors of $16.4$ and $4.3$, respectively.

Table~7 in the supplemental material indicates that Alpine+SP2 achieves at least a $10\times$ speedup over default Alpine on $20.5\%$, $54.6\%$, and $42.1\%$ of the $n = 10$, $n = 20$, and $n = 50$ instances.
Alpine+ML2 provides at least a $5\times$ speedup on $65.7\%$, $34.7\%$, and $42.7\%$ of these instances.
Alpine+SP2 achieves a maximum speedup of $22\times$, $87\times$, and $98\times$ on the $n = 10$, $n = 20$, and $n = 50$ instances, while Alpine+ML2 results in a maximum speedup of $19\times$, $56\times$, and $32\times$ on these instances.
{We reiterate that the times for Alpine+SP2 and Alpine+SP4 do not include the time required for running Algorithm~\ref{alg:enhancements}. 
These times reflect the potential performance that could be realized by an efficient ML model perfectly imitating the strong partitioning strategy.}

Table~\ref{tab:alpine_gaps} shows that Alpine+SP2 reduces the GM of default Alpine's effective gap~\eqref{eqn:effective_gap} after the first iteration by factors of $13$, $300$, and $50$ for the $n = 10$, $n = 20$, and $n = 50$ instances.
Alpine+ML2 reduces the GM of default Alpine's gap by factors of $4.3$, $31$, and $17$ for these instances.
Notably, Alpine+SP2 closes the gap in the first iteration for $100\%$, $52.2\%$, and $39\%$ of the $n = 10$, $n = 20$, and $n = 50$ instances, while default Alpine is unable to close the gap in the first iteration for any of these instances.

\textbf{Pooling instances.}
Table~\ref{tab:alpine_times} indicates that Alpine+SP2 and Alpine+ML2 reduce the shifted GM of default Alpine's solution time by factors of $3.6$ and $2.1$.
Table~7 in the supplemental material shows that Alpine+SP2 and Alpine+ML2 achieve at least a $5\times$ speedup over default Alpine on $45.7\%$ and $11.5\%$ of the instances.
Table~\ref{tab:alpine_gaps} reveals that Alpine+SP2 and Alpine+ML2 reduce the GM of default Alpine's effective optimality gap~\eqref{eqn:effective_gap} after the first iteration by factors of $28$ and $4.5$.
After the first iteration, Alpine+SP2 closes the effective optimality gap for $45.2\%$ of the instances, while default Alpine fails to close the gap for any of these instances.
Finally, Alpine+SP2 and Alpine+ML2 result in maximum speedups of $120\times$ and $41\times$.
{We reiterate that the times for Alpine+SP2 and Alpine+SP4 do not include the time required for running Algorithm~\ref{alg:enhancements}.}

\vspace*{0.05in}
\textbf{Summary.} Tables~\ref{tab:alpine_times} and~\ref{tab:alpine_gaps}, Figures~\ref{fig:plots_bilinear} to~\ref{fig:plots_pooling}, and Table~7 in the supplemental material clearly demonstrate the advantages of our SP and AdaBoost policies over Alpine's default partitioning policy {for solving homogeneous QCQPs with fixed structure}.
{Comparing Alpine+SP4 and Alpine+SP2 reveals that using SP to select more partitioning points results in more effective, non-myopic partition refinement, albeit at a significant increase in the time required to solve the max-min problem~\eqref{eqn:strong_part}.
Table~8 in the supplemental material indicates that solving the \mbox{max-min} problem~\eqref{eqn:strong_part} to local optimality can be time-consuming, making the direct use of the SP policy in Alpine impractical.
However, the performance metrics for Alpine+SP2 and Alpine+SP4 suggest significant potential improvements, reflecting the best-case scenario achievable by an efficient ML model that perfectly imitates the SP policy. 
This highlights the substantial potential for accelerating the global optimization of QCQPs through tailored ML approaches imitating strong partitioning.}

\section{Future work}
\label{sec:conclusion}

There are several promising avenues for future research.
One potential direction is to move beyond the current approach of prespecifying a fixed number of partitioning points per variable in the strong partitioning problem~\eqref{eqn:strong_part}. 
Instead, we could allocate different numbers of partitioning points to each variable based on their relative impact on the lower bound.
For example, consider a scenario where we aim to allocate up to $d+2$ partitioning points per partitioned variable and have a total budget $B \in [d \times \abs{\mathcal{NC}}]$ for unfixed partitioning points across all variables (excluding variable bounds). 
To optimally allocate and specify these partitioning points, we can solve the following max-min problem:
\[
\max\biggl\{\underline{v}(\bfP) \: : \: (\bfP,\bfZ) \in \mathcal{P}_d \times \{0,1\}^{\abs{\mathcal{NC}} \times d}, \:\: \sum_{i \in \mathcal{NC}} \sum_{j \in [d]} Z_{ij} = B, \:\: Z_{ij} = 0 \implies P_{i(j+1)} = 0, \: \forall (i,j) \in \mathcal{NC} \times [d] \biggr\},
\]
where $\bfP \in [0,1]^{\abs{\mathcal{NC}} \times (d+2)}$ represents the potential {\pp} and $\underline{v}$ is the optimal value function of problem~\eqref{eqn:piecewise_mccormick_mip_oa}.
Here, if $Z_{ij} = 0$, the partitioning point $P_{i(j+1)}$ is forced to $0$, rendering it redundant.
Unlike the original strong partitioning problem~\eqref{eqn:strong_part}, the outer-maximization problem above involves binary decision variables~$\bfZ$,  which necessitates new techniques for its solution.

{Additionally, future work could focus on designing more efficient methods for solving the max-min problem~\eqref{eqn:strong_part} to enhance the scalability of generating strong partitioning expert data.}
{Another exciting direction is designing tailored ML architectures capable of achieving similar speedups as the SP policy for non-homogeneous QCQPs with varying structure.}
Finally, it would be interesting to explore generalizations of the SP policy that optimally select a subset of variables for partitioning at each iteration of Algorithm~\ref{alg:partitioning_algorithm}.


\bibliographystyle{informs2014} 
\bibliography{main.bib} 

\begin{thebibliography}{62}
\providecommand{\natexlab}[1]{#1}
\providecommand{\url}[1]{\texttt{#1}}
\providecommand{\urlprefix}{URL }

\bibitem[{Achterberg(2007)}]{achterberg2007constraint}
Achterberg T (2007) \emph{Constraint integer programming}. Ph.D. thesis, TU Berlin.

\bibitem[{Al-Khayyal \protect\BIBand{} Falk(1983)}]{al1983jointly}
Al-Khayyal FA, Falk JE (1983) Jointly constrained biconvex programming. \emph{Mathematics of Operations Research} 8(2):273--286.

\bibitem[{Alvarez et~al.(2017)Alvarez, Louveaux, \protect\BIBand{} Wehenkel}]{alvarez2017machine}
Alvarez AM, Louveaux Q, Wehenkel L (2017) A machine learning-based approximation of strong branching. \emph{INFORMS Journal on Computing} 29(1):185--195.

\bibitem[{Baltean-Lugojan et~al.(2019)Baltean-Lugojan, Bonami, Misener, \protect\BIBand{} Tramontani}]{baltean2019scoring}
Baltean-Lugojan R, Bonami P, Misener R, Tramontani A (2019) Scoring positive semidefinite cutting planes for quadratic optimization via trained neural networks. \emph{Optimization Online} \urlprefix\url{https://optimization-online.org/2018/11/6943/}.

\bibitem[{Bao et~al.(2011)Bao, Sahinidis, \protect\BIBand{} Tawarmalani}]{bao2011semidefinite}
Bao X, Sahinidis NV, Tawarmalani M (2011) Semidefinite relaxations for quadratically constrained quadratic programming: A review and comparisons. \emph{Mathematical Programming} 129(1):129--157.

\bibitem[{Bengio et~al.(2021)Bengio, Lodi, \protect\BIBand{} Prouvost}]{bengio2021machine}
Bengio Y, Lodi A, Prouvost A (2021) Machine learning for combinatorial optimization: a methodological tour d’horizon. \emph{European Journal of Operational Research} 290(2):405--421.

\bibitem[{Bergamini et~al.(2008)Bergamini, Grossmann, Scenna, \protect\BIBand{} Aguirre}]{bergamini2008improved}
Bergamini ML, Grossmann I, Scenna N, Aguirre P (2008) An improved piecewise outer-approximation algorithm for the global optimization of {MINLP} models involving concave and bilinear terms. \emph{Computers \& Chemical Engineering} 32(3):477--493.

\bibitem[{Bestuzheva et~al.(2021)Bestuzheva, Besan{\c{c}}on, Chen et~al.}]{bestuzheva2021scip}
Bestuzheva K, Besan{\c{c}}on M, Chen WK, et~al. (2021) The {SCIP} optimization suite 8.0. \emph{arXiv preprint: 2112.08872} .

\bibitem[{Bienstock et~al.(2022)Bienstock, Escobar, Gentile, \protect\BIBand{} Liberti}]{bienstock2022mathematical}
Bienstock D, Escobar M, Gentile C, Liberti L (2022) Mathematical programming formulations for the alternating current optimal power flow problem. \emph{Annals of Operations Research} 1--39.

\bibitem[{Billionnet et~al.(2012)Billionnet, Elloumi, \protect\BIBand{} Lambert}]{billionnet2012extending}
Billionnet A, Elloumi S, Lambert A (2012) Extending the {QCR} method to general mixed-integer programs. \emph{Mathematical Programming} 131(1):381--401.

\bibitem[{Bonami et~al.(2018)Bonami, Lodi, \protect\BIBand{} Zarpellon}]{bonami2018learning}
Bonami P, Lodi A, Zarpellon G (2018) Learning a classification of mixed-integer quadratic programming problems. \emph{International Conference on the Integration of Constraint Programming, Artificial Intelligence, and Operations Research}, 595--604 (Springer).

\bibitem[{Breiman et~al.(2017)Breiman, Friedman, Olshen, \protect\BIBand{} Stone}]{breiman2017classification}
Breiman L, Friedman JH, Olshen RA, Stone CJ (2017) \emph{Classification and regression trees} (Routledge).

\bibitem[{Burer \protect\BIBand{} Vandenbussche(2008)}]{burer2008finite}
Burer S, Vandenbussche D (2008) A finite branch-and-bound algorithm for nonconvex quadratic programming via semidefinite relaxations. \emph{Mathematical Programming} 113(2):259--282.

\bibitem[{Cappart et~al.(2023)Cappart, Ch{\'e}telat, Khalil, Lodi, Morris, \protect\BIBand{} Veli{\v{c}}kovi{\'c}}]{cappart2021combinatorial}
Cappart Q, Ch{\'e}telat D, Khalil EB, Lodi A, Morris C, Veli{\v{c}}kovi{\'c} P (2023) Combinatorial optimization and reasoning with graph neural networks. \emph{Journal of Machine Learning Research} 24(130):1--61.

\bibitem[{Castro(2016)}]{castro2016normalized}
Castro PM (2016) Normalized multiparametric disaggregation: an efficient relaxation for mixed-integer bilinear problems. \emph{Journal of Global Optimization} 64(4):765--784.

\bibitem[{Cengil et~al.(2022)Cengil, Nagarajan, Bent, Eksioglu, \protect\BIBand{} Eksioglu}]{cengil2022learning}
Cengil F, Nagarajan H, Bent R, Eksioglu S, Eksioglu B (2022) Learning to accelerate globally optimal solutions to the {AC} optimal power flow problem. \emph{Electric Power Systems Research} 212:108275.

\bibitem[{De~Wolf \protect\BIBand{} Smeers(2021)}]{de2021generalized}
De~Wolf D, Smeers Y (2021) Generalized derivatives of the optimal value of a linear program with respect to matrix coefficients. \emph{European Journal of Operational Research} 291(2):491--496.

\bibitem[{Di~Liberto et~al.(2016)Di~Liberto, Kadioglu, Leo, \protect\BIBand{} Malitsky}]{di2016dash}
Di~Liberto G, Kadioglu S, Leo K, Malitsky Y (2016) {DASH}: Dynamic approach for switching heuristics. \emph{European Journal of Operational Research} 248(3):943--953.

\bibitem[{Drucker(1997)}]{drucker1997improving}
Drucker H (1997) Improving regressors using boosting techniques. \emph{ICML}, volume~97, 107--115 (Citeseer).

\bibitem[{Freund(1985)}]{freund1985postoptimal}
Freund RM (1985) Postoptimal analysis of a linear program under simultaneous changes in matrix coefficients. \emph{Mathematical Programming Essays in Honor of George B. Dantzig Part I}, 1--13 (Springer).

\bibitem[{Freund \protect\BIBand{} Schapire(1997)}]{freund1997decision}
Freund Y, Schapire RE (1997) A decision-theoretic generalization of on-line learning and an application to boosting. \emph{Journal of Computer and System Sciences} 55(1):119--139.

\bibitem[{Furini et~al.(2019)Furini, Traversi, Belotti, Frangioni, Gleixner, Gould, Liberti, Lodi, Misener, Mittelmann et~al.}]{furini2019qplib}
Furini F, Traversi E, Belotti P, Frangioni A, Gleixner A, Gould N, Liberti L, Lodi A, Misener R, Mittelmann H, et~al. (2019) {QPLIB}: a library of quadratic programming instances. \emph{Mathematical Programming Computation} 11:237--265.

\bibitem[{Gasse et~al.(2019)Gasse, Ch{\'e}telat, Ferroni, Charlin, \protect\BIBand{} Lodi}]{gasse2019exact}
Gasse M, Ch{\'e}telat D, Ferroni N, Charlin L, Lodi A (2019) Exact combinatorial optimization with graph convolutional neural networks. \emph{Advances in Neural Information Processing Systems} 32.

\bibitem[{Ghaddar et~al.(2023)Ghaddar, G{\'o}mez-Casares, Gonz{\'a}lez-D{\'\i}az, Gonz{\'a}lez-Rodr{\'\i}guez, Pateiro-L{\'o}pez, \protect\BIBand{} Rodr{\'\i}guez-Ballesteros}]{ghaddar2022learning}
Ghaddar B, G{\'o}mez-Casares I, Gonz{\'a}lez-D{\'\i}az J, Gonz{\'a}lez-Rodr{\'\i}guez B, Pateiro-L{\'o}pez B, Rodr{\'\i}guez-Ballesteros S (2023) Learning for spatial branching: An algorithm selection approach. \emph{INFORMS Journal on Computing} 35(5):1024--1043.

\bibitem[{Gonz{\'a}lez-Rodr{\'\i}guez et~al.(2022)Gonz{\'a}lez-Rodr{\'\i}guez, Alvite-Paz{\'o}, Alvite-Paz{\'o}, Ghaddar, \protect\BIBand{} Gonz{\'a}lez-D{\'\i}az}]{gonzalez2022polynomial}
Gonz{\'a}lez-Rodr{\'\i}guez B, Alvite-Paz{\'o} R, Alvite-Paz{\'o} S, Ghaddar B, Gonz{\'a}lez-D{\'\i}az J (2022) Polynomial optimization: Enhancing {RLT} relaxations with conic constraints. \emph{arXiv preprint arXiv:2208.05608} .

\bibitem[{Haverly(1978)}]{haverly1978studies}
Haverly CA (1978) Studies of the behavior of recursion for the pooling problem. \emph{ACM SIGMAP Bulletin} 25:19--28.

\bibitem[{He et~al.(2014)He, Daume~III, \protect\BIBand{} Eisner}]{he2014learning}
He H, Daume~III H, Eisner JM (2014) Learning to search in branch and bound algorithms. \emph{Advances in Neural Information Processing Systems} 27:3293--3301.

\bibitem[{Im(2018)}]{im2018sensitivity}
Im J (2018) \emph{Sensitivity Analysis and Robust Optimization: A Geometric Approach for the Special Case of Linear Optimization}. Master's thesis, University of Waterloo.

\bibitem[{Jungen et~al.(2023)Jungen, Zingler, Djelassi, \protect\BIBand{} Mitsos}]{jungen2023libdips}
Jungen D, Zingler A, Djelassi H, Mitsos A (2023) {libDIPS--Discretization-Based Semi-Infinite and Bilevel Programming Solvers}. \emph{Available on Optimization Online. URL: \url{https://optimization-online.org/?p=24914}} 1--38.

\bibitem[{Kannan(2018)}]{kannan2018algorithms}
Kannan R (2018) \emph{Algorithms, analysis and software for the global optimization of two-stage stochastic programs}. Ph.D. thesis, Massachusetts Institute of Technology.

\bibitem[{Kannan \protect\BIBand{} Barton(2017)}]{kannan2017cluster}
Kannan R, Barton PI (2017) The cluster problem in constrained global optimization. \emph{Journal of Global Optimization} 69(3):629--676.

\bibitem[{Kannan \protect\BIBand{} Barton(2018)}]{kannan2018convergence}
Kannan R, Barton PI (2018) Convergence-order analysis of branch-and-bound algorithms for constrained problems. \emph{Journal of Global Optimization} 71(4):753--813.

\bibitem[{Kannan et~al.(2025)Kannan, Nagarajan, \protect\BIBand{} Deka}]{kannan2025github}
Kannan R, Nagarajan H, Deka D (2025) Strong partitioning and a machine learning approximation for accelerating the global optimization of nonconvex {QCQPs}. \urlprefix\url{http://dx.doi.org/10.1287/ijoc.2023.0424.cd}, {A}vailable for download at https://github.com/INFORMSJoC/2023.0424.

\bibitem[{Khalil et~al.(2016)Khalil, Le~Bodic, Song, Nemhauser, \protect\BIBand{} Dilkina}]{khalil2016learning}
Khalil E, Le~Bodic P, Song L, Nemhauser G, Dilkina B (2016) Learning to branch in mixed integer programming. \emph{Proceedings of the AAAI Conference on Artificial Intelligence}, volume~30.

\bibitem[{Kim et~al.(2022)Kim, Richard, \protect\BIBand{} Tawarmalani}]{kim2022piecewise}
Kim J, Richard JPP, Tawarmalani M (2022) Piecewise polyhedral relaxations of multilinear optimization. \emph{Optimization Online} .

\bibitem[{Koopmans \protect\BIBand{} Beckmann(1957)}]{koopmans1957assignment}
Koopmans TC, Beckmann M (1957) Assignment problems and the location of economic activities. \emph{Econometrica: Journal of the Econometric Society} 53--76.

\bibitem[{Lee et~al.(2020)Lee, Ma, Yu, \protect\BIBand{} Dai}]{lee2020accelerating}
Lee M, Ma N, Yu G, Dai H (2020) Accelerating generalized {B}enders decomposition for wireless resource allocation. \emph{IEEE Transactions on Wireless Communications} 1233--1247.

\bibitem[{Li et~al.(2011)Li, Armagan, Tomasgard, \protect\BIBand{} Barton}]{li2011stochastic}
Li X, Armagan E, Tomasgard A, Barton PI (2011) Stochastic pooling problem for natural gas production network design and operation under uncertainty. \emph{AIChE Journal} 57(8):2120--2135.

\bibitem[{Liu et~al.(2019)Liu, Ploskas, \protect\BIBand{} Sahinidis}]{liu2019tuning}
Liu J, Ploskas N, Sahinidis NV (2019) Tuning {BARON} using derivative-free optimization algorithms. \emph{Journal of Global Optimization} 74(4):611--637.

\bibitem[{Lodi \protect\BIBand{} Zarpellon(2017)}]{lodi2017learning}
Lodi A, Zarpellon G (2017) On learning and branching: a survey. \emph{Top} 25(2):207--236.

\bibitem[{Lu et~al.(2018)Lu, Nagarajan, Bent, Eksioglu, \protect\BIBand{} Mason}]{lu2018tight}
Lu M, Nagarajan H, Bent R, Eksioglu SD, Mason SJ (2018) Tight piecewise convex relaxations for global optimization of optimal power flow. \emph{2018 Power Systems Computation Conference}, 1--7 (IEEE).

\bibitem[{Luedtke et~al.(2020)Luedtke, d'Ambrosio, Linderoth, \protect\BIBand{} Schweiger}]{luedtke2020strong}
Luedtke J, d'Ambrosio C, Linderoth J, Schweiger J (2020) Strong convex nonlinear relaxations of the pooling problem. \emph{SIAM Journal on Optimization} 30(2):1582--1609.

\bibitem[{M{\"a}kel{\"a}(2003)}]{makela2003multiobjective}
M{\"a}kel{\"a} MM (2003) Multiobjective proximal bundle method for nonconvex nonsmooth optimization: Fortran subroutine {MPBNGC 2.0}. \emph{Reports of the Department of Mathematical Information Technology, Series B. Scientific Computing, B} 13.

\bibitem[{McCormick(1976)}]{mccormick1976computability}
McCormick GP (1976) Computability of global solutions to factorable nonconvex programs: Part {I}—convex underestimating problems. \emph{Mathematical Programming} 10(1):147--175.

\bibitem[{Misener \protect\BIBand{} Floudas(2013)}]{misener2013glomiqo}
Misener R, Floudas CA (2013) {GloMIQO}: Global mixed-integer quadratic optimizer. \emph{Journal of Global Optimization} 57(1):3--50.

\bibitem[{Misener \protect\BIBand{} Floudas(2014)}]{misener2014antigone}
Misener R, Floudas CA (2014) {ANTIGONE}: algorithms for continuous/integer global optimization of nonlinear equations. \emph{Journal of Global Optimization} 59(2):503--526.

\bibitem[{Nagarajan et~al.(2019)Nagarajan, Lu, Wang, Bent, \protect\BIBand{} Sundar}]{nagarajan2019adaptive}
Nagarajan H, Lu M, Wang S, Bent R, Sundar K (2019) An adaptive, multivariate partitioning algorithm for global optimization of nonconvex programs. \emph{Journal of Global Optimization} 74(4):639--675.

\bibitem[{Nagarajan et~al.(2016)Nagarajan, Lu, Yamangil, \protect\BIBand{} Bent}]{nagarajan2016tightening}
Nagarajan H, Lu M, Yamangil E, Bent R (2016) Tightening {McC}ormick relaxations for nonlinear programs via dynamic multivariate partitioning. \emph{International Conference on Principles and Practice of Constraint Programming}, 369--387 (Springer).

\bibitem[{Nair et~al.(2020)Nair, Bartunov, Gimeno et~al.}]{nair2020solving}
Nair V, Bartunov S, Gimeno F, et~al. (2020) Solving mixed integer programs using neural networks. \emph{arXiv preprint: 2012.13349} .

\bibitem[{Nannicini et~al.(2011)Nannicini, Belotti, Lee, Linderoth, Margot, \protect\BIBand{} W{\"a}chter}]{nannicini2011probing}
Nannicini G, Belotti P, Lee J, Linderoth J, Margot F, W{\"a}chter A (2011) A probing algorithm for {MINLP} with failure prediction by {SVM}. \emph{International Conference on AI and OR Techniques in Constriant Programming for Combinatorial Optimization Problems}, 154--169 (Springer).

\bibitem[{Nohra et~al.(2021)Nohra, Raghunathan, \protect\BIBand{} Sahinidis}]{nohra2021spectral}
Nohra CJ, Raghunathan AU, Sahinidis N (2021) Spectral relaxations and branching strategies for global optimization of mixed-integer quadratic programs. \emph{SIAM Journal on Optimization} 31(1):142--171.

\bibitem[{Nohra et~al.(2022)Nohra, Raghunathan, \protect\BIBand{} Sahinidis}]{nohra2021sdp}
Nohra CJ, Raghunathan AU, Sahinidis NV (2022) {SDP}-quality bounds via convex quadratic relaxations for global optimization of mixed-integer quadratic programs. \emph{Mathematical Programming} 196(1):203--233.

\bibitem[{Pardalos \protect\BIBand{} Vavasis(1991)}]{pardalos1991quadratic}
Pardalos PM, Vavasis SA (1991) Quadratic programming with one negative eigenvalue is {NP}-hard. \emph{Journal of Global Optimization} 1(1):15--22.

\bibitem[{Pedregosa et~al.(2011)Pedregosa, Varoquaux, Gramfort, Michel, Thirion, $\ldots$, \protect\BIBand{} Duchesnay}]{scikit-learn}
Pedregosa F, Varoquaux G, Gramfort A, Michel V, Thirion B, $\ldots$, Duchesnay E (2011) Scikit-learn: Machine learning in {P}ython. \emph{Journal of Machine Learning Research} 12:2825--2830.

\bibitem[{Sahinidis(1996)}]{sahinidis1996baron}
Sahinidis NV (1996) {BARON}: A general purpose global optimization software package. \emph{Journal of Global Optimization} 8(2):201--205.

\bibitem[{Saif et~al.(2008)Saif, Elkamel, \protect\BIBand{} Pritzker}]{saif2008global}
Saif Y, Elkamel A, Pritzker M (2008) Global optimization of reverse osmosis network for wastewater treatment and minimization. \emph{Industrial \& Engineering Chemistry Research} 47(9):3060--3070.

\bibitem[{Sherali \protect\BIBand{} Adams(2013)}]{sherali2013reformulation}
Sherali HD, Adams WP (2013) \emph{A reformulation-linearization technique for solving discrete and continuous nonconvex problems}, volume~31 (Springer Science \& Business Media).

\bibitem[{Shor(1987)}]{shor1987quadratic}
Shor NZ (1987) Quadratic optimization problems. \emph{Soviet Journal of Computer and Systems Sciences} 25:1--11.

\bibitem[{Still(2018)}]{still2018lectures}
Still G (2018) Lectures on parametric optimization: An introduction. \emph{Optimization Online} .

\bibitem[{Sundar et~al.(2021)Sundar, Nagarajan, Linderoth, Wang, \protect\BIBand{} Bent}]{sundar2021piecewise}
Sundar K, Nagarajan H, Linderoth J, Wang S, Bent R (2021) Piecewise polyhedral formulations for a multilinear term. \emph{Operations Research Letters} 49(1):144--149.

\bibitem[{Wicaksono \protect\BIBand{} Karimi(2008)}]{wicaksono2008piecewise}
Wicaksono DS, Karimi IA (2008) Piecewise {MILP} under- and overestimators for global optimization of bilinear programs. \emph{AIChE Journal} 54(4):991--1008.

\bibitem[{Yang \protect\BIBand{} Barton(2016)}]{yang2016integrated}
Yang Y, Barton PI (2016) Integrated crude selection and refinery optimization under uncertainty. \emph{AIChE journal} 62(4):1038--1053.

\end{thebibliography}

\clearpage
\begin{APPENDICES}

\setSupplementStyle  

\vspace*{-0.5cm}
\section*{Online supplement}

We begin in Section~\ref{subsec:mpbngc} with a summary of the convergence guarantees of MPBNGC. 
Section~\ref{subsec:omitted_proofs} provides omitted proofs of results stated in the main text. 
In Section~\ref{subsubsec:test_instances}, we describe the methodology used to generate our random test instances. 
Section~\ref{subsec:benchmark} benchmarks the difficulty of these instances using the solvers BARON and Gurobi. 
Finally, Section~\ref{subsec:omitted_tables} presents additional tables from our numerical experiments.

\subsection{Convergence guarantees of MPBNGC}
\label{subsec:mpbngc}

We summarize the convergence guarantees of MPBNGC~\citep{makela2003multiobjective} for completeness.

\begin{definition}
Let $Z \subset \mathbb{R}^N$ be open.
A locally Lipschitz function $f: Z \to \mathbb{R}$ is said to be weakly semismooth if the directional derivative
$f'(\bfz,\bfd) = \underset{t \downarrow 0}{\lim} \: \frac{f(\bfz+t\bfd) - f(\bfz)}{t}$
exists for all $\bfz \in Z$ and $\bfd \in \R^N$, and we have $f'(\bfz,\bfd) = \underset{t \downarrow 0}{\lim} \: \tr{\boldsymbol\xi(\bfz+t\bfd)} \bfd$ for some $\boldsymbol\xi(\bfz+t\bfd) \in \partial f(\bfz+t\bfd)$.
\end{definition}

\begin{definition}
Let $f: \R^N \to \R$ and $\bfg : \R^N \to \R^M$ be locally Lipschitz continuous functions.
Consider the problem $\underset{\bfz: \bfg(\bfz) \leq \bfzero}{\min} f(\bfz)$.
A feasible point $\bfz^*$ is said to be substationary if there exist multipliers $\lambda \geq 0$ and $\boldsymbol\mu \in \mathbb{R}^M_+$, with $(\lambda, \boldsymbol\mu) \neq (0, \bfzero)$, such that
$0 \in \lambda \partial f(\bfz^*) + \sum_{j=1}^{M} \mu_j \partial g_j(\bfz^*)$ and $\mu_j g_j(\bfz^*) = 0, \: \forall j \in [M]$.
\end{definition}

\begin{theorem}
Suppose the value function $\underline{v}$ of~\eqref{eqn:piecewise_mccormick_mip_oa} is weakly semismooth.
Then MPBNGC either terminates finitely with a substationary point to~\eqref{eqn:strong_part}, or any accumulation point of a sequence of MPBNGC solutions is a substationary point to~\eqref{eqn:strong_part}.
\end{theorem}
\proof{Proof.}
See Theorem~9 of~\citet{makela2003multiobjective}. \Halmos
\endproof

\subsection{Omitted proofs}
\label{subsec:omitted_proofs}

We now provide omitted proofs of results in the main text.

\subsubsection{Proof of Lemma~\ref{lem:value_fn_active_part}}

\proof{Proof.}
Because $\bfY^*$ is the unique $\bfY$-solution to~\eqref{eqn:piecewise_mccormick_mip_oa} at $\bfP \in \mathcal{P}_d$, we have $v(\bfP,\bfY^*) < v(\bfP,\bfY)$, \mbox{$\forall \bfY \in \mathcal{Y} \backslash \{\bfY^*\}$}.
To demonstrate that $\underline{v}(\cdot) \equiv v(\cdot,\bfY^*)$ in a sufficiently small relative neighborhood of~$\bfP$, we show that the value function $v(\cdot,\bfY)$ is lower semicontinuous on~$\mathcal{P}_d$ for each $\bfY \in \mathcal{Y}$.
The stated result then follows, as $v(\cdot,\bfY^*)$ is assumed to be continuous at $\bfP$.

The set-valued mapping $\bfP \in \mathcal{P}_d \mapsto \{\bfz \geq \bfzero : \bfM(\bfP) \bfz = \bfd(\bfP), \: \bar{\bfM} \bfz = \bar{\bfB} \:\text{vec}(\bfY)\}$ is locally compact for each $\bfY \in \mathcal{Y}$ due to the continuity of the mappings $\bfM$ and $\bfd$ and the finite bounds that can be deduced for all variables in~\eqref{eqn:piecewise_mccormick_mip_oa}.
Hence, Lemma~5.3 of~\citet{still2018lectures} implies $v(\cdot,\bfY)$ is lower semicontinuous on~$\mathcal{P}_d$, for each $\bfY \in \mathcal{Y}$.
\Halmos
\endproof

\subsubsection{Proof of Theorem~\ref{thm:smooth_case}}

\proof{Proof.}
Lemma~\ref{lem:value_fn_active_part} implies that $\underline{v}(\cdot) \equiv v(\cdot,\bfY^*)$ in a sufficiently small relative neighborhood of $\bfP$, provided $v(\cdot,\bfY^*)$ is continuous at $\bfP$.
Theorem~4.3 of~\citet{still2018lectures} (cf.\ Theorem~1 of~\citet{freund1985postoptimal} and Proposition~4.1 of~\citet{de2021generalized}) and the fact that the functions $\bfM$ and $\bfd$ are continuously differentiable on $\mathcal{P}_d$ imply $v(\cdot,\bfY^*)$ is continuously differentiable at $\bfP$ and the stated equalities hold.
\Halmos
\endproof

\subsubsection{Proof of Theorem~\ref{thm:nonsmooth_case}}

\proof{Proof.}
Lemma~\ref{lem:value_fn_active_part} implies that $\underline{v}(\cdot) \equiv v(\cdot,\bfY^*)$ in a sufficiently small relative neighborhood of $\bfP$.
The stated equalities hold by mirroring the proof of Theorem~5.1 of~\citet{de2021generalized} and noting that the functions $\bfM$ and $\bfd$ are continuously differentiable on $\mathcal{P}_d$.
\Halmos
\endproof

\subsubsection{Proof of Lemma~\ref{lem:fullrank}}

\proof{Proof.}
Fix $\bfY \in \mathcal{\bfY}$.
Since $\bfP \in \text{ri}(\mathcal{P}_d)$, we have $0 = P_{i1} < P_{i2} < \dots < P_{i(d+1)} < P_{i(d+2)} = 1$, $\forall i \in \mathcal{NC}$.
We show that for each $(i,j) \in \mathcal{B}$ and $k \in \mathcal{Q}$, the equality constraints in~\eqref{eqn:lambda_form_bilinear_conv_comb}-\eqref{eqn:lambda_form_bilinear_bin} and~\eqref{eqn:lambda_form_quadratic_conv_comb}-\eqref{eqn:lambda_form_quadratic_bin} have full row rank, which implies that $\bfM(\bfP)$ has full row rank.
We ignore the inequality constraints because they are converted into equality constraints by adding unique slack variables.

We begin by focusing on the equality constraints in~\eqref{eqn:lambda_form_bilinear_conv_comb}-\eqref{eqn:lambda_form_bilinear_bin} involving the $\bfx$, $\bfW$, and $\bfLambda$ variables. 
Consider a fixed $(i,j) \in \mathcal{B}$ and any pair of indices $k, l \in [d+1]$.
We can rewrite these equality constraints as follows after suppressing the terms associated with most of the $\bfLambda^{ij}$ variables:
{
\[
\begin{pmatrix}
-1 & 0 & 0 & P_{il} & P_{il} & P_{i(l+1)} & P_{i(l+1)} & & \dots & \\
0 & -1 & 0 & P_{jk} & P_{j(k+1)} & P_{jk} & P_{j(k+1)} & & \ddots & \\
0 & 0 & -1 & P_{il}P_{jk} & P_{il}P_{j(k+1)} & P_{i(l+1)}P_{jk} & P_{i(l+1)}P_{j(k+1)} & & \dots & \\
0 & 0 & 0 & 1 & 1 & 1 & 1 &  & \dots & 
\end{pmatrix}
\begin{pmatrix}
x_i \\
x_j \\
W_{ij} \\
\Lambda^{ij}_{kl} \\
\Lambda^{ij}_{(k+1)l} \\
\Lambda^{ij}_{k(l+1)} \\
\Lambda^{ij}_{(k+1)(l+1)} \\
\vdots
\end{pmatrix}
= \begin{pmatrix}
0 \\
0 \\
0 \\
1
\end{pmatrix}.
\]
}%
We show that the fourth, fifth, sixth, and seventh columns of the matrix above are linearly independent whenever $P \in \text{ri}(\mathcal{P}_d)$.
Suppose, by way of contradiction, that these four columns are linearly dependent. Then, there exist scalars $\mu_1,\mu_2,\mu_3$, and $\mu_4$, not all zero, such that
\[
\mu_1 \begin{pmatrix}
P_{il} \\
P_{jk} \\
P_{il} P_{jk} \\
1
\end{pmatrix} +
\mu_2 \begin{pmatrix}
P_{il} \\
P_{j(k+1)} \\
P_{il} P_{j(k+1)} \\
1
\end{pmatrix} +
\mu_3 \begin{pmatrix}
P_{i(l+1)} \\
P_{jk} \\
P_{i(l+1)} P_{jk} \\
1
\end{pmatrix} +
\mu_4 \begin{pmatrix}
P_{i(l+1)} \\
P_{j(k+1)} \\
P_{i(l+1)} P_{j(k+1)} \\
1
\end{pmatrix} = 
\begin{pmatrix}
0 \\
0 \\
0 \\
0
\end{pmatrix}.
\]
This implies the following linear equations in $\mu_1,\mu_2,\mu_3$, and $\mu_4$:
\begin{subequations}
\begin{align}
(\mu_1 + \mu_2) P_{il} &= -(\mu_3 + \mu_4) P_{i(l+1)}, \label{eqn:int_1}\\
(\mu_1 + \mu_3) P_{jk} &= -(\mu_2 + \mu_4) P_{j(k+1)}, \label{eqn:int_2}\\
P_{il} \big( P_{jk} \mu_1 + P_{j(k+1)} \mu_2 \big) &= -P_{i(l+1)} \big( P_{jk} \mu_3 + P_{j(k+1)} \mu_4 \big), \label{eqn:int_3}\\
\mu_1 + \mu_2 &= -(\mu_3 + \mu_4). \label{eqn:int_4}
\end{align}
\end{subequations}
Since $P_{il} < P_{i(l+1)}$ and $P_{jk} < P_{j(k+1)}$, equations~\eqref{eqn:int_1} and~\eqref{eqn:int_4} imply that $\mu_1 + \mu_2 = 0$ and $\mu_3 + \mu_4 = 0$. Similarly, equations~\eqref{eqn:int_2} and~\eqref{eqn:int_4} imply that $\mu_1 + \mu_3 = 0$ and $\mu_2 + \mu_4 = 0$. Together, these equations imply $\mu_1 = \mu_4$, $\mu_2 = \mu_3$, and $\mu_1 = -\mu_2$. Combining these with equation~\eqref{eqn:int_3} leads to $\mu_1 = \mu_2 = \mu_3 = \mu_4 = 0$, which is a contradiction.

Next, we focus on the equality constraints in~\eqref{eqn:lambda_form_quadratic_conv_comb}-\eqref{eqn:lambda_form_quadratic_bin} involving the $\bfx$, $\bfW$, and $\bfLambda^k$ variables for a fixed $k \in \mathcal{Q}$.
Consider any index $l \in [d+1]$.
We can rewrite these equality constraints as follows after suppressing the terms associated with most of the $\bfLambda^k$ variables:
{
\[
\begin{pmatrix}
-1 & P_{kl} & P_{k(l+1)} & \dots \\
0 & 1 & 1 & \dots
\end{pmatrix}
\begin{pmatrix}
x_k \\
\Lambda^{k}_l \\
\Lambda^{k}_{l+1} \\
\vdots
\end{pmatrix}
= \begin{pmatrix}
0 \\
1
\end{pmatrix}.
\]
}%
The second and third columns of the matrix above are linearly independent whenever $P_{kl} < P_{k(l+1)}$.
\Halmos
\endproof

\subsection{Test instances}
\label{subsubsec:test_instances}

{We describe the process of generating {homogeneous} families of random QCQPs, {motivated by real-world applications that require solving the same underlying QCQP with slightly varying model parameters.}
Although QCQP instances are available from libraries such as QPLIB~\citep{furini2019qplib}, we are not aware of any libraries containing \textit{homogeneous} QCQP instances. 
To address this need, we generate random homogeneous QCQP instance families by building on instance generation schemes from the literature.}

\subsubsection{Random bilinear programs}
\label{subsubsubsec:random_bilinear}

Based on the numerical study in~\citet{bao2011semidefinite}, we consider the following family of parametric bilinear programs that are parametrized by $\bftheta \in [-1,1]^{d_{\theta}}$ {to reflect real-world applications where the same underlying QCQP is solved with slightly varying parameters:}
\begin{alignat*}{2}
v(\bftheta) := &\min_{\bfx \in [0,1]^{n}} \:\: && \bfx^{\textup{T}} \bfQ^0(\bftheta) \bfx + (\bfr^0(\bftheta))^{\textup{T}} \bfx \\
&\quad\: \text{s.t.} && \bfx^{\textup{T}} \bfQ^i(\bftheta) \bfx + (\bfr^i(\bftheta))^{\textup{T}} \bfx \leq b_i, \quad \forall i \in [m_I], \\
& && (\bfa^j)^{\textup{T}} \bfx = d_j, \quad \forall j \in [m_E],
\end{alignat*}
where vectors $\bfr^k(\bftheta) \in \R^n$, $\forall k \in \{0\} \cup [m_I]$, $\bfa^j \in \R^n$, $\forall j \in [m_E]$, $\bfb \in \R^{m_I}$, $\bfd \in \R^{m_E}$, and the matrices $\bfQ^k(\bftheta) \in \mathcal{S}^n$, $\forall k \in \{0\} \cup [m_I]$, are not assumed to be positive semidefinite.

We generate $1000$ instances each with $n \in \{10, 20, 50\}$ variables and $\abs{\mathcal{B}} = \min\{5n, \binom{n}{2}\}$ bilinear terms,
$\abs{\mathcal{Q}} = 0$ quadratic terms, $m_I = n$ bilinear inequalities, and $m_E = 0.2n$ linear equalities~\citep{bao2011semidefinite}.
Each family of instances with a fixed dimension $n$ is constructed to have the same set of $\abs{\mathcal{B}}$ bilinear terms.
We set the dimension $d_{\theta} = 3 \times (0.2m_I + 1)$ (an explanation for this choice is provided below).
The problem data is generated as follows~\citep[cf.][]{bao2011semidefinite}.
All entries of the vectors $\bfa^j$ and $\bfd$ are drawn i.i.d.\ from the uniform distribution $U(-1,1)$, while all entries of the vector $\bfb$ are drawn i.i.d.\ from $U(0,100)$.
The components of $\bftheta$ are drawn i.i.d.\ from $U(-1,1)$.
Each $\bfQ^k$ and $\bfr^k$, $k \in \{0,1,\dots,0.2m_I\}$, is of the form:
\[
\bfQ^k(\bftheta) = \bar{\bfQ}^k + \displaystyle\sum_{l=3k+1}^{3k+3} \theta_{l} \tilde{\bfQ}^{k,l-3k}, \qquad \bfr^k(\bftheta) = \bar{\bfr}^k + \displaystyle\sum_{l=3k+1}^{3k+3} \theta_{l} \tilde{\bfr}^{k,l-3k}.
\]
The nonzero entries of $\bar{\bfQ}^k \in \mathcal{S}^n$ are drawn i.i.d.\ from $U(-0.5,0.5)$, whereas the nonzero entries of $\bar{\bfr}^k \in \R^n$ are drawn i.i.d.\ from $U(-1,1)$. 
The matrices $\tilde{\bfQ}^{k,l} \in \mathcal{S}^n$, $k \in \{0\} \cup [0.2m_I]$, $l \in [3]$, are generated as follows.
For each $(i,j) \in \mathcal{B}$, $k \in \{0\} \cup [0.2m_I]$, and $l \in [3]$, we set $\tilde{Q}^{k,l}_{ij} := \Gamma^{k,l}_{ij} \bar{Q}^k_{ij}$, where $\Gamma^{k,l}_{ij}$ are drawn i.i.d.\ from $U(0,0.5)$. We set $\tilde{Q}^{k,l}_{ij} = \tilde{Q}^{k,l}_{ji}$ for each $(i,j) \in \mathcal{B}$ to ensure that $\tilde{\bfQ}^{k,l}$ is symmetric.
Next, for each $i \in [n]$, $k \in \{0\} \cup [0.2m_I]$, and $l \in [3]$, we set $\tilde{r}^{k,l}_i := \delta^{k,l}_i \bar{r}^k_i$, where $\delta^{k,l}_i$ are drawn i.i.d.\ from $U(0,0.5)$. 
{The data $\bar{\bfQ}^k$, $\tilde{\bfQ}^{k,l}$, $\bar{\bfr}^k$, and $\tilde{\bfr}^{k,l}$ are fixed across the $1000$ instances for each value of $n$.}
Since each $\tilde{\bfQ}^{k,l}$ and $\tilde{\bfr}^{k,l}$ is a different perturbation of $\bar{\bfQ}^k$ and $\bar{\bfr}^k$, the expansions of $\bfQ^k$ and $\bfr^k$ can be motivated using principal component analysis.
The nonzero entries of $\bfQ^k$ and $\bfr^k$, $k \in \{0.2m_I+1,\dots,m_I\}$, are identical across all $1000$ instances for each value of $n$, with each entry of $\bfQ^k$ drawn i.i.d.\ from $U(-0.5,0.5)$ and each entry of $\bfr^k$ drawn i.i.d.\ from $U(-1,1)$.
Finally, the constraint coefficients are re-scaled so that all entries of the vectors $\bfb$ and $\bfd$ equal one.
Note that for a fixed dimension~$n$, each instance is uniquely specified by the parameters~$\bftheta$ of dimension $d_{\theta} = 3 (0.2n + 1)$.
{Given that we assume that the forms of each $\bfQ^k(\bftheta)$ and $\bfr^k(\bftheta)$ and the data $\bar{\bfQ}^k$, $\tilde{\bfQ}^{k,l}$, $\bar{\bfr}^k$, and $\tilde{\bfr}^{k,l}$ are known, with only the parameter $\bftheta$ varying across different instances within each family, we can use $\bftheta$ as part of the features for each QCQP instance.}

\subsubsection{Random QCQPs with bilinear \textit{and} univariate quadratic terms}
\label{subsubsubsec:random_qcqps}

We also generate $1000$ random QCQPs with $\abs{\mathcal{B}} = \min\{5n, \binom{n}{2}\}$ bilinear terms 
and $\abs{\mathcal{Q}} = \lfloor0.25n\rfloor$ univariate quadratic terms for each of $n \in \{10,20,50\}$ variables. Once again, all instances for a fixed dimension $n$ comprise the same set of bilinear and univariate quadratic terms.
The coefficients of quadratic terms in the objective and constraints are generated similarly to the coefficients of bilinear terms in Section~\ref{subsubsubsec:random_bilinear}.
The rest of the model parameters and problem data (including $\bftheta$) are also generated similarly as in Section~\ref{subsubsubsec:random_bilinear}.

\subsubsection{The pooling problem}
\label{subsubsubsec:pooling_problem}

The pooling problem is a classic example of a bilinear program introduced by~\citet{haverly1978studies}.
It has several important applications,
including petroleum refining~\citep{kannan2018algorithms,yang2016integrated}, natural gas production~\citep{kannan2018algorithms,li2011stochastic}, and wastewater treatment~\citep{bergamini2008improved,misener2013glomiqo,saif2008global}.
Its goal is to blend inputs of differing qualities at intermediate pools to produce outputs that meet quality specifications while satisfying capacity constraints at inputs, pools, and outputs.
Solving the pooling problem is in general NP-hard.

We consider instances of the pooling problem with $45$ inputs, $15$ pools, $30$ outputs, and a single quality.
Each instance has $116$ input-output arcs, $71$ input-pool arcs, and $53$ pool-output arcs, yielding $572$ variables and $621$ constraints, including $360$ linear constraints and $261$ bilinear equations (with $124$ variables involved in bilinear terms).
We use the $pq$-formulation of the pooling problem from Section~2 of~\citet{luedtke2020strong}.
Unlike the random bilinear instances in Section~\ref{subsubsubsec:random_bilinear} where all of the original ``$\bfx$ variables'' participate in bilinear terms, only $124$ out of the $311$ original variables in the pooling model participate in bilinear terms.

We first generate a nominal instance using the ``random Haverly'' instance generation approach (see \url{https://github.com/poolinginstances/poolinginstances} for details) in~\cite{luedtke2020strong} that puts together $15$ perturbed copies of one of the~\citet{haverly1978studies} pooling instances and adds $150$ edges to it.
We modify the target output quality concentrations generated by~\citet{luedtke2020strong} to construct harder instances.
For each output $j$, we compute the minimum $c^{\text{min}}_j$ and maximum $c^{\text{max}}_j$ input concentrations of the quality over the subset of inputs from which there exists a path to output $j$.
We then specify the lower and upper bound on the quality concentration at output $j$ to be \mbox{$c^{\text{min}}_j + \alpha_j (c^{\text{max}}_j - c^{\text{min}}_j)$} and $c^{\text{min}}_j + \beta_j (c^{\text{max}}_j - c^{\text{min}}_j)$, respectively, where $\alpha_j \sim U(0.2,0.4)$ and $\beta_j \sim U(0.6,0.8)$ are generated independently.
We also rescale the capacities of the inputs, pools, and outputs and the costs of the arcs for better numerical performance.
Note that while all variables in the formulation are nonnegative, upper bounds on the variables are not necessarily equal to one after rescaling.
After constructing a nominal instance using the above procedure, we use it to generate $1000$ random pooling instances by randomly perturbing each input's quality concentration (parameters $\bftheta$ for this problem family) by up to $20\%$, uniformly and independently.

{
Each pooling instance in the family is a nonconvex bilinear program that can be expressed as:
\begin{alignat*}{2}
v(\bftheta) := &\min_{\bfx \in [\bfzero,\mathbf{u}]} \:\: && (\bfr^0)^{\textup{T}} \bfx \\
&\quad\: \text{s.t.} && \bfx^{\textup{T}} \bfQ^i(\bftheta) \bfx + (\bfr^i(\bftheta))^{\textup{T}} \bfx \leq 0, \quad \forall i \in [m_I], \\
& && \bfx^{\textup{T}} \bfQ^i \bfx + (\bfr^i)^{\textup{T}} \bfx = 0, \quad \forall i \in \{m_I+1,\dots,m_E\}, \\
& && (\bfa^j)^{\textup{T}} \bfx \leq d_j, \quad \forall j \in [m'_I], \\
& && (\bfa^j)^{\textup{T}} \bfx = 1, \quad \forall j \in \{m'_I+1,\dots,m'_E\},
\end{alignat*}
where $\mathbf{u} \in \mathbb{R}^{311}_+$, $m_I = 131$, $m_E = 124$, $m'_I = 90$, $m'_E = 15$, and the matrices $\bfQ^i$ and vectors $\bfr^i$, $\bfa^j$, and $\bfd$ are of appropriate dimensions. 
The parameters $\bftheta \in \R^{45}_+$ represent input qualities for this family of instances, with the coefficients $\bfQ^i(\bftheta)$ and $\bfr^i(\bftheta)$ being known affine functions of $\bftheta$ for each $i \in [m_I]$.
We refer the reader to formulation~(1) in~\citet{luedtke2020strong} for further details.
We assume that the decision-maker can measure and use the values of $\bftheta$ as features for each pooling instance.
}

\subsection{{Benchmarking test instances using state-of-the-art global optimization solvers}}
\label{subsec:benchmark}

To highlight the nontrivial nature of our nonconvex QCQP test instances, we solve them to global optimality using the state-of-the-art global optimization solvers BARON and Gurobi, evaluating them with the metrics detailed below.
{It is important to note that our goal is not to compare BARON or Gurobi with the different versions of Alpine, but rather to demonstrate that both our QCQP instances and the accelerations of Alpine achieved are nontrivial.}

\subsubsection{Computational setup}

We use Gurobi 9.1.2 via Gurobi.jl v0.11.3 and BARON 23.6.23 via BARON.jl v0.8.2, with the option of using CPLEX 22.1.0 as BARON's MILP solver. {(Note that BARON does not support Gurobi as an MILP solver.)}
Each BARON and Gurobi run is assigned a time limit of $2$ hours, with target relative and absolute optimality gaps of $10^{-4}$ and $10^{-9}$. 
(Alpine's definition of relative gap differs \textit{slightly} from those of BARON and Gurobi, as noted in Section~\ref{subsubsec:outline_experiments}.) 
All other settings in BARON and Gurobi were kept at default.

\subsubsection{{Evaluation metrics}}

We use the following metrics to evaluate the performance of BARON and Gurobi:
\begin{enumerate}[label=\roman*.]
\item \textbf{Solution times:} shifted geometric mean (GM), median, minimum, and maximum solution times.
\item \textbf{Number of iterations or B\&B nodes:} GM and median number of iterations or B\&B nodes explored.
\item \textbf{Number of instances solved:} the number of instances solved within the time limit and the GM of the residual optimality gap for the unsolved instances.
\end{enumerate}
The shifted geometric mean of a positive vector $t$ of solution times (in seconds) is defined as:
\[
\text{Shifted GM}(t) = \exp\Bigl(\frac{1}{N}\sum_{i=1}^{N} \ln(t_i + 10) \Bigr) - 10.
\]
The residual optimality gap for an unsolved instance is defined as:
\[
\textup{TLE Gap } = \frac{\text{UB} - \text{LB}}{10^{-6} + \abs{\text{UB}}},
\]
where $\text{UB}$ and $\text{LB}$ are the upper and lower bounds on the optimal value returned by the solver at termination.

\begin{table}
\centering
\caption{Statistics on BARON solution times,
including the shifted geometric mean, median, minimum, and maximum times over the subset of 1000 instances for which BARON does not hit the {2 hour} time limit.
{The sixth and seventh columns show the geometric mean and median of the total number of BARON iterations across the 1000 instances.}
The last two columns denote the number of instances for which BARON hits the time limit and the corresponding geometric mean of residual optimality gaps at termination, respectively.
}
\begin{tabular}{ r | c c c c | c c | c c }
\hline
Problem Family & \multicolumn{4}{c|}{BARON Solution Time (seconds)} & \multicolumn{2}{c|}{{BARON Iterations}} \\
 & Shifted GM & Median & Min & Max & {GM}  &  {Median} & \# TLE & TLE Gap (GM) \\ \hline
Bilinear, $n = 10$  &  0.2  &  0.2  &  0.1  &  0.4  &  1  &  1  &  0  & $\underline{\hspace{0.3cm}}$  \\
Bilinear, $n = 20$  &  3.8  &  3.9  &  1.3  &  7.6  &  9  &  9  &  0  & $\underline{\hspace{0.3cm}}$ \\
Bilinear, $n = 50$  &  312.4  &  281.4  &  57.7  &  5194.9  &  9934  &  9193  &  0  & $\underline{\hspace{0.3cm}}$ \\[0.1in]
QCQP, $n = 10$  &  0.4  &  0.4  &  0.1  &  0.7  &  1  &  1  &  0  & $\underline{\hspace{0.3cm}}$  \\
QCQP, $n = 20$  &  5.0  &  4.8  &  1.5  &  9.8  &  17  &  13  &  0  & $\underline{\hspace{0.3cm}}$ \\
QCQP, $n = 50$  &  1311.3  &  1741.6  &  16.3  &  7165.4 &  54691  &  73221  &  107  & $3.2 \times 10^{-2}$ \\[0.1in]
Pooling  &  553.5  &  647.2  &  17.0  &  7189.2  &  50954  &  263547  &  417  &  $2.4 \times 10^{-2}$ \\  \hline
\end{tabular}
\label{tab:baron_times}
\end{table}

\begin{table}
\centering
\caption{{Statistics on Gurobi solution times,
including the shifted geometric mean, median, minimum, and maximum times over the subset of 1000 instances for which Gurobi does not hit the 2 hour time limit.
The sixth and seventh columns show the geometric mean and median of the total number of nodes explored by Gurobi across the 1000 instances.
The last two columns denote the number of instances for which Gurobi hits the time limit and the corresponding geometric mean of residual optimality gaps at termination, respectively.}
}
\begin{tabular}{ r | c c c c | c c | c c }
\hline
Problem Family & \multicolumn{4}{c|}{Gurobi Solution Time (seconds)} & \multicolumn{2}{c|}{Gurobi B\&B Nodes} \\
 & Shifted GM & Median & Min & Max & GM & Median & \# TLE & TLE Gap (GM) \\ \hline
Bilinear, $n = 10$  &  0.03  &  0.03  &  0.01  &  0.13  &  1  &  1  &  0  & $\underline{\hspace{0.3cm}}$  \\
Bilinear, $n = 20$  &  0.7  &  0.7  &  0.3  &  1.3  &  2732  &  2687  &  0  & $\underline{\hspace{0.3cm}}$ \\
Bilinear, $n = 50$  &  7.2  &  7.1  &  2.6  &  18.3  &  17181  &  17263  &  0  & $\underline{\hspace{0.3cm}}$ \\[0.1in]
QCQP, $n = 10$  &  0.02  &  0.01  &  0.01  &  0.08  &  1  &  1  &  0  & $\underline{\hspace{0.3cm}}$  \\
QCQP, $n = 20$  &  0.6  &  0.6  &  0.3  &  1.1  &  2235  &  2389  &  0  & $\underline{\hspace{0.3cm}}$ \\
QCQP, $n = 50$  &  8.0  &  8.1  &  2.1  &  25.1  &  26108  &  27711  &  0  & $\underline{\hspace{0.3cm}}$ \\[0.1in]
Pooling  &  63.6  &  44.3  &  1.1  &  6367.1  &  833312  &  739876  &  10  &  $3 \times 10^{-4}$ \\  \hline
\end{tabular}
\label{tab:gurobi_times}
\end{table}

\subsubsection{Benchmarking using BARON}

Table~\ref{tab:baron_times} present statistics of BARON run times across the different QCQP families.
BARON solves the $10$-variable and $20$-variable random bilinear and QCQP instances within seconds.
However, it takes over $5$ minutes on average to solve the $50$-variable random bilinear instances and over $20$ minutes on average to solve the $50$-variable random QCQP instances. 
BARON also times out on $107$ out of $1000$ of the $50$-variable instances with univariate quadratic terms.
BARON finds the random pooling instances to be significantly harder, timing out on $417$ out of $1000$ instances and taking roughly $9$ minutes on average to solve the remaining $583$ out of $1000$ instances (while BARON finds global solutions, it is unable to prove global optimality within the time limit).
The last column in Table~\ref{tab:baron_times} indicates the GM of the relative optimality gap at termination for instances where BARON hits the time limit.
{Finally, the sixth and seventh columns note the GM and median of the total number of BARON iterations.
BARON requires significantly more iterations for the $50$-variable random bilinear and QCQP instances, as well as for the pooling instances, compared to the $10$-variable and $20$-variable instances.}

\subsubsection{{Benchmarking using Gurobi}}

{Table~\ref{tab:gurobi_times} present statistics of Gurobi run times across the different QCQP families.
Gurobi solves the $10$-variable and $20$-variable random bilinear and QCQP instances within a second and requires about $7$ seconds on average to solve the $50$-variable random bilinear and QCQP instances. 
However, Gurobi finds the random pooling instances to be relatively more challenging, timing out on $10$ out of $1000$ instances and taking about one minute on average to solve the remaining $990$ out of $1000$ instances.
The last column in Table~\ref{tab:gurobi_times} notes the GM of the relative optimality gap at termination for instances where Gurobi hits the time limit.
The sixth and seventh columns note the GM and median of the total number of B\&B nodes explored by Gurobi.
Despite requiring a significant number of nodes for convergence, particularly for the pooling instances, Gurobi solves our test instances much faster than BARON, potentially due to the use of cheaper relaxations, efficient parallelization and effective engineering.}

\subsection{Additional tables for the numerical experiments}
\label{subsec:omitted_tables}

Table~\ref{tab:ml_times} reports the total time taken by Algorithm~\ref{alg:ml_approx} to train the $dK \lvert \mathcal{NC} \rvert$ AdaBoost regression models for each QCQP family.
These training times are excluded from the Alpine+ML2 and Alpine+ML4 times reported in Section~\ref{subsec:results}, as they are typically incurred only once.
Once trained, these AdaBoost regression models can be used to predict partitioning points for use within Alpine's first iteration for all future test instances from the same QCQP family.

Table~\ref{tab:alpine_speedup} records the speedup or slowdown of Alpine with the strong partitioning and AdaBoost-based policies relative to default Alpine. 
Table~\ref{tab:maxmin_times} reports statistics on the time required to run Algorithm~\ref{alg:enhancements} to determine strong partitioning points for the different QCQP families.

\begin{table}[t]
\centering
\caption{{Total time taken by Algorithm~\ref{alg:ml_approx} to train all $dK \lvert \mathcal{NC} \rvert$ AdaBoost regression models. These times exclude the total time required to generate and sort the predictions of the AdaBoost regression models in Algorithm~\ref{alg:ml_approx}, which is typically less than $1\%$ (often much less) of the reported times.}}
\begin{tabular}{ r | c | c }
\hline
Problem Family & \hspace*{0.05in} $d$ \hspace*{0.05in} & Training Time (seconds) \\ \hline
Bilinear, $n = 10$ &  2  & 390  \\
Bilinear, $n = 20$ &  2  & 4063  \\
Bilinear, $n = 20$ &  4  & 5518  \\
Bilinear, $n = 50$ &  2  & 22825  \\[0.05in]
QCQP, $n = 10$ &  2  & 378  \\
QCQP, $n = 20$ &  2  & 3917  \\
QCQP, $n = 20$ &  4  & 6033  \\
QCQP, $n = 50$ &  2  & 14663  \\[0.05in]
Pooling &  2  & 91353  \\ \hline
\end{tabular}
\label{tab:ml_times}
\end{table}

\begin{table}[t]
\centering
\caption{
Statistics on the speedup or slowdown of Alpine with the strong partitioning and ML-based policies relative to default Alpine.
{The speedups for Alpine+SP2 and Alpine+SP4 do not account for the time spent in running Algorithm~\ref{alg:enhancements} to determine strong partitioning points.}
}
\begin{tabular}{ c | c | c c c c c c c c }
\hline
Problem Family & Solution Method & \multicolumn{8}{c}{Speedup/Slowdown Factor} \\
& & $<0.5$ & $0.5 - 1$ & $1 - 2$ & $2 - 5$ & $5 - 10$ & $10 - 20$ & $20 - 50$ & $> 50$ \\ \hline
Bilinear, $n = 10$ & \% Alpine+SP2 inst.\  &  $\underline{\hspace{0.3cm}}$  &  $\underline{\hspace{0.3cm}}$  &  1.1  &  57.6  &  40.1  &  1.2  &  0  &  0 \\
& \% Alpine+ML2 inst.\  &  0.2  &  2.1  &  7.7  &  49.9  &  40.0  &  0.1  &  0  &  0 \\[0.1in]
& \% Alpine+SP2 inst.\  &  0.2  &  3.3  &  7.2  &  18.2  &  31.2  &  29.9  &  10.0  &  0.0 \\
Bilinear, $n = 20$ & \% Alpine+ML2 inst.\  &  3.3  &  9.8  &  25.5  &  39.2  &  15.3  &  6.0  &  0.9  &  0.0 \\
& \% Alpine+SP4 inst.\  &  0.2  &  0.7  &  1.3  &  13.4  &  32.7  &  37.1  &  14.5  &  0.1 \\
& \% Alpine+ML4 inst.\  &  2.8  &  10.5  &  23.3  &  41.4  &  15.2  &  5.9  &  0.9  &  0.0 \\[0.1in]
Bilinear, $n = 50$ & \% Alpine+SP2 inst.\  &  0.4  &  1.3  &  7.2  &  18.7  &  26.3  &  24.3  &  14.9  &  6.9 \\
& \% Alpine+ML2 inst.\  &  0.7  &  4.7  &  16.9  &  32.5  &  25.3  &  13.7  &  5.4  &  0.8 \\[0.1in]
QCQP, $n = 10$ & \% Alpine+SP2 inst.\  &  $\underline{\hspace{0.3cm}}$  &  $\underline{\hspace{0.3cm}}$  &  0.1  &  3.3  &  76.1  &  20.4  &  0.1  &  0 \\
& \% Alpine+ML2 inst.\  &  1.0  &  3.9  &  20.9  &  8.5  &  53.4  &  12.3  &  0  &  0 \\[0.1in]
& \% Alpine+SP2 inst.\  &  0.1  &  3.2  &  12.2  &  18.4  &  11.5  &  19.4  &  32.6  &  2.6 \\
QCQP, $n = 20$ & \% Alpine+ML2 inst.\  &  0.5  &  5.1  &  19.0  &  40.7  &  23.1  &  9.6  &  1.9  &  0.1 \\
& \% Alpine+SP4 inst.\  &  0  &  0.2  &  1.3  &  3.8  &  5.5  &  28.2  &  53.7  &  7.3 \\
& \% Alpine+ML4 inst.\  &  0  &  2.9  &  11.6  &  33.3  &  27.0  &  17.2  &  7.6  &  0.4 \\[0.1in]
QCQP, $n = 50$ & \% Alpine+SP2 inst.\  &  0.9  &  1.3  &  10.7  &  22.0  &  23.0  &  32.5  &  7.2  &  2.4 \\
& \% Alpine+ML2 inst.\  &  1.4  &  4.0  &  19.5  &  32.4  &  22.7  &  16.6  &  3.4  &  0 \\[0.1in]
Pooling & \% Alpine+SP2 inst.\  &  2.2  &  6.4  &  19.7  &  26.0  &  21.8  &  16.8  &  6.7  &  0.4 \\
& \% Alpine+ML2 inst.\  &  2.1  &  11.5  &  34.5  &  40.4  &  9.8  &  1.4  &  0.3  &  0 \\ \hline
\end{tabular}%
\label{tab:alpine_speedup}
\end{table}

\begin{table}
\centering
\caption{
Statistics on the time required to run Algorithm~\ref{alg:enhancements} to determine strong partitioning points. 
Columns denote the shifted geometric mean, median, minimum, maximum, and standard deviation of these times.
}
\begin{tabular}{ r | c | c c c c c }
\hline
Problem Family & Solution & \multicolumn{5}{c}{Max-Min Solution Time (seconds)} \\
& Method & Shifted GM & Median & Min & Max & Std.\ Deviation \\ \hline
Bilinear, $n = 10$  &  SP2  &  16  &  14  &  6  &  96  &  13  \\[0.05in]
Bilinear, $n = 20$ &  SP2  &  528  &  445  &  136  &  2389  &  544  \\
&  SP4  &  1244  &  1117  &  374  &  4360  &  893  \\[0.05in]
Bilinear, $n = 50$  &  SP2  &  7070  &  7404  &  1271  &  23166  &  3268  \\[0.1in]
QCQP, $n = 10$  &  SP2  &  8  &  8  &  6  &  53  &  3  \\[0.05in]
QCQP, $n = 20$ &  SP2  &  1731  &  1826  &  171  &  4244  &  654  \\
&  SP4  &  2152  &  2740  &  471  &  5965  &  961  \\[0.05in]
QCQP, $n = 50$  &  SP2  &  16964  &  17074  &  8626  &  23551  &  2319  \\[0.05in]
Pooling  &  SP2  &  15658  &  15148  &  1088  &  77029  &  8657  \\ \hline
\end{tabular}
\label{tab:maxmin_times}
\end{table}

\end{APPENDICES}

\end{document}